%% file: paper.tex
\pdfoutput=1
\documentclass[11pt]{amsart}
\usepackage{amsmath, amssymb}
\usepackage[margin=1.15in]{geometry}
\usepackage{microtype}
\usepackage{booktabs}
\usepackage{array}
\usepackage{capt-of}
\usepackage{graphicx}
\usepackage{tikz}
\usetikzlibrary{arrows.meta, decorations.pathmorphing, fadings, calc, angles, quotes}
\usepackage[hidelinks]{hyperref}
\hypersetup{pdftitle={The two most symmetric flat tori as eight-vertex paper tori}, pdfauthor={Fabian Lander}}
\definecolor{figblue}{RGB}{40,90,160}
\definecolor{figred}{RGB}{170,60,50}
\tikzset{
  vlab/.style={font=\scriptsize, inner sep=1pt, fill=white, fill opacity=0.75, text opacity=1},
  flab/.style={font=\tiny, gray},
}

\newcommand{\R}{\mathbb{R}}
\newcommand{\Q}{\mathbb{Q}}
\newcommand{\Z}{\mathbb{Z}}
\newcommand{\C}{\mathbb{C}}
\newcommand{\HH}{\mathbb{H}}
\newcommand{\tauhat}{\hat\tau}
\newcommand{\Rep}{\operatorname{Re}}
\newcommand{\Imp}{\operatorname{Im}}
\newcommand{\file}[1]{\texttt{#1}}

\theoremstyle{plain}
\newtheorem{theorem}{Theorem}
\newtheorem{prop}{Proposition}
\newtheorem{lemma}{Lemma}
\newtheorem{cor}{Corollary}
\newtheorem*{theoremx}{Theorem}

\title[The most symmetric paper tori]{The two most symmetric flat tori as eight-vertex paper tori}
\author{Fabian Lander}
\address{Max Planck Institute for Mathematics in the Sciences, Inselstr.\ 22, 04103 Leipzig, Germany}
\email{fabian.lander@mis.mpg.de}
\thanks{The author is supported by the International Max Planck Research School Mathematics in the Sciences. He acknowledges the support of the Institut Henri Poincar\'e (UAR 839 CNRS-Sorbonne Universit\'e), and LabEx CARMIN (ANR-10-LABX-59-01).}

\begin{document}

\begin{abstract}
A paper torus is a polyhedral torus embedded in $\R^3$ whose intrinsic metric is locally Euclidean, a torus folded from a flat sheet of paper. Many constructions are known, from dozens of vertices into the thousands, and it is natural to ask how few suffice. Schwartz proved that no paper torus has seven vertices, and built one with eight. Doyle and Schwartz then built paper tori realizing almost every shape of flat torus with eight vertices, their construction missing exactly two rays in moduli space, all the rectangular tori, and the rhombic tori of aspect ratio at least $\sqrt3$. These rays end at the two most symmetric flat tori, the square torus and the hexagonal torus. We build an eight-vertex paper torus realizing each of them. We also prove that eight-vertex paper tori realize every shape near these two, since the flat configurations around both form a smooth seventeen dimensional manifold on which the modulus is a submersion.
\end{abstract}

\maketitle

\setcounter{tocdepth}{1}
\tableofcontents

\section{Introduction}\label{sec:intro}

Take a flat sheet of paper, and fold and glue it into a torus. The intrinsic geometry of the sheet never changes, and the result is a \emph{flat torus}, a torus whose intrinsic metric is locally Euclidean. It is realized in three-space as a polyhedral surface. Such \emph{paper tori} were first constructed by Burago and Zalgaller \cite{burago-zalgaller-1960}, who later proved that every isometry class of flat torus is realized, though only by constructions of enormous size \cite{burago-zalgaller} (see \cite{lazarus-tallerie}). Explicit constructions followed, from Brehm's diplotori \cite{brehm} to origami tori \cite{tsuboi} and the diplotorus family of Arnoux, Leli\`evre, and M\'alaga \cite{alm}. Quintanar embedded the square torus explicitly \cite{quintanar}, and Lazarus and Tallerie built a single universal triangulation, with $2434$ faces, admitting every flat torus \cite{lazarus-tallerie}. The question these constructions leave open is how few vertices a paper torus can have. It was a long-standing one, and appears in print in \cite{quintanar} and \cite{lazarus-tallerie}. Tugay\'e found one with $9$ \cite{tugaye}, and Schwartz then settled the question, proving that no paper torus has $7$ vertices and that $8$-vertex paper tori exist \cite{schwartz-bound}.

The natural next question is \emph{which} flat tori arise with eight vertices. The shape of a flat torus is recorded by a single point $\tauhat$ in the classical modular domain, its \emph{reduced modulus} (Section \ref{sec:flat}). Doyle and Schwartz proved \emph{near universality} for $8$-vertex paper tori \cite{doyle-schwartz}: almost every point in moduli space admits one. Their construction uses the valence-regular triangulation and enforces an order-$2$ rotational symmetry \cite{schwartz-bound}. Its moduli sweep out everything except two rays, the rectangular tori and the rhombic tori of aspect ratio at least $\sqrt3$, so in particular every flat torus \emph{without reflection symmetry} is realized by an $8$-vertex paper torus. Doyle and Schwartz ask whether these excluded shapes can be realized once that symmetry is dropped.

The two excluded rays end at the square torus and the hexagonal torus. We build an $8$-vertex paper torus of modulus exactly $\tauhat = i$, on the valence-regular triangulation, and one of modulus exactly $\tauhat = \rho = e^{i\pi/3}$, on a second triangulation. They are the two orbifold points of moduli space, the most symmetric flat tori there are.

\begin{figure}[t]
\centering
\includegraphics[width=\textwidth]{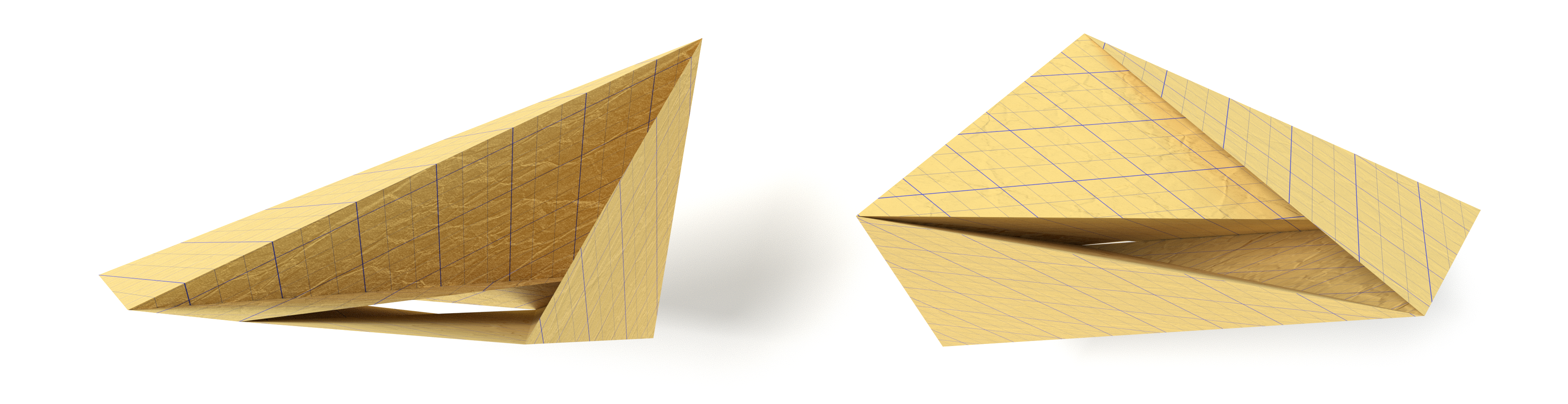}
\caption{An inflated square torus, left, and an inflated hexagonal torus, right, each with eight vertices. The square and hexagonal tori lie at the ends of the two rays that the constructions of \cite{doyle-schwartz} leave out. Both are exactly embedded, their cone angles and moduli correct to fifteen digits, and numerically each is joined to the tori of Theorem \ref{t:main} by a path of $8$-vertex paper tori of the same modulus (Appendix \ref{app:fat}).}\label{fig:fat-renders}
\end{figure}

\begin{theorem}\label{t:main}
The square torus and the hexagonal torus are each realized by an $8$-vertex paper torus.
\end{theorem}

Both achieve the minimum vertex count possible for their modulus, by Schwartz's overall bound. The proof gives more than the two tori.

\begin{theorem}\label{t:path}
Near each of these two tori the flat configurations form a smooth seventeen dimensional manifold, on which the modulus is a submersion.
\end{theorem}

So every shape near the square torus and near the hexagonal torus is realized by an $8$-vertex paper torus as well. As the square and hexagonal tori are the endpoints of the excluded rays in \cite{doyle-schwartz}, our construction provides vertex-minimal realizations of an initial segment of each.

\begin{cor}\label{c:rays}
There is an $\varepsilon > 0$ such that every rectangular torus $\tauhat = iY$ with $1 \le Y < 1 + \varepsilon$ and every rhombic torus $\tauhat = \tfrac12 + iH$ with $\tfrac{\sqrt3}{2} \le H < \tfrac{\sqrt3}{2} + \varepsilon$ is realized by an $8$-vertex paper torus.
\end{cor}

The square torus has been built before. Quintanar embedded it polyhedrally with $40$ vertices \cite{quintanar}, and an earlier $8$-vertex flat square torus of Gott and Vanderbei \cite{gott-vanderbei} is not embedded. The hexagonal torus has been built before as well, by Lazarus and Tallerie with $38$ vertices \cite{lazarus-tallerie}. This paper provides the first vertex-minimal embedded constructions.

The tori guaranteed by the proofs are nearly collapsed, arising from an arbitrarily small lift off a planar configuration. Theorem \ref{t:path} suggests more may be available, since near each of them the paper tori of the same modulus form a smooth manifold of dimension fifteen, which may extend well beyond the neighborhood given by the existence proof. Numerically moving along this manifold away from the planar configuration produces the two inflated tori of Figure \ref{fig:fat-renders}, recorded in Appendix \ref{app:fat} with printable nets, and Figure \ref{fig:inflation} shows the square torus at three points of that motion. How many components the paper tori of one modulus form, and whether every such torus can be driven into the plane this way, are among the open questions of Section \ref{sec:discussion}.

\begin{figure}[!htbp]
\centering
\includegraphics[width=\textwidth]{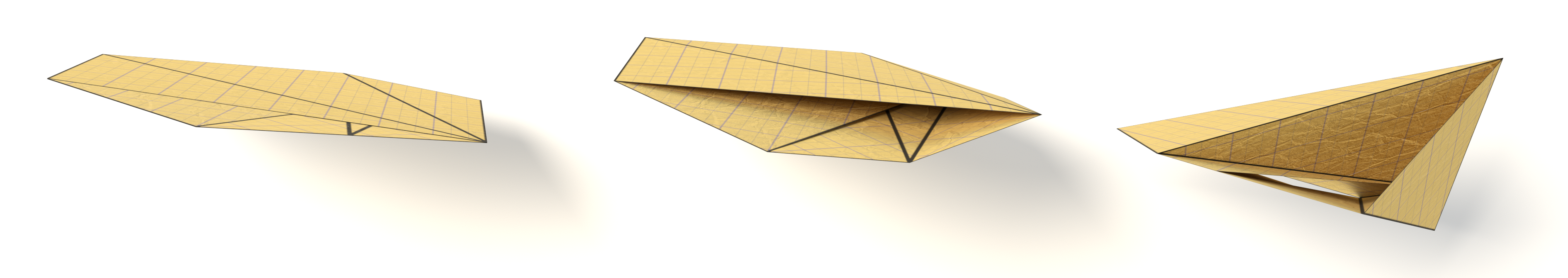}
\caption{The square torus at three points of the numerical path, all flat of modulus $i$. Left, the planar base $Q^0$ of Table \ref{tab:base}, exactly flat with modulus exactly $i$. Middle, a slightly inflated torus. Right, the inflated torus of Figure \ref{fig:fat-renders}. The two inflated ones are numerical, their cone angles and moduli correct to fifteen digits.}\label{fig:inflation}
\end{figure}

\subsection*{The strategy}
The proof of Theorem \ref{t:main} is a pure existence argument. Following Doyle and Schwartz, we do not look for an embedded torus directly but for a configuration collapsed into the plane, where the computations of Section \ref{sec:flat} land in the field of the vertex coordinates, number fields in our case. This allows us to verify flatness and the modulus in exact arithmetic. Such a configuration is not embedded, but its faces lie against one another rather than through one another, folded back and forth, so a small push separates them. In the geometry of configuration space, the subspace of collapsed configurations meets the boundary of the embedded set, and the proof starts on that boundary and walks into the embedded flat tori of the pinned modulus (Section \ref{ss:geometry}). The work is to make that push while keeping flatness and the modulus exact. An explicit lift breaks both, and the implicit function theorem fixes them. Theorem \ref{t:main} is proved this way, in three steps, and Theorem \ref{t:path} follows from the same computation.

\begin{itemize}
\item \textbf{The base.}\  To get everything started, we need to produce \emph{exact} planar configurations realizing the square and hexagonal tori. Section \ref{sec:flat} develops the tools to certify flatness and compute the modulus of a planar configuration exactly, so long as its vertices lie in a number field. This allows a rigorous version of guess-and-check. After some educated experimentation, we propose two planar configurations, the \emph{bases}, with exact coordinates listed in Table \ref{tab:base}, and use these tools to confirm that they are exactly flat square and hexagonal tori, respectively (Proposition \ref{p:base}).

\item \textbf{The lift.}\  We \emph{lift} each base to a one-parameter family of embedded configurations, by raising each vertex vertically out of the plane at its own constant speed. The condition that every pair of lifted faces separates is a system of strict linear inequalities in the eight speeds, so the admissible speeds form an open cone, and we choose a rational point of it, which allows an exact certification of embeddedness (Proposition \ref{p:separation}). Flatness is lost and the modulus drifts as the vertices rise, but the third step restores both.

\item \textbf{The correction.}\  We now move the vertices horizontally, each at its fixed height, to undo that drift (Section \ref{sec:correction}). Flatness and the modulus are nine equations in the sixteen horizontal coordinates. The base satisfies them exactly, and their derivative there, exact in the field, has rank nine (Proposition \ref{p:jacobian}). So the implicit function theorem provides horizontal displacements of the vertices, continuous along the lift and vanishing at the base, that adjust the lifted torus at each small time so that it is exactly flat with exactly the right modulus. One neighborhood of the base in the horizontal coordinates lifts to embedded configurations at every height at once, and the displacements are continuous, so the corrected tori stay embedded, and are the \emph{exact} paper tori we seek.

\end{itemize}
Concretely, for each of the two tori we produce a path $q_t$ of embedded paper tori for $0 < t < \varepsilon$, on the valence-regular triangulation $T$ in the square case and on the second triangulation $T'$ in the hexagonal one. The limit $q_0$ is the base, an explicit planar configuration in $\R^2 \times \{0\}$, its eight vertices the points of Table \ref{tab:base}, exact in $\Q$ for the square torus and in $\Q(\sqrt3)$ for the hexagonal one, and as $t$ varies the points move continuously. For each vertex $v$,
\[ q_t(v) = q_0(v) + \bigl(x_v(t),\ y_v(t),\ \zeta_v\, t\bigr), \]
where $x_v$ and $y_v$ are continuous functions vanishing at $0$ and the $\zeta_v$ are explicit values in the same field, the speeds of the lift, printed in Proposition \ref{p:separation}. The continuous functions are not explicit. They come from the implicit function theorem.

Doyle and Schwartz reach their tori the same way, pushing off from configurations collapsed into a plane and correcting by the implicit function theorem \cite{doyle-schwartz}. Besides the dropped symmetry, two things differ in our approach. Their correction restores flatness and they recover moduli by a surjectivity argument, which yields an open set, while we pin the modulus and reach a prescribed point. And their path is engineered so that the cone angles are stationary to second order, since their embedding margin decays, while ours is a straight vertical motion at constant speeds, which is what gives us one neighborhood for every height. So we never compare rates, and continuity of the correction is all we use.

\subsection*{Acknowledgements}
The results were found using data that came out of code written jointly with Steve Trettel in early June 2026, as a gift for Richard Schwartz's birthday conference. It grew out of many conversations with Steve, from which I learned a great deal, and I am grateful to him for his generosity and insight throughout. I thank Richard Schwartz for many conversations and insights, and for his encouragement to write up this result. I thank Marit Bobb and Clara Winne, who folded the tori of Appendix \ref{app:fat} from paper when I could not.

\section{Paper tori and configurations}\label{sec:flat}

A \emph{flat torus} develops onto $\C$, with holonomy acting by translations through a lattice $\Lambda$. A positively oriented basis $\Lambda = \Z v_1 \oplus \Z v_2$ determines a point $\tau = v_2/v_1 \in \HH$. Changing the basis acts on $\tau$ by a linear fractional transformation in $\mathrm{PSL}(2,\Z)$, so its orbit $[\tau] \in \HH/\mathrm{PSL}(2,\Z)$ records the modulus. Let $\tauhat$ denote its representative in the closed fundamental domain $\mathcal{D} = \{\, z \in \HH : |z| \ge 1,\ |\Rep z| \le \tfrac12 \,\}$ (Figure \ref{fig:fd}), taking $\Rep\tauhat \ge 0$ when the representative is not unique (see \cite[\S12.2]{farb-margalit}). With these conventions the square torus is represented by $i$ and the hexagonal torus by $\rho = e^{i\pi/3}$, the two orbifold points of $\mathcal{D}$, their lattices carrying rotational symmetries of order four and six. The vertical rays originating at them trace the two excluded rays, the rectangular tori and the rhombic tori of aspect ratio at least $\sqrt3$.

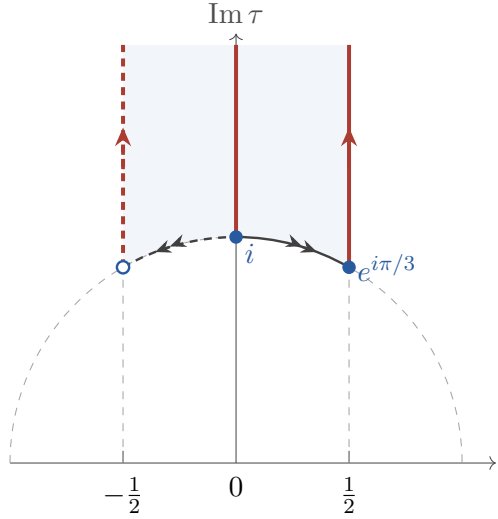
\begin{figure}[!htbp]
\centering
\begin{tikzpicture}[scale=1.15]
\input{figs/s5-fd.tex}
\end{tikzpicture}
\caption{A fundamental domain $\mathcal{D}$ of the modular group, the arrows being the identifications that fold it into the modular surface. The construction of \cite{doyle-schwartz} realizes every modulus off the two excluded rays. The open arc between $i$ and $\rho$ is among the moduli they realize. This paper realizes the two orbifold points $i$ and $\rho = e^{i\pi/3}$ at its corners. The dashed parts are what our convention $\Rep\tauhat \ge 0$ leaves out.}\label{fig:fd}
\end{figure}

\subsection{Geometry of configuration space}\label{ss:geometry}

Fix a triangulation $T$ of the torus with vertex set $V$, and a \emph{marking} on it, an ordered pair of edge loops whose classes form a basis of $H_1(T;\Z)$, as in Figure \ref{fig:type}. A \emph{paper torus} is the image of a piecewise affine isometric embedding $\phi \colon T \to \R^3$ of a flat torus, the faces of the triangulation going to the pieces on which $\phi$ is affine, the notion of \cite{schwartz-bound, doyle-schwartz}. Equivalently, it is an embedded polyhedral surface in $\R^3$ whose intrinsic metric is a flat torus, and we say that flat torus is \emph{realized} by it. An Euler characteristic computation forces $|V| \ge 7$. The triangulation with $|V| = 7$ is unique and carries no paper torus \cite{schwartz-bound}. For $|V| = 8$, the case of this paper, $T$ has $24$ edges and $16$ faces, and up to combinatorial isomorphism there are exactly seven such $T$ \cite{sulanke-lutz, lutz-page}. A \emph{configuration} is a map $q \colon V \to \R^3$. With $T$ fixed it extends uniquely to a piecewise linear map of $T$ into space, affine on each face, and we write $q$ for that map too. The configuration is \emph{flat} if the cone angle at every vertex, the sum of the incident triangle angles measured in $\R^3$, equals $2\pi$. Write $\delta_v$ for the \emph{deficit} at $v$, that is, $2\pi$ minus the cone angle there. Each $\delta_v$ is $2\pi$ minus a sum of triangle angles, so it is a smooth function on the configurations with nondegenerate faces. The sixteen triangles carry total angle $16\pi$, so
\[
\sum_{v \in V} \delta_v \;=\; 8 \cdot 2\pi - 16\pi \;=\; 0
\]
for every configuration with nondegenerate faces, and flatness is the vanishing of any seven of the eight deficits. The intrinsic metric then has no cone points and the surface is a flat torus with a well-defined modulus.

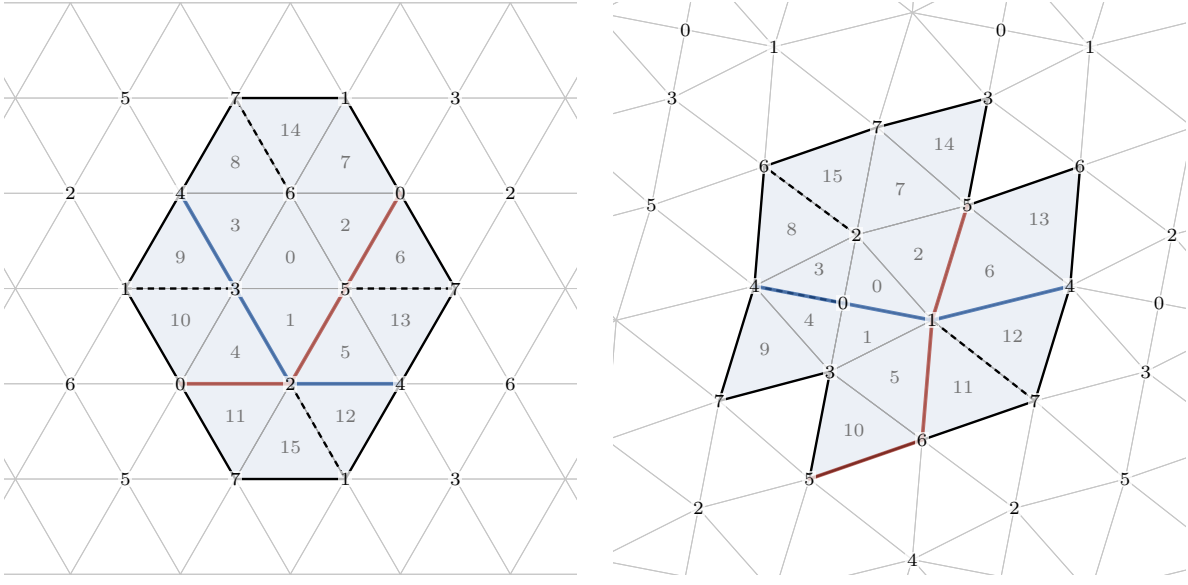
\begin{figure}[!htbp]
\centering
\begin{tikzpicture}[scale=1.26]
\input{figs/universal-square.tex}
\end{tikzpicture}%
\hspace{0.5cm}%
\begin{tikzpicture}[scale=1.26]
\input{figs/universal-hex.tex}
\end{tikzpicture}
\caption{The two combinatorial triangulations through their universal covers. On the left $T$, whose cover is the regular $\{3,6\}$ triangulation of the plane. On the right the second triangulation $T'$, of degrees $(4,6,6,6,6,6,7,7)$. In each panel the shaded region is the net of Section \ref{ss:nets} drawn combinatorially, its faces without their true shapes, $16$ triangles and $8$ vertices labeled $0$ through $7$, and Figure \ref{fig:nets} draws the same net with the true shapes. The blue and red paths are the two marking loops fixed in Section \ref{sec:bases}.}\label{fig:type}
\end{figure}

The configurations form the space
\[
\mathcal{C} \;=\; (\R^3)^V \;\cong\; \R^{24},
\]
and we single out four subsets of it. The set $\mathcal{C}^\circ$ of configurations with nondegenerate faces is open. A configuration is \emph{embedded} when its piecewise linear map of $T$ into space is injective. Inside $\mathcal{C}^\circ$, the set $\mathcal{E}$ of embedded configurations is open as well, since embeddedness is an open condition on a configuration. The flat configurations form the \emph{flat locus} $\mathcal{F} \subset \mathcal{C}^\circ$. And the configurations with all vertices in $\R^2 \times \{0\}$ form the sixteen dimensional subspace $\mathcal{P}$ of planar configurations, none of them embedded. The universal cover of the intrinsic metric of a point of $\mathcal{F}$ is the Euclidean plane, each class of $H_1(T;\Z)$ acting on it by a translation, and this identifies $H_1(T;\Z)$ with the lattice $\Lambda$ of a flat torus $\R^2/\Lambda$, up to rotation and scaling. The marking becomes a basis $(v_1, v_2)$ of $\Lambda$, which we take positively oriented, and we write
\[
\tau \colon \mathcal{F} \longrightarrow \HH, \qquad \tau \;=\; v_2/v_1,
\]
the ratio being unchanged by rotation and scaling. Section \ref{ss:developing} computes $\tau$ exactly, by developing. The similarities of $\R^3$ (translation, rotation, and scaling) act on $\mathcal{C}$, preserving $\mathcal{C}^\circ$, $\mathcal{E}$, $\mathcal{F}$ and $\tau$, and this group has dimension seven, so every dimension count below includes those.

A point of $\mathcal{E} \cap \mathcal{F}$ is a paper torus with eight vertices. Schwartz's theorem \cite{schwartz-bound} says that $\mathcal{E} \cap \mathcal{F}$ is nonempty on the valence-regular triangulation, and near universality \cite{doyle-schwartz} says that every reduced modulus off the two excluded rays, the rectangular ray $\tauhat = iY$ with $Y \ge 1$ and the rhombic ray $\tauhat = \tfrac12 + iH$ with $H \ge \tfrac{\sqrt3}{2}$, can be achieved. Theorem \ref{t:main} produces points of $\mathcal{E} \cap \mathcal{F}$ with $\tauhat = i$ and $\tauhat = \rho$, one on each of the two triangulations of Section \ref{sec:bases}, and Theorem \ref{t:path} describes $\mathcal{F}$ and the fibers of $\tau$ near them. In these terms the plan of Section \ref{sec:intro} reads as follows. The proof finds its starting point in $\mathcal{P} \cap \mathcal{F} \cap \partial\mathcal{E}$, a planar flat configuration of the right modulus on the boundary of the embedded set, and moves from it into $\mathcal{E} \cap \mathcal{F}$ along a fiber of $\tau$. Theorem \ref{t:main} constructs an initial segment of such a path, Theorem \ref{t:path} shows that near its start the fibers are smooth manifolds of dimension fifteen, and Appendix \ref{app:fat} suggests, numerically, that the path continues far from the plane.

\begin{figure}[t]
\centering
\begin{tikzpicture}[
  panel/.style={inner sep=0pt, outer sep=0pt},
  flow/.style={-{Stealth[length=2.2mm, width=1.7mm]}, black!55, line width=0.6pt},
  flab/.style={font=\small, text=black!70, inner sep=2.5pt},
]
\node[panel] (tri) at (0,0)
  {\begin{tikzpicture}[scale=0.88]\input{figs/universal-square-wide.tex}\end{tikzpicture}};
\node[panel, anchor=west] (cfg) at ($(tri.east)+(2.3cm,0)$)
  {\includegraphics[width=4.4cm]{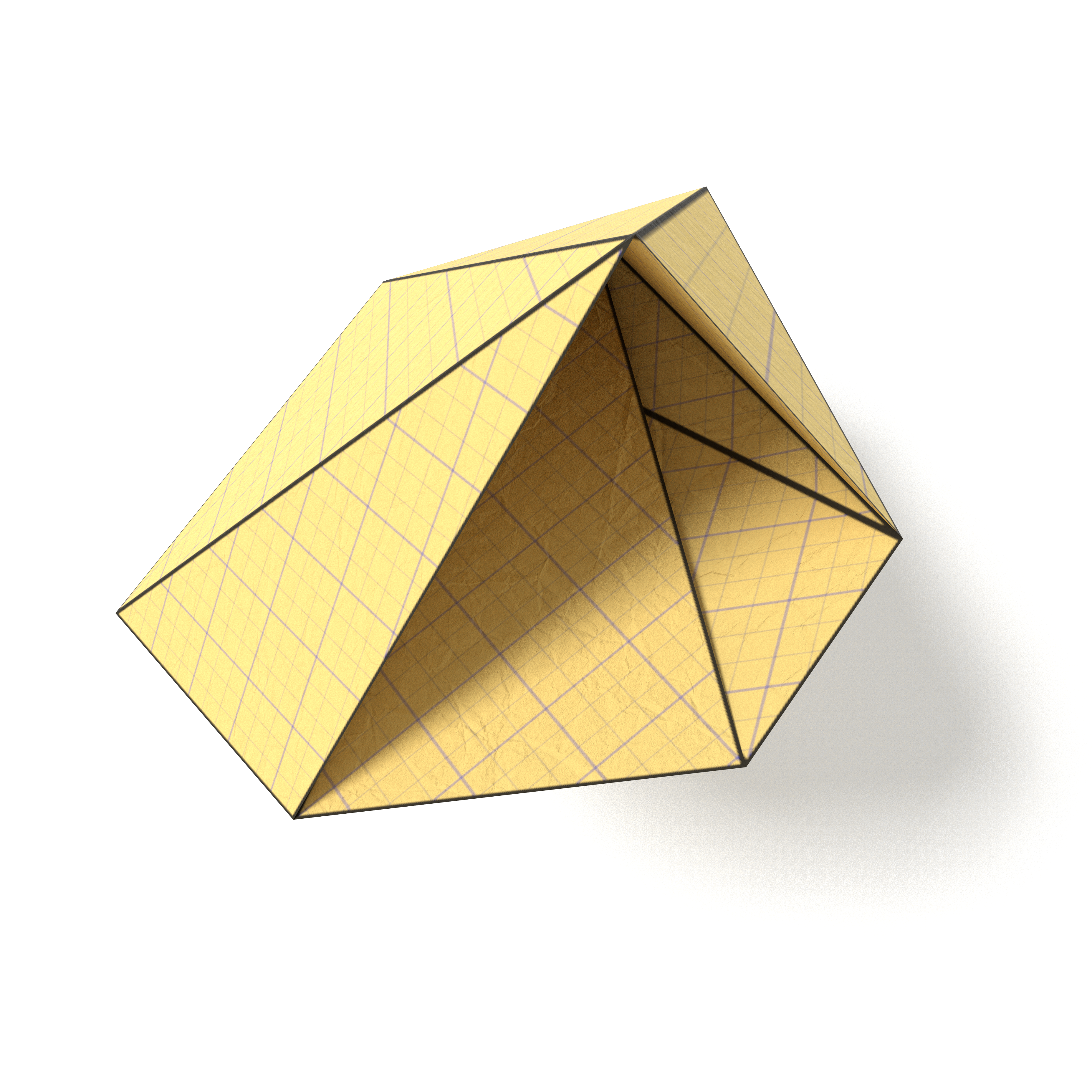}};
\node[panel, anchor=north] (net) at
  ($(tri.west)!0.5!(cfg.east) + (0,-3.55cm)$)
  {\begin{tikzpicture}[scale=1.45]\input{figs/net-pushed.tex}\end{tikzpicture}};
\draw[flow] (tri) -- (cfg) node[flab, midway, above] {configuration};
\draw[flow] (cfg.south) to[out=-90, in=60] node[flab, pos=0.55, above] {develop} (net.north);
\end{tikzpicture}
\caption{An overview of the three objects. Above, an $8$-vertex triangulation through its universal cover and a configuration of it in $\R^3$. Below, that configuration developed along the tree, its sixteen faces numbered $0$ to $15$. The configuration is embedded but not flat, and the four named corners show the notation $u_f(v)$ and where flatness fails: the developed edge vectors $u_{570}(7) - u_{570}(0)$ and $u_{207}(7) - u_{207}(0)$ have the same length but different directions. It is the square base of Section \ref{sec:bases}, lifted along $0.75\,\zeta$ and not yet corrected, see Sections \ref{sec:lift} and \ref{sec:correction}.}\label{fig:triple}
\end{figure}
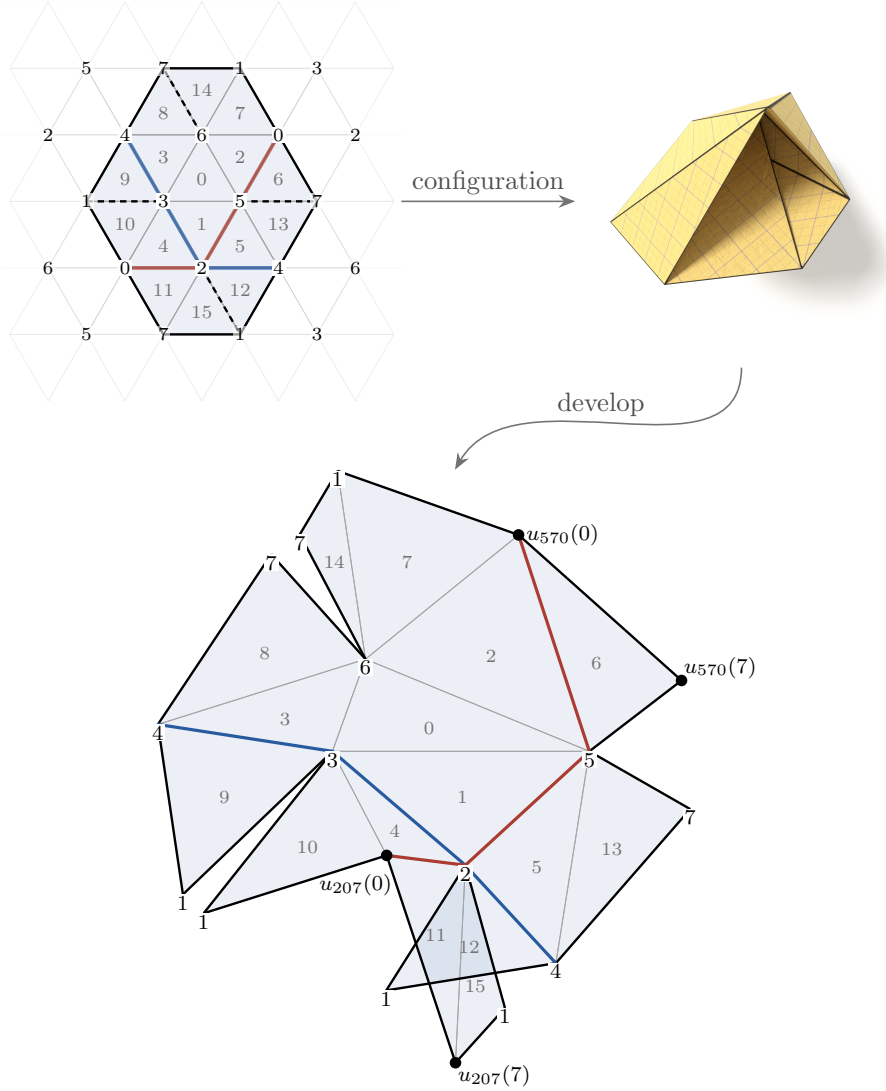

\subsection{Developing a configuration}\label{ss:developing}

We need to verify two exact statements about a configuration, that it is flat and that its modulus is the one we want. The angle sums at the vertices decide the first, but they are a transcendental mess, sums of arccosines to be compared with exactly $2\pi$. They are also local, one number per vertex, and say nothing about the modulus, which is global. Developing the configuration into the plane treats both at once. Flatness and the modulus become exact statements about the holonomy of the development, and no angle is ever computed. Figure \ref{fig:triple} illustrates a non-flat development and Figure \ref{fig:nets} shows two flat ones.

A \emph{development} lays the sixteen faces of a given configuration into the plane. Place the first face as a positively oriented triangle similar to the face, scaled so that its first edge runs from $(0,0)$ to $(1,0)$, then place each further face across an edge it shares with a face already placed, again positively oriented and at the same scale. Doing this for all sixteen faces gives a \emph{developed configuration}. The sixteen images need not be disjoint. For example in Figure \ref{fig:triple} the faces $11$ and $12$ overlap.

The placement order determines a spanning tree of the dual graph, whose sixteen nodes are the faces and whose edges are the $24$ edges of $T$, the tree edges being the ones the faces were placed across. The tree glues the fifteen edges it uses and leaves the other nine unglued, and those nine are what decide flatness. Fixing the tree and the starting position fixes every quantity computed below. The development assigns to each face $f$ and each vertex $v$ of $f$ a point $u_f(v)$ of the plane (Figure \ref{fig:triple}), and the \emph{developed edge vector} of an edge $(a,b)$ read in a face $f$ that contains it is $u_f(b) - u_f(a)$.

The nine non-glued edges are the boundary edges of the shaded net and the interior dashed edges in Figure \ref{fig:type}. To show that a configuration is flat, it suffices to show that for each such edge $\{p,q\}$, shared by two faces $f$ and $g$, the developed edge vectors agree, i.e. that
\[
u_f(p) - u_f(q) = u_g(p) - u_g(q),
\]
which makes each of the nine gluings a translation. The difference of the two sides is the \emph{gluing vector} of the edge, and the configuration is flat exactly when all nine gluing vectors vanish.

Here is why. Going once around a vertex composes the transitions between the consecutive incident faces, the identity across a tree edge and a gluing across a non-tree edge, so the composition is a product of gluings and their inverses, and it is the holonomy of the development around the vertex. Its linear part is the rotation by the cone angle there. If every gluing is a translation, every such composition is a translation. Hence each cone angle is a positive multiple of $2\pi$ and the angles of the sixteen triangles, distributed among the eight vertices, sum to $16\pi$, so every cone angle is exactly $2\pi$ and the configuration is flat.

At a flat configuration the developed edge vector does not depend on which of the two faces reads it, the two copies of the edge differing by a translation. The two \emph{periods} $v_1$, $v_2$ of a configuration are the sums of the developed edge vectors along the two marking loops. Different choices of tree, and of the face in which each marking edge is read, give different smooth extensions of $\tau$ off the flat locus, but on the flat locus they all agree. Each extension is smooth wherever the sixteen faces are nondegenerate and $v_1 \ne 0$.

\subsection{Explicit computations for a configuration}\label{ss:exact}

To decide flatness and read off the modulus we have to actually compute the development. Section \ref{ss:developing} gave the construction, and here we look at the computation itself. It turns out to be elementary. A Euclidean triangle is determined up to congruence by its side lengths, so the development depends on the configuration only through the $24$ squared edge lengths $|q_a - q_b|^2$. Areas involve a square root of those, so we carry the sixteen doubled face areas $S_f$ alongside them, and call these forty numbers the \emph{geometric data} of the configuration. Rational, below, always means rational in them. Everything the development produces, the developed points, the gluing vectors, the periods and the modulus, is a rational function of the geometric data, and so are the coefficients of all their differentials and of the differentials of the angle deficits, although the deficits themselves are transcendental. So wherever these numbers are known exactly, everything else can be computed exactly. Section \ref{ss:planar} shows that for a planar configuration they lie in the field of the vertex coordinates, which is what makes the computations of Section \ref{sec:bases} exact.

Fix a configuration $q$ with nondegenerate faces. For a face $f = ABC$ let $S_f = S = 2\,\mathrm{Area}(ABC)$. At its corner $A$ write $\alpha = |AB|^2$, $\beta = |AC|^2$, $\gamma = |BC|^2$, and let
\[
p \;=\; \langle AB,\, AC\rangle \;=\; \tfrac12(\alpha + \beta - \gamma)
\]
be the inner product of the two edges at the corner, linear in the squared lengths. Since $|u \times v|^2 + \langle u, v\rangle^2 = |u|^2 |v|^2$ for any two vectors, $ S^2 \;=\; \alpha\beta - p^2 .$ Figure \ref{fig:corner} collects the notation on one face.

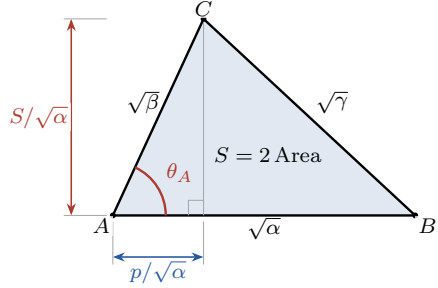
\begin{figure}[!htbp]
\centering
\begin{tikzpicture}
\input{figs/corner-data.tex}
\end{tikzpicture}
\caption{The data at the corner $A$ of a face $ABC$, the squared edge lengths $\alpha = |AB|^2$, $\beta = |AC|^2$ and $\gamma = |BC|^2$, the doubled area $S$, and the angle $\theta_A$. The perpendicular from $C$ meets $AB$ at distance $p/\sqrt\alpha$ from $A$ and has length $S/\sqrt\alpha$, which is $(p, S) = \sqrt{\alpha\beta}\,(\cos\theta_A,\ \sin\theta_A)$ read as a picture.}\label{fig:corner}
\end{figure}

\smallskip\noindent\textbf{The developed data, and its derivatives.}\ The whole development is built by repeating one step, a face with one edge already placed getting its third corner placed. So we write that step down, check that it is rational, and let induction do the rest. Take the face $ABC$ in its boundary order, and suppose the edge $AB$ has already been placed by the neighbour across it. Where the third corner goes is then decided by the squared lengths of its edges and its doubled area alone, namely
\[
C \;=\; A + \frac{p}{\alpha}\,(B - A) \;+\; \frac{S}{\alpha}\,(B - A)^{\perp},
\]
with $\alpha$, $p$ and $S$ read at the corner $A$ of $ABC$, and $\perp$ the rotation by a right angle. Indeed the edge from $A$ to $C$ has component $p/\sqrt\alpha$ along $AB$ and $S/\sqrt\alpha$ perpendicular to it, and $B - A$ has length $\sqrt\alpha$, which gives the two coefficients. Both are rational and dimensionless, so the formula also holds in the development, where every face appears at one common scale. The root face is placed with an edge from $(0,0)$ to $(1,0)$ and its third corner at $(p/\alpha,\ S/\alpha)$, and induction along the tree gives that every developed point is rational in $\alpha$, $p$ and $S$. The gluing vectors, the periods $v_1$, $v_2$ and the modulus $\tau = v_2 \bar v_1/|v_1|^2$ follow immediately.

A rational function of the squared edge lengths and the doubled face areas differentiates, by the quotient rule, into the same quantities together with the differentials of the areas, and these are rational expressions too. Indeed $S^2 = \alpha\beta - p^2$ is a polynomial in the squared edge lengths, so
\[
dS \;=\; \frac{d(S^2)}{2S} \;=\; \frac{d(\alpha\beta - p^2)}{2S} \;=\; \frac{\alpha\, d\beta + \beta\, d\alpha - 2p\, dp}{2S}.
\]
So the differentials of the developed points, the gluing vectors, the periods and the modulus are rational in the squared edge lengths and the doubled face areas, as forms in the differentials of the squared edge lengths.

\smallskip\noindent\textbf{The angles, and their derivatives.}\ The angle sum at a vertex is a sum of angles in the incident triangles, so it is enough to treat one triangle. At the corner $A$ of a face, $(p, S) = \sqrt{\alpha\beta}\,(\cos\theta_A,\ \sin\theta_A)$, so
\[
\theta_A \;=\; \arg(p + iS),
\]
and the corner angles, and with them the deficits $\delta_v = 2\pi - \sum \theta$, are transcendental, but their derivatives are rational expressions. Since $S > 0$ the point $p + iS$ stays in the upper half plane, where $\arg$ is smooth with $d\arg z = \Imp(dz/z)$, so
\[
d\theta_A \;=\; \Imp \frac{d(p + iS)}{p + iS} \;=\; \Imp \frac{(dp + i\,dS)(p - iS)}{p^2 + S^2} \;=\; \frac{p\,dS - S\,dp}{p^2 + S^2},
\]
and with $dS$ as above the coefficients on the right are rational in the same data, so $d\theta_A$, and with it each $d\delta_v = -\sum d\theta$, has rational coefficients as well. Unlike the developed data this did not come for free and took a computation, $\theta_A$ itself being transcendental. What makes it work is that the differential of an inverse trigonometric function is algebraic.

We record this for later use.

\begin{prop}\label{p:rational}
Let $q$ be a configuration with nondegenerate faces. The developed points, the gluing vectors, the periods $v_1$, $v_2$ and the modulus $\tau$ are rational functions of the geometric data of $q$, and so are the coefficients of their differentials, and of the differentials of the angle deficits $\delta_v$, taken with respect to the squared edge lengths. The deficits themselves are functions of the same data, and taken with respect to the vertex coordinates, all these differentials have coefficients rational in that data and the coordinates.
\end{prop}

The last claim is the chain rule. Each squared edge length is the polynomial $|q_a - q_b|^2$ in the vertex coordinates, so passing to them is the substitution $d|q_a - q_b|^2 = 2\langle q_a - q_b,\, dq_a - dq_b\rangle$, whose coefficients are linear in the coordinates. Section \ref{sec:correction} differentiates this way.

\subsection{Planar configurations}\label{ss:planar}

Our construction in Section \ref{sec:bases} begins with a configuration collapsed into the plane, and here we specialize the computations of Section \ref{ss:exact} to that case.

A configuration $Q$ is \emph{planar} if its vertices lie in $\R^2 \times \{0\}$, which we identify with $\R^2$, so that it is given by a map $V \to \R^2$. Fix the orientation of the torus, and write each face $f = (a,b,c)$ in its oriented boundary order. For a nondegenerate face the \emph{orientation sign} is
\[
\varepsilon_f \;=\; \operatorname{sign}\det(Q_b - Q_a,\ Q_c - Q_a) \;\in\; \{+1, -1\},
\]
so $\varepsilon_f = +1$ when the planar image of the face runs the same way round as its boundary cycle, and $\varepsilon_f = -1$ when it is folded over.

By Proposition \ref{p:rational} every quantity we need, and every coefficient of every differential we need, is a rational function of three kinds of data, the vertex coordinates, the squared edge lengths and the doubled face areas. The first two lie in the field of the vertex coordinates for any configuration, the squared edge lengths being polynomials in the coordinates. For a planar configuration, so do the areas. Everything then takes place in the field the vertex coordinates generate, which allows exact arithmetic when they lie in a number field.

\begin{lemma}[planar exactness]\label{l:planar}
Let $Q$ be a planar configuration with nondegenerate faces and vertex coordinates in a field $K \subset \R$. Then its doubled face areas lie in $K$, and consequently so does every quantity of Proposition \ref{p:rational}, values and coefficients of differentials alike.
\end{lemma}

\begin{proof}
For a face $f = (a,b,c)$ set $u = Q_b - Q_a$ and $v = Q_c - Q_a$. Since the face is planar,
\[
S_f^2 \;=\; |u|^2 |v|^2 - \langle u, v\rangle^2 \;=\; \det(u,v)^2 ,
\]
so
\[
S_f \;=\; |\det(u,v)| \;=\; \varepsilon_f \det(u,v) \;\in\; K .
\]
The rest is Proposition \ref{p:rational}. Everything it lists is rational in the coordinates, the squared edge lengths and the $S_f$, all of which now lie in $K$, the heights of $Q$ being zero.
\end{proof}

This exactness is used at three points of the argument. In Section \ref{sec:bases} we compute the developed vertices, the nine gluing vectors and the two periods exactly, and confirm that the two planar configurations are flat with the moduli we claim. In Section \ref{sec:lift} we check a system of inequalities exactly, and conclude that a lift is embedded. In Section \ref{sec:correction} we compute the $144$ entries of a Jacobian exactly, and conclude that the differential of flatness and modulus is nonsingular, which is what lets the implicit function theorem run. Each is a finite computation by addition, multiplication and division in the field of the vertex coordinates.

The deficits themselves are transcendental and are never computed. Flatness is checked through the development instead, where the nine gluing vectors vanish exactly when the configuration is flat (Section \ref{ss:developing}). The Jacobian of Section \ref{sec:correction} does differentiate the deficits rather than the gluing vectors, out of convenience. Flatness is the vanishing of seven of the eight deficits, so the deficits give seven independent equations directly, while seven independent ones would first have to be selected among the eighteen redundant gluing equations.

The lemma restricts which flat tori a planar configuration can realize at all.

\begin{cor}\label{c:field}
Let $Q$ be a planar flat configuration with vertex coordinates in a field $K \subset \R$, and let $\tau$ be its modulus. Then $\Q(\Rep\tau,\, \Imp\tau) \subseteq K$, and this field does not depend on the marking. In particular the square torus imposes no restriction on $K$, while a planar configuration realizing the hexagonal torus has $\sqrt3 \in K$.
\end{cor}

\begin{proof}
By Lemma \ref{l:planar} both coordinates of $\tau$ lie in $K$. Changing the marking replaces $\tau$ by $(a\tau + b)/(c\tau + d)$ with integer coefficients and $ad - bc = 1$, and writing $\tau = x + iy$,
\[
\Rep\frac{a\tau + b}{c\tau + d} = \frac{ac(x^2 + y^2) + (ad + bc)x + bd}{(cx + d)^2 + c^2y^2}, \qquad
\Imp\frac{a\tau + b}{c\tau + d} = \frac{y}{(cx + d)^2 + c^2y^2},
\]
so both lie in $\Q(x, y)$. The inverse change of marking gives the reverse inclusion, so the field is unchanged. It is $\Q$ at $\tau = i$, and $\Q(\tfrac12, \tfrac{\sqrt3}{2}) = \Q(\sqrt3)$ at $\tau = \rho$.
\end{proof}

\section{The base}\label{sec:bases}

This section is the first of the three steps of Section \ref{sec:intro}. We seek a starting point in $\mathcal{P} \cap \mathcal{F} \cap \partial\mathcal{E}$ for each torus: a configuration folded entirely into the plane, exactly flat and with exactly the desired modulus. The two bases below were discovered numerically before they were found exactly. We first searched directly for embedded $8$-vertex configurations of the square and hexagonal moduli, flat to machine precision. Convinced by these numerics that such tori existed, but left with only floating point objects, we looked for a route to an exact construction. Following the strategy of Doyle and Schwartz, we tried to collapse the configurations into the plane, where Section \ref{ss:planar} makes exact arithmetic available.

We numerically flowed the configurations toward the plane while keeping flatness and the modulus pinned and preventing the faces from crossing. As the configurations collapsed onto the plane, simple relations emerged among their vertices, with collections lying non-generically on lines, rectangles and circles, the same relations for every limit sharing a pattern of orientation signs (Section \ref{ss:planar}). These suggested exact geometric constraints that we could impose in place of the numerical conditions. Solving the resulting equations produced two four-parameter rational families of folded flat tori, one sweeping the rectangular ray $\tauhat = iY$, $Y \ge 1$, and one the rhombic ray $\tauhat = \tfrac12 + iH$, $H \ge \tfrac{\sqrt3}{2}$, the subject of a sequel, and the two planar configurations below are their corner members, with coordinates in $\Q$ in the square case and in $\Q(\sqrt3)$ in the hexagonal case. These are also the smallest fields one could hope for. By Corollary \ref{c:field}, a planar square base may have coordinates in $\Q$, while every planar configuration realizing the hexagonal torus must have $\sqrt3$ in its coordinate field.

From this point on, the numerical search plays no role. The proof needs only the coordinates of Table \ref{tab:base} and the two directions of Section \ref{sec:lift}, nothing else. We simply take the two configurations as given and verify in exact arithmetic that their faces are nondegenerate, that they are flat, and that their marked moduli are exactly $i$ and $\rho$, respectively.

Throughout, a triangulation is presented by its face list, with vertex set $V$, edge set $E$, and face set $F$, faces written as positively ordered triples. One can see in Figure \ref{fig:type} that the two printed lists triangulate the torus, each as a net of sixteen triangles with its boundary glued in pairs, and Figure \ref{fig:nets} shows the same with the true shapes of the faces from our exact planar configurations.

\subsection{The square base}

The triangulation $T$ has vertex set $V = \{0, \dots, 7\}$ and the sixteen faces
\[
\begin{aligned}
&356,\ 325,\ 364,\ 302,\ 341,\ 310,\ 506,\ 524,\\
&547,\ 570,\ 674,\ 601,\ 617,\ 214,\ 207,\ 271,
\end{aligned}
\]
written as positive boundary cycles, in the fixed order used by every table and script, its $24$ edges being the pairs of vertices that occur in them. We mark it by the loops $4 \to 2 \to 3 \to 4$ and $5 \to 2 \to 0 \to 5$, in that order. They meet in the single vertex $2$, as Figure \ref{fig:type} shows, so their classes form a basis of $H_1(T;\Z)$. Among the seven combinatorial types of $8$-vertex torus triangulation \cite{sulanke-lutz, lutz-page} it is the one in which every vertex has valence $6$, the type of Schwartz's example \cite{schwartz-bound}.

The square base is the planar configuration $Q^0 \colon V \to \Q^2$ of Table \ref{tab:base}, read as a configuration in $\R^3$ with all $z$ coordinates zero, and drawn on the left in Figure \ref{fig:base}. Its sixteen faces are nondegenerate. Their orientation signs (Section \ref{ss:planar}), in the face order above, are
\[
\texttt{+ + - + + - - + - + - + + - - -}
\]
so eight of them are folded over.

\subsection{The hexagonal base}

The triangulation $T'$ has vertex set $\{0, \dots, 7\}$, the sixteen faces
\[
\begin{aligned}
&012,\ 031,\ 024,\ 043,\ 152,\ 136,\ 145,\ 174,\\
&167,\ 264,\ 257,\ 276,\ 347,\ 356,\ 375,\ 465,
\end{aligned}
\]
written as positive boundary cycles. We mark it by the loops $1 \to 4 \to 0 \to 1$ and $1 \to 5 \to 6 \to 1$, in that order, meeting in the single vertex $1$, so their classes form a basis of $H_1(T';\Z)$. The vertex degrees are $(4, 6, 6, 6, 6, 6, 7, 7)$, the degree-$7$ vertices being $1$ and $4$ and the degree-$4$ vertex $0$. Its vertex cycles are now of lengths $(4, 7, 6, 6, 7, 6, 6, 6)$ in vertex order.

The hexagonal base is the planar configuration $P^0$ of Table \ref{tab:base}, again with heights zero, and drawn on the right of Figure \ref{fig:base}. Every number in it lies in $\Q(\sqrt3)$, written $a + b\sqrt3$ with $a, b \in \Q$, and every claim we make about it is decided by exact arithmetic in this field. Its sixteen faces are nondegenerate. Their orientation signs, in the face order above, are
\[
\texttt{- + + - - - - - + - + + + + - +}
\]
so here too eight are folded over.

The vertices of both bases satisfy the exact relations that the collapse suggested, drawn on them in Figure \ref{fig:base}. On the square base the vertices $0$, $2$, $4$ are collinear, the straight crease, the vertices $0$, $4$, $7$, $6$ form a rectangle, and $5$ lies on its circumscribed circle. The hexagonal base carries concyclicities instead, the five vertices $0$, $2$, $4$, $5$, $6$ on one circle and the quadruples $0$, $1$, $3$, $6$ and $3$, $4$, $5$, $7$ on two more. They play no role in the proof.

\begin{table}[!htbp]
\centering
\small
\setlength{\tabcolsep}{5pt}
\begin{tabular}{@{}c l l@{}}
\toprule
$v$ & $Q^0_v$ & $P^0_v$\\
\midrule
$0$ & $(0,\ 0)$ & $(0,\ 0)$\\
$1$ & $(\tfrac{73}{60},\ \tfrac{11}{10})$ & $\tfrac{1}{8073677900}\,(896842303 + 729553300\sqrt3,\ \ 1029317275 + 568256878\sqrt3)$\\
$2$ & $(\tfrac12,\ 0)$ & $\tfrac{1}{47100}\,(-2625 + 7310\sqrt3,\ \ 8925 - 24854\sqrt3)$\\
$3$ & $(-\tfrac{15977}{50320},\ \tfrac{42471}{50320})$ & $\tfrac{1}{10525}\,(3789,\ \ 595 - 4713\sqrt3)$\\
$4$ & $(1,\ 0)$ & $(1,\ 0)$\\
$5$ & $(\tfrac{35937}{25160},\ \tfrac{27357}{25160})$ & $\tfrac{1}{200}\,(64 + 25\sqrt3,\ \ 75 - 136\sqrt3)$\\
$6$ & $(0,\ \tfrac{33}{20})$ & $\tfrac{1}{200}\,(-36 + 25\sqrt3,\ \ 75 - 36\sqrt3)$\\
$7$ & $(1,\ \tfrac{33}{20})$ & $\tfrac{1}{109900}\,(279568 - 117825\sqrt3,\ \ 264900 - 143557\sqrt3)$\\
\bottomrule
\end{tabular}
\caption{The two folded bases, the eight planar vertex positions of each, $Q^0$ exact in $\Q$ and $P^0$ exact in $\Q(\sqrt3)$, the hexagonal entries each scaled by one common denominator.}\label{tab:base}
\end{table}

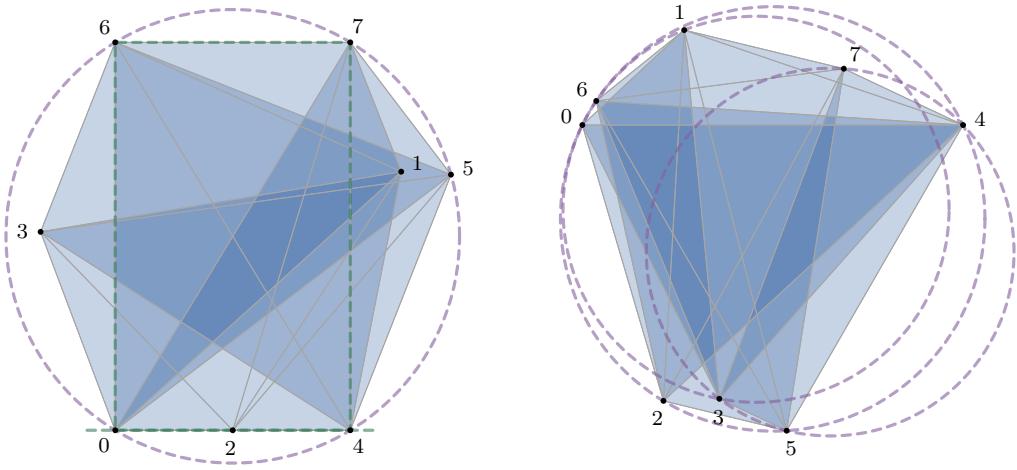
\begin{figure}[!htbp]
\centering
\begin{tikzpicture}
\input{figs/fold-relations-square-allblue.tex}
\end{tikzpicture}
\qquad
\begin{tikzpicture}
\input{figs/fold-relations-hex-allblue.tex}
\end{tikzpicture}
\caption{The two folded bases, the square base $Q^0$ on the left and the hexagonal base $P^0$ on the right, with the relations of the text drawn on them. On the left the line through the vertices $0$, $2$, $4$ and the rectangle $0$, $4$, $7$, $6$ are green and the circumscribed circle of that rectangle is purple, with the vertex $5$ on it. On the right the three circles are purple, the first through $0$, $2$, $4$, $5$, $6$ and the other two through $0$, $1$, $3$, $6$ and through $3$, $4$, $5$, $7$.}\label{fig:base}
\end{figure}

\subsection{Unfolding the planar configurations}\label{ss:nets}

\begin{prop}\label{p:base}
The square base $Q^0$ and the hexagonal base $P^0$ are flat, with periods satisfying $v_2 = i\,v_1$ on $Q^0$ and $v_2 = \rho\,v_1$ on $P^0$, so the intrinsic metric of each is that of the square, respectively the hexagonal, torus up to scaling.
\end{prop}

The proof is a finite computation on the printed data, checking three things, that the nine gluings are translations, that $v_1 \ne 0$, and that the periods satisfy $v_2 = i\,v_1$, respectively $v_2 = \rho\,v_1$. Both bases have nondegenerate faces, with the orientation signs printed above, so the development applies, and every quantity it produces lies in the field of the base by Lemma \ref{l:planar}, so every comparison is exact (Appendix \ref{app:verify}). The rest of this subsection is that computation.

We unfold both configurations along the same tree. The first face $(a,b,c)$ of the printed list goes into the plane first, normalized by $u(a) = (0,0)$ and $u(b) = (1,0)$, and after each face its three neighbors follow, in the order the shared edges appear in its boundary cycle, skipping ones that are already placed. This is the breadth-first order on the dual graph rooted at the first face. The face numbers of Figures \ref{fig:type} and \ref{fig:nets} are the positions in this order, so the order can be read off either figure. Each placement is simple for a planar configuration. The new face already lies in the plane, and it goes across the shared edge by nothing at all when its orientation sign agrees with its neighbor's, and by the reflection in the line of that edge when they differ.

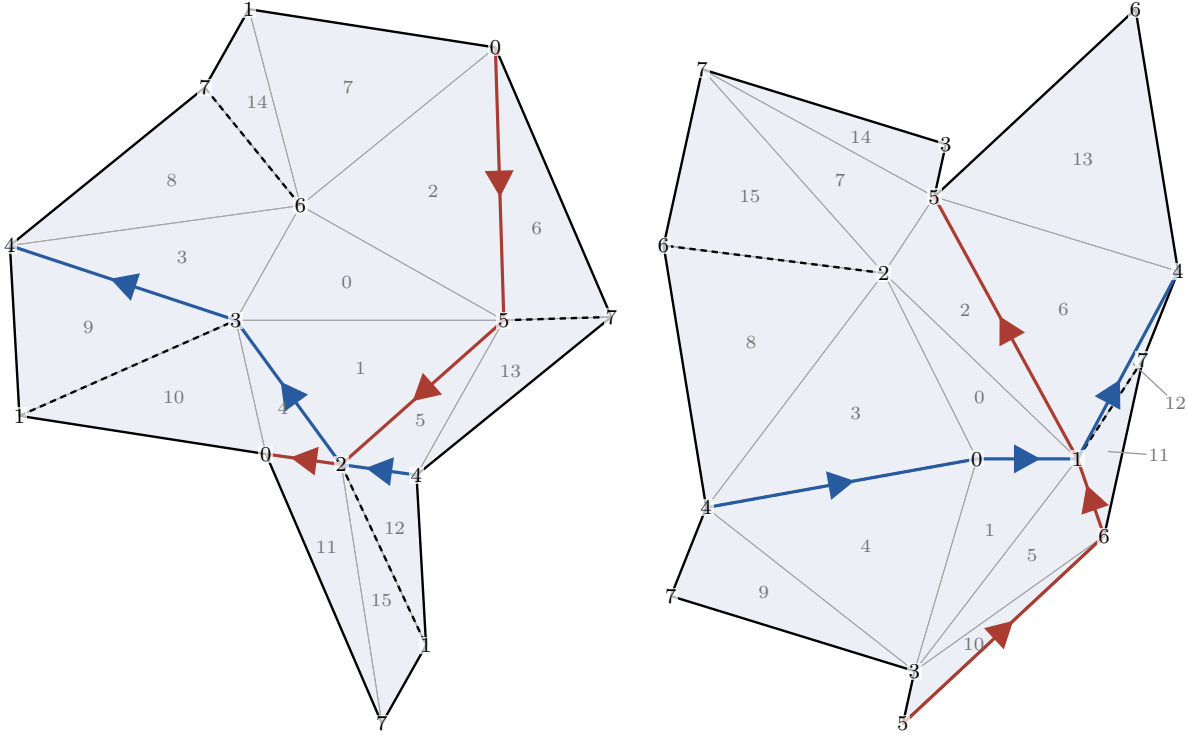
\begin{figure}[!htbp]
\centering
\begin{tikzpicture}[scale=1.75]
\input{figs/net-square.tex}
\end{tikzpicture}%
\hspace{0.5cm}%
\begin{tikzpicture}[scale=1.75]
\input{figs/net-hex.tex}
\end{tikzpicture}
\caption{The two unfolded planar configurations, sixteen faces unfolded along the dual tree, the square base left and the hexagonal base right. Thin interior segments are the fifteen tree edges, the heavy outline is the free boundary, and dashed segments are non-tree edges along which the net closes back up. The blue and red arrows are the developed edge vectors of the two marking loops, each triple summing to its period.}\label{fig:nets}
\end{figure}

The forty-eight placed corners fall on fourteen points of the plane on the square side and on fifteen on the hexagonal one, the corners of the two nets, and Tables \ref{tab:netsq} and \ref{tab:nethex} give them exactly, each point named $u_f(v)$ by the vertex $v$ it is a copy of and the face $f$ that places it first.

\begin{table}[!htbp]
\centering
\small
\setlength{\tabcolsep}{6pt}
\begin{tabular}{@{}l l@{}}
\toprule
point & position\\
\midrule
$u_{356}(3)$ & $(0,\ 0)$\\
$u_{356}(5)$ & $(1,\ 0)$\\
$u_{356}(6)$ & $\tfrac{1}{6254125}\,(1510441,\ 2676762)$\\
$u_{325}(2)$ & $\tfrac{1}{6254125}\,(2459423,\ -3366264)$\\
$u_{506}(0)$ & $\tfrac{1}{6254125}\,(6055927,\ 6365964)$\\
$u_{364}(4)$ & $\tfrac{1}{250165}\,(-210837,\ 69696)$\\
$u_{302}(0)$ & $\tfrac{1}{6254125}\,(702403,\ -3121404)$\\
$u_{524}(4)$ & $\tfrac{1}{6254125}\,(4216443,\ -3611124)$\\
$u_{570}(7)$ & $\tfrac{1}{6254125}\,(8761929,\ 78078)$\\
$u_{601}(1)$ & $\tfrac{1}{18762375}\,(915847,\ 21774654)$\\
$u_{674}(7)$ & $\tfrac{1}{6254125}\,(-725439,\ 5431602)$\\
$u_{341}(1)$ & $\tfrac{1}{750495}\,(-605789,\ -267498)$\\
$u_{207}(7)$ & $\tfrac{1}{1250825}\,(681681,\ -1881858)$\\
$u_{214}(1)$ & $\tfrac{1}{18762375}\,(13317379,\ -22748022)$\\
\bottomrule
\end{tabular}
\caption{The square net of Figure \ref{fig:nets}. Its sixteen faces put their forty-eight corners on fourteen points of the plane, each named here by the face that places it first and listed in the order the unfolding places them, exact in $\Q$ and with the root edge of length one.}\label{tab:netsq}
\end{table}

\begin{table}[!htbp]
\centering
\footnotesize
\setlength{\tabcolsep}{6pt}
\begin{tabular}{@{}l c l@{}}
\toprule
point & & position\\
\midrule
$u_{012}(0)$ & $x$ & $0$\\
 & $y$ & $0$\\
$u_{012}(1)$ & $x$ & $1$\\
 & $y$ & $0$\\
$u_{012}(2)$ & $x$ & $\tfrac{1}{555694913304951}\,(2515566967447176 - 1748881409442255\sqrt3)$\\
 & $y$ & $\tfrac{1}{555694913304951}\,(6921139243529985 - 3403655728740217\sqrt3)$\\
$u_{031}(3)$ & $x$ & $\tfrac{1}{2118893}\,(-6578800 + 3039430\sqrt3)$\\
 & $y$ & $\tfrac{1}{2118893}\,(-22367920 + 10334062\sqrt3)$\\
$u_{152}(5)$ & $x$ & $\tfrac{1}{740926551073268}\,(544497372063943 - 495475141869625\sqrt3)$\\
 & $y$ & $\tfrac{1}{740926551073268}\,(13541673711173875 - 6703740324583475\sqrt3)$\\
$u_{024}(4)$ & $x$ & $\tfrac{1}{185231637768317}\,(1305770380419850 - 1041620420752000\sqrt3)$\\
 & $y$ & $\tfrac{1}{185231637768317}\,(-1339666698144250 + 721657938141100\sqrt3)$\\
$u_{136}(6)$ & $x$ & $\tfrac{1}{740926551073268}\,(-5077052117214257 + 3472341516107375\sqrt3)$\\
 & $y$ & $\tfrac{1}{740926551073268}\,(-2147368694014125 + 911781885620925\sqrt3)$\\
$u_{145}(4)$ & $x$ & $\tfrac{1}{185231637768317}\,(-1823142417564850 + 1266463823447500\sqrt3)$\\
 & $y$ & $\tfrac{1}{185231637768317}\,(419398228148750 - 42809631023600\sqrt3)$\\
$u_{257}(7)$ & $x$ & $\tfrac{1}{1296621464378219}\,(-33749292045964731 + 17441837192871625\sqrt3)$\\
 & $y$ & $\tfrac{1}{1296621464378219}\,(92012050691116625 - 50224279449101350\sqrt3)$\\
$u_{264}(6)$ & $x$ & $\tfrac{1}{740926551073268}\,(5767578893743343 - 4661956824877625\sqrt3)$\\
 & $y$ & $\tfrac{1}{740926551073268}\,(8183006918596875 - 3817108572019075\sqrt3)$\\
$u_{347}(7)$ & $x$ & $\tfrac{1}{1296621464378219}\,(-30825006729247631 + 15520269580198875\sqrt3)$\\
 & $y$ & $\tfrac{1}{1296621464378219}\,(61620438884996375 - 36597448164078450\sqrt3)$\\
$u_{356}(5)$ & $x$ & $\tfrac{1}{740926551073268}\,(2215517553045143 - 1593513777682625\sqrt3)$\\
 & $y$ & $\tfrac{1}{740926551073268}\,(-3824961606609125 + 1083020409715325\sqrt3)$\\
$u_{167}(7)$ & $x$ & $\tfrac{1}{1296621464378219}\,(-52727396315140531 + 31676859289595375\sqrt3)$\\
 & $y$ & $\tfrac{1}{1296621464378219}\,(73933893369047375 - 41948721148231350\sqrt3)$\\
$u_{465}(6)$ & $x$ & $\tfrac{1}{740926551073268}\,(-6748072298195457 + 4570380151920375\sqrt3)$\\
 & $y$ & $\tfrac{1}{740926551073268}\,(15219266623768875 - 6874978848677875\sqrt3)$\\
$u_{375}(3)$ & $x$ & $\tfrac{1}{77982519500461457}\,(-417997246102752500 + 227430000312057320\sqrt3)$\\
 & $y$ & $\tfrac{1}{77982519500461457}\,(1004622302211424670 - 439227692049635862\sqrt3)$\\
\bottomrule
\end{tabular}
\caption{The hexagonal net of Figure \ref{fig:nets}, its fifteen points read the same way, exact in $\Q(\sqrt3)$, the two coordinates of a point on consecutive lines and each point scaled by one common denominator.}\label{tab:nethex}
\end{table}

\begin{proof}[Proof of Proposition \ref{p:base}]
The nine non-glued edges are each placed twice in Tables \ref{tab:netsq} and \ref{tab:nethex}, and the two copies differ by a translation. Four of the nine translations vanish on the square side and three on the hexagonal one, the dashed edges of Figure \ref{fig:nets}. The edge $(0,7)$ of the left unfolding in Figure \ref{fig:nets}, for example, is placed by the faces $570$ and $207$, and the second copy is the first translated by $v_2$. Once the nine gluings are translations the two copies of a free edge agree, so summing the developed edge vectors along the two marking loops gives the same $v_1$ and $v_2$ whichever copy is read, and they come out as $v_1 \ne 0$ and $v_2 = i\,v_1$ on $T$, $v_2 = \rho\,v_1$ on $T'$, exactly, in $\Q$ and in $\Q(\sqrt3)$. That also makes $\det(v_1, v_2) > 0$ automatic. The nine translations make every cone angle exactly $2\pi$, by Section \ref{ss:developing}.
\end{proof}

\section{The lift}\label{sec:lift}

Section \ref{sec:bases} leaves us with two exactly flat configurations of exactly the right moduli, collapsed into the plane and hence not embedded. This section shows that they lie on the boundary of the embedded configurations. Lifting the vertices vertically out of the plane, each at its own constant speed, makes them embedded, for every choice of a positive time.

Throughout this section $Q \colon V \to \R^2$ is a planar configuration with nondegenerate faces and $t > 0$. The lift of $Q$ along a direction $\zeta \in \Q^V$ at time $t$ places vertex $v$ at $(Q_v,\ t\zeta_v) \in \R^3$, so $\zeta_v$ is the speed of vertex $v$, upward or downward by its sign, and $\zeta$ collects the eight speeds. The two directions $\zeta$ used for the bases are printed in Proposition \ref{p:separation}.

The trick is to choose the speeds $\zeta$ so that the lift is actually embedded, and it suffices to do so at one time, the lift being a family of affine stretches. Because we lift vertically, each lifted face is the graph of an affine height function over the planar triangle beneath it, and two faces meet exactly where their heights agree, over the region beneath both. Embeddedness thus becomes a pairwise condition on differences of height functions. Each may vanish where its two faces share a vertex or an edge, and must vanish nowhere else. Such a difference is affine, so its sign is decided at the corners of a convex polygon, and the condition comes down to finitely many exact sign evaluations.

The \emph{shadow} of a face $F = (a,b,c)$ is the planar triangle $Q(F)$ that the lifted face sits above. The three values $\zeta_a$, $\zeta_b$, $\zeta_c$, prescribed at the corners $Q_a$, $Q_b$, $Q_c$, which are not collinear, determine a unique affine function
\[
h_F \colon \R^2 \to \R ,
\]
the \emph{height function} of $F$. More concretely, $h_F(x) = \langle m, x\rangle + m_0$ with $m \in \R^2$ and $m_0 \in \R$. The three conditions $h_F(Q_v) = \zeta_v$ at the corners are the linear system
\[
\begin{pmatrix}
1 & Q_a^x & Q_a^y\\
1 & Q_b^x & Q_b^y\\
1 & Q_c^x & Q_c^y
\end{pmatrix}
\begin{pmatrix} m_0\\ m_1\\ m_2 \end{pmatrix}
\;=\;
\begin{pmatrix} \zeta_a\\ \zeta_b\\ \zeta_c \end{pmatrix}.
\]
So in particular, if the coordinates of $Q$ and $\zeta$ lie in a field $K \subset \R$, then so do $m$ and $m_0$.

The lifted face is the graph of $t\,h_F$ over its shadow, in particular injective, and two lifted faces meet over a point $x$ common to their shadows exactly when the \emph{gap} $h_F - h_G$ vanishes at $x$. The gap is the relative speed of the two sheets over $x$.

The argument now runs as follows. A corner criterion (Lemma \ref{l:sepcert}) decides, by finitely many signs, that two lifted faces meet only along the vertex or edge they share, if any. The two directions we propose satisfy it on all $120$ pairs of faces, so both lifts are embedded for every positive time (Proposition \ref{p:separation}). And the same holds for the lift of every planar configuration near enough the base (Lemma \ref{l:persist}), which is what Section \ref{sec:correction} uses.

We consider all simplices closed, and the \emph{shared simplex} of two faces is their intersection $F \cap G$ in the triangulation, a vertex or an edge or empty. It is read off the face lists and does not depend on the configuration. We write $Q(S)$ for its shadow, which can be either empty, a single point $Q_s$, or a segment from $Q_a$ to $Q_b$. The shadow of a face is exactly the set where three affine functions are nonnegative, one for each edge $(p,q)$, given by $\det(Q_q - Q_p,\ x - Q_p)$ signed so that its value at the third vertex is positive. The \emph{overlap} $K = A \cap B$ of two shadows is a compact convex polygon, possibly empty and possibly degenerate, and its \emph{corners} are its extreme points, the points of $K$ not interior to any segment of $K$.

\begin{lemma}[separation criterion]\label{l:sepcert}
Let $F \ne G$ be faces with shared simplex $S$, shadows $A$ and $B$, gap $g = h_F - h_G$, and overlap $K$. If $g$ is nonzero and has the same sign at every corner of $K$ outside $Q(S)$, then the two lifted faces meet exactly in the lift of $S$. They are disjoint when $S$ is empty, they meet only at the lifted vertex when $S$ is a vertex, and they meet only along the lifted edge when $S$ is an edge.
\end{lemma}

\begin{proof}
The lifted faces meet over $x$ precisely when $x$ lies in $K$ and $g(x) = 0$. Replacing $g$ by $-g$ if necessary, assume $g > 0$ at every corner of $K$ outside $Q(S)$. Both heights take the value $\zeta_v$ at a shared vertex $v$, so the gap vanishes on $Q(S)$, and $Q(S)$ lies in $K$ by convexity. Write $x$ in $K$ as a convex combination $x = \sum_i \lambda_i k_i$ of the corners $k_i$ of $K$. The gap is affine, so $g(x) = \sum_i \lambda_i\, g(k_i)$. The corners in $Q(S)$ contribute zero and the others a positive value, so $g(x)$ vanishes exactly when $\lambda_i = 0$ at every corner outside $Q(S)$, and $x$ then lies in $Q(S)$. Over $Q(S)$ the two lifts agree, and their common graph is the lift of $S$.
\end{proof}

Figure \ref{fig:corners} draws the criterion for a pair of faces of the square base that shares nothing, the one on which the smallest corner gap of Proposition \ref{p:separation} is attained.

\begin{figure}[!htbp]
\centering
\begin{tikzpicture}
\input{figs/sep-corners.tex}
\end{tikzpicture}
\qquad
\begin{tikzpicture}
\input{figs/sep-profile.tex}
\end{tikzpicture}
\caption{Two faces of the square base that share nothing. Left, the shadows of the faces $341$ and $506$, the overlap filled, its five corners dotted, and the corner circled where the smallest gap of Proposition \ref{p:separation} is attained. Right, a schematic profile of the same pair, each sheet solid over its own shadow and the sheets at least $0.00548\,t$ apart over the marked overlap.}\label{fig:corners}
\end{figure}
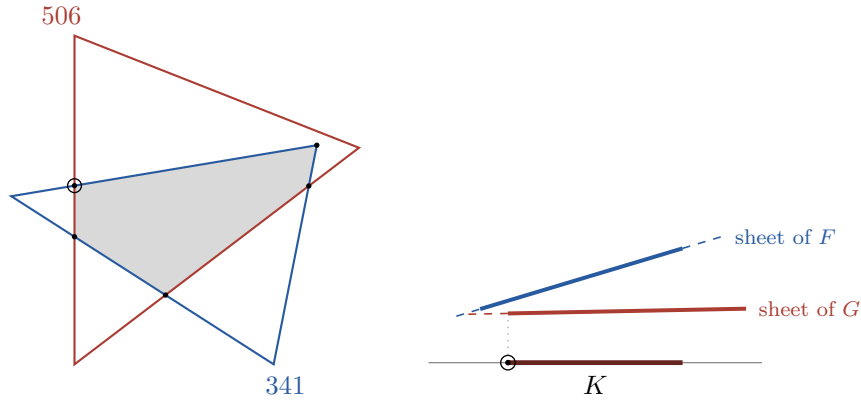

The hypothesis of Lemma \ref{l:sepcert} is a system of strict linear inequalities in $\zeta$, since each $h_F$ depends linearly on $\zeta$ and the corners of each overlap depend only on the base. Once the sign of each gap is fixed, its solutions form an open convex cone, and we took a point in it with short rational coordinates. Each pair of faces of the base now has to pass the criterion.

\begin{prop}\label{p:separation}
For the square base $Q^0$ with the lift direction
\[
\zeta \;=\; \big(\tfrac{3}{13},\ -\tfrac38,\ -\tfrac37,\ \tfrac{12}{13},\ 1,\ -1,\ 1,\ 0\big),
\]
and for the hexagonal base $P^0$ with the lift direction
\[
\zeta' \;=\; \big(\tfrac12,\ -1,\ \tfrac{6}{17},\ -\tfrac{7}{34},\ -\tfrac{33}{47},\ -1,\ \tfrac{19}{41},\ 1\big),
\]
the lift is embedded for every $t > 0$.
\end{prop}

\begin{proof}
Every one of the $\binom{16}{2} = 120$ pairs of faces of each base satisfies Lemma \ref{l:sepcert}, in exact arithmetic in its field. Everything is rational in the printed data, so the check is finitely many exact sign evaluations. The corners of an overlap are the vertices of one shadow lying in the other together with the crossings of their edges. Testing whether a vertex lies in a shadow is three sign evaluations, and finding a crossing is solving a linear $2 \times 2$ system. The gap is then evaluated at the corners outside $Q(S)$. On the square side, twenty-four pairs share an edge, seventy-two share exactly a vertex, and twenty-four share nothing, seventeen of the last with overlapping shadows. The smallest value of $|g|$ at a corner outside $Q(S)$ is $\tfrac{132043}{24086296}$, about $0.00548$, against a diameter of about two for $Q^0$. On the hexagonal side, twenty-four pairs share an edge, seventy-five share exactly a vertex, and twenty-one share nothing, sixteen of the last with overlapping shadows. The smallest such value is
\[
\tfrac{-25244566233401343 + 14596454895110643\,\sqrt3}{2513111807172148},
\]
a positive element of $\Q(\sqrt3)$, approximately $0.0148$, against a diameter of about one for $P^0$. The script \file{separation.py} of Appendix \ref{app:verify} runs all of it. So by Lemma \ref{l:sepcert} every two distinct lifted faces meet exactly in the lift of their shared simplex, and a configuration with nondegenerate faces in which every two faces meet exactly in their shared simplex is embedded, so the lift is embedded, for any $t > 0$.
\end{proof}

The next lemma tells us that once a lift is embedded, so are the lifts of a whole neighborhood of the horizontal coordinates, at every height.

\begin{lemma}[persistence of the separation]\label{l:persist}
There is, for each base, a neighborhood $U$ of it in $(\R^2)^V$, the same for every $t > 0$, such that for every planar configuration $Q \in U$ and every $t > 0$, the lift of $Q$ with heights $t\zeta$, respectively $t\zeta'$, is an embedded torus.
\end{lemma}

\begin{proof}
Embeddedness is an open condition on a configuration: at an embedded one, each pair of faces meets exactly in the image of its shared simplex by a positive margin, the margins move continuously with the vertices, and the pairwise fact from the proof of Proposition \ref{p:separation} keeps every nearby configuration embedded. The lift at $t = 1$ is continuous in $Q$ and embedded at the base by Proposition \ref{p:separation}, so the planar configurations with embedded lift at $t = 1$ form an open neighborhood $U$ of the base, and the invertible map $(x, y, z) \mapsto (x, y, tz)$ carries that lift to the lift at height $t$, so the same $U$ works for every $t > 0$.
\end{proof}

\section{The correction}\label{sec:correction}

The pieces are now in place, and we recall what each provides. Section \ref{sec:bases} produced two configurations, exactly flat and of exactly the right moduli, collapsed into the plane rather than embedded. Section \ref{sec:lift} lifted their vertices into $\R^3$, giving a one-parameter family of embedded configurations, along which flatness and the modulus drift, as Figure \ref{fig:triple} shows for $Q^0$, so none is yet the torus we want. This section remedies that. The correction moves only the horizontal positions, while the heights stay at their lifted values and are the parameter. We write flatness and the modulus together as a map $\Phi$ of the sixteen horizontal coordinates of the eight vertices, and show three things at the base. $\Phi$ vanishes, it is smooth, and its derivative, exact in the field, has full rank (Proposition \ref{p:jacobian}). The implicit function theorem then returns corrected horizontal positions, at each sufficiently small time, making the lifted configuration exactly flat with exactly the right modulus, and Lemma \ref{l:persist} keeps it embedded.

To build a map whose zeros are exactly the flat configurations of the desired modulus, we take
\[
\Phi \;=\; \big(\delta_0,\ \dots,\ \delta_6,\ \Rep\tau - \Rep\tau_0,\ \Imp\tau - \Imp\tau_0\big),
\]
where $\delta_0, \dots, \delta_6$ are seven of the eight angle deficits, the eighth being determined by the deficit sum of Section \ref{ss:geometry}, and the last two coordinates pin the modulus $\tau = v_2/v_1$, relative to the marking of Section \ref{sec:bases}, to its target $\tau_0$, the marked modulus $i$ for the square base and $\rho$ for the hexagonal one. A zero of $\Phi$ is then exactly what we want, all eight cone angles equal to $2\pi$ and marked modulus exactly $\tau_0$, and the base is one, by Proposition \ref{p:base}.

The domain of $\Phi$ needs care. The deficits are defined on all of $\mathcal{C}^\circ$, and the implicit function theorem needs $\Phi$ on an open neighborhood of the base, not only on the flat locus $\mathcal{F}$ through it. Section \ref{ss:developing} provides $\tau$ there as well, developed along a fixed tree with the periods summed along the marking loops, once we fix in which of its two faces each marking edge is read, the two developed copies no longer agreeing off $\mathcal{F}$. We develop along the tree of Section \ref{sec:bases} and read each marking edge in the earlier of its two faces, any other choice working equally well, all extensions agreeing on $\mathcal{F}$ (Section \ref{ss:developing}). With this, $\Phi$ is defined on a neighborhood of the base and is smooth there, the deficits wherever the faces are nondegenerate (Section \ref{ss:geometry}) and $\tau$ wherever in addition $v_1 \ne 0$ (Section \ref{ss:developing}), and both hold at the base by Section \ref{sec:bases}.

We select nine of the horizontal coordinates as unknowns and hold the remaining seven fixed, at their base values, leaving $\Phi$ a map of nine variables. The selection must make the corresponding $9 \times 9$ block of the Jacobian invertible, which is possible exactly when the full $9 \times 16$ Jacobian has rank nine. That Jacobian can be computed exactly. Although $\Phi$ is transcendental as a function of the configuration, its derivative at a planar base has entries in the field of the base, by Proposition \ref{p:rational} and Lemma \ref{l:planar}.

\begin{prop}\label{p:jacobian}
At each planar base the $9 \times 16$ Jacobian of $\Phi$ with respect to the horizontal coordinates has entries in the field of the base and exact rank $9$. The nine columns selected by Gaussian elimination are, for the square and the hexagonal base respectively,
\[
Q_0^x,\quad Q_0^y,\quad Q_1^x,\quad Q_1^y,\quad Q_2^y,\quad Q_3^x,\quad Q_4^x,\quad Q_4^y,\quad Q_5^x,
\]
\[
P_0^x,\quad P_0^y,\quad P_1^x,\quad P_1^y,\quad P_2^x,\quad P_2^y,\quad P_3^x,\quad P_3^y,\quad P_4^x,
\]
and the $9 \times 9$ minor $M$ in these columns is invertible.
\end{prop}

\begin{proof}
The computation is the script \file{development.py} of Appendix \ref{app:verify}, run in exact arithmetic on the printed data, over $\Q$ on the square side and over $\Q(\sqrt3)$ on the hexagonal one. It rebuilds the development of Section \ref{ss:developing} from each base, assembles the $9 \times 16 = 144$ entries from the angle derivative formula of Section \ref{ss:exact}, carrying one gradient alongside every intermediate quantity by the chain rule, one arithmetic operation at a time. On the square side
\[
\det M \;=\; \tfrac{4243180825891307520000}{4349139164624211581} \;\ne\; 0,
\]
and on the hexagonal side $\det M$ is a number $a + b\sqrt3$ whose two rational parts have $89$-digit numerators over a common $85$-digit denominator, printed in full by the script, both parts negative, so the determinant is nonzero, approximately $-7.328 \times 10^4$.
\end{proof}

We take these nine coordinates as the unknowns, written $x$, with $x^0$ their values at the base, the other seven held at their base values. Write $q(x, t)$ for the configuration whose nine free horizontal coordinates are $x$, whose other seven are those of the base, and whose heights are $t\zeta$, respectively $t\zeta'$, and set
\[
\Psi(x, t) \;=\; \Phi\big(q(x, t)\big),
\]
the map the implicit function theorem is applied to.

One structural fact finishes the construction, the same on each base. The heights enter $\Phi$ only through the geometric data (Proposition \ref{p:rational}), which the squared edge lengths determine, and the squared edge length of $(a,b)$ at heights $t\zeta$ is
\[
|Q_a - Q_b|^2 \;+\; t^2\,(\zeta_a - \zeta_b)^2,
\]
twenty-four numbers that keep every face nondegenerate and $v_1$ nonzero near the base. So $\Psi$ is smooth in $x$ and $t$ jointly near $(x^0, 0)$, by the smoothness of $\Phi$ established before Proposition \ref{p:jacobian}. Then $\Psi(x^0, 0) = 0$, by Proposition \ref{p:base}, with $\partial_x \Psi(x^0,0) = M$ invertible by Proposition \ref{p:jacobian}.

\begin{lemma}[correction lemma]\label{l:opening}
On each base there are $\varepsilon > 0$ and a continuous map $t \mapsto x(t)$ on $(-\varepsilon, \varepsilon)$ with $x(0) = x^0$ and $\Psi(x(t), t) = 0$. In particular, for $0 < t < \varepsilon$ the configuration $q_t = q(x(t), t)$ has all eight cone angles equal to $2\pi$ and modulus exactly $\tau_0$ relative to the marking.
\end{lemma}

\begin{proof}
The implicit function theorem applied to $\Psi$ at $(x^0, 0)$, where $\partial_x \Psi$ is the invertible matrix $M$, gives $\varepsilon > 0$ and a continuously differentiable map $x(t)$ on $(-\varepsilon, \varepsilon)$ with $x(0) = x^0$ and $\Psi(x(t), t) = 0$. The second statement is the deficit sum of Section \ref{ss:geometry}, which turns the seven vanishing deficits into eight.
\end{proof}

Nothing quantitative about the curve $t \mapsto x(t)$ is needed. The argument below uses only its continuity at $t = 0$. A marking is a pair of loops on $T$ and does not move when the configuration does, so its classes form a basis of $H_1(T;\Z)$ all along the curve, by the intersection argument of Section \ref{sec:bases}, and the modulus $\tau_0$ of Lemma \ref{l:opening} is therefore the reduced modulus.

\begin{proof}[Proof of Theorem \ref{t:main}]
On each base, let $Q(t)$ be the planar configuration whose nine free coordinates are the curve $x(t)$ of Lemma \ref{l:opening} and whose other seven are those of the base, so that $q_t$ is its lift at time $t$. By continuity $Q(t)$ stays in the neighborhood $U$ of Lemma \ref{l:persist} for all small $t$, so for small $t > 0$ the configuration $q_t$ is an embedded torus, and it is exactly flat of reduced modulus $i$, respectively $\rho$, by Lemma \ref{l:opening} and the paragraph preceding this proof. No torus triangulation has fewer than seven vertices (Section \ref{ss:geometry}), and no paper torus has seven \cite{schwartz-bound}, so eight is minimal.
\end{proof}

The same computation describes the flat configurations around the two tori.

\begin{theoremx}[\ref{t:path}, restated]
On each of the two triangulations there is an open neighborhood $W$ of the printed folded base in the set $\mathcal{C}^\circ$ of configurations with nondegenerate faces, with the following three properties. The flat locus $\mathcal{F} \cap W$, the configurations in $W$ with all eight cone angles $2\pi$, is a smooth manifold of dimension seventeen. The modulus is a submersion on it, so each of its nonempty fibers is a smooth manifold of dimension fifteen. And the fiber over the target, $i$ on $T$ and $\rho$ on $T'$, contains the configurations $q_t$ of the proof of Theorem \ref{t:main} for all sufficiently small $t > 0$, and these are embedded.
\end{theoremx}

\begin{proof}
The map $\Phi$ is continuously differentiable on the open set of configurations with nondegenerate faces and $v_1 \ne 0$ (Section \ref{ss:developing}), so the determinant of the $9 \times 9$ minor of the Jacobian of $\Phi$ in the nine free coordinates of Proposition \ref{p:jacobian} is continuous there, and it is nonzero at the base. Take $W$ to be the open neighborhood of the base where it stays nonzero and where $\det(v_1, v_2)$, positive at the base (Section \ref{sec:bases}), stays positive, so that $\tau$ maps $\mathcal{F} \cap W$ into $\HH$. The $9 \times 24$ Jacobian of $\Phi$ then has rank nine at every point of $W$.

The seven deficit rows of that Jacobian are rows of a rank-nine matrix, so they are linearly independent at every point of $W$. Zero is therefore a regular value of $(\delta_0, \dots, \delta_6)$ on $W$, its zero set $\mathcal{F} \cap W$ is a smooth manifold of dimension $24 - 7 = 17$, and the tangent space at each of its points is the kernel of the seven deficit differentials. On $\mathcal{F}$ the fixed-tree extension agrees with the modulus of Section \ref{ss:geometry}, by Section \ref{ss:developing}, so the differential of $\tau$ on that tangent space is the restriction of the two modulus rows. Suppose $\tau$ failed to be a submersion at a point of $\mathcal{F} \cap W$. Some nontrivial combination of the two modulus rows would then be orthogonal to the tangent space there, the common kernel of the seven deficit rows. A vector orthogonal to that kernel lies in the span of the seven, the row space of a matrix being the orthogonal complement of its null space. So the nine rows would be dependent, contradicting rank nine. Hence $\tau$ is a submersion on $\mathcal{F} \cap W$, and a nonempty fiber is a smooth manifold of dimension $17 - 2 = 15$.

The configurations $q_t$ converge to the base as $t \to 0$, their planar part being the continuous curve $Q(t)$ with $Q(0)$ the base and their heights being the lifted ones, so they lie in $W$ for all small $t > 0$. They are zeros of $\Phi$ (Lemma \ref{l:opening}), which puts them in the fiber over the target, and they are embedded for small $t$ by the proof of Theorem \ref{t:main}.
\end{proof}

The rectangular and the rhombic ray, the two that Doyle and Schwartz leave out \cite{doyle-schwartz}, both start at an orbifold point, and the submersion covers an initial segment in each.

\begin{proof}[Proof of Corollary \ref{c:rays}]
Fix a small $t > 0$ and let $p = q_t$ on $Q^0$, an embedded point of the fiber of $\tau$ over $i$ (Theorem \ref{t:path}). The set $\mathcal{E}$ of embedded configurations is open (Section \ref{ss:geometry}), so $p$ has an open neighborhood $N \subset \mathcal{E} \cap W$, and $\mathcal{F} \cap N$ is an open neighborhood of $p$ in the manifold $\mathcal{F} \cap W$. A submersion is an open map, so $\tau(\mathcal{F} \cap N)$ contains a ball $B$ around $i$ in $\HH$, and every point of $\mathcal{F} \cap N$ is flat and embedded, a paper torus of marked modulus $\tau$. The rhombic segment comes from $p = q_t$ on $P^0$ the same way, and $\varepsilon$ is the smaller of the two ball radii. The marked moduli in the statement are their own reduced representatives, $iY$ with $Y \ge 1$ lying in $\mathcal{D}$ and $\tfrac12 + iH$ with $H \ge \tfrac{\sqrt3}{2}$ as well, by $|\tau|^2 = \tfrac14 + H^2 \ge 1$ and the convention $\Rep\tauhat \ge 0$.
\end{proof}

\section{Open questions}\label{sec:discussion}

The first question is what remains of the two rays. Doyle and Schwartz supply every modulus off them \cite{doyle-schwartz}, and Theorem \ref{t:main} with Corollary \ref{c:rays} supplies an initial segment of each, so what is open is the rest. The two families of Section \ref{sec:bases} sweep the rays, and proving embeddedness along them is the subject of the sequel paper still in progress.

The second question asks for an explicitly given paper torus. Theorem \ref{t:main} produces the tori $q_t$ without coordinates, and we know no $8$-vertex paper torus of modulus exactly $i$ or exactly $\rho$ with exact coordinates. The tori of Appendix \ref{app:fat} are only numerical, however they raise another mathematical question, taken up below.

Two more questions concern the triangulations. Of the seven $8$-vertex torus triangulations, one carries no paper torus at all. It has a vertex of valence $3$, where the three angles sum to $2\pi$ by flatness, and three unit vectors whose pairwise angles sum to $2\pi$ are coplanar, so the star is flat and erasing the vertex leaves a $7$-vertex paper torus, impossible by \cite{schwartz-bound}. Of the six that remain, five realize the square modulus numerically. The last, with degree sequence $(4,5,5,6,7,7,7,7)$, resisted every numerical search. We found no paper torus on it at any modulus, and we suspect that none exists. Near universality was proved on the valence-regular $T$ \cite{doyle-schwartz}, and the hexagonal torus is realized here on $T'$. A single triangulation carrying every flat torus exists, the universal one of \cite{lazarus-tallerie} with $2434$ faces, but we do not know whether one of the $8$-vertex triangulations realizes every flat torus.

The last question concerns the space of all realizations. Fix one of the two triangulations and consider the embedded flat configurations on it whose marked modulus is exactly the target $\tau_0$ of Section \ref{sec:correction}, $i$ on $T$ and $\rho$ on $T'$, the set $\mathcal{E} \cap \mathcal{F} \cap \tau^{-1}(\tau_0)$ in the notation of Section \ref{ss:geometry}. Every point of it is an $8$-vertex paper torus realizing the square torus, respectively the hexagonal one, up to scaling. Theorem \ref{t:main} makes this set nonempty, and near the tori $q_t$ it is a smooth fifteen dimensional manifold, the fiber of the submersion of Theorem \ref{t:path} (eight dimensional after quotienting by the similarities of $\R^3$). Its global shape is unknown to us, and the same is true at every modulus off the two rays, where near universality \cite{doyle-schwartz} makes the set nonempty. Two realizations in one path component deform into one another through paper tori of the same shape. We ask how many path components there are, and whether every one of them has a planar configuration in its closure, so that every realization can be driven into the plane through paper tori of its own modulus. Appendix \ref{app:fat} gives the one piece of numerical evidence we have on the second question, an inflated torus joined to the tori $q_t$ by a path traced numerically inside this set, and following such paths toward the plane is how the two families of Section \ref{sec:bases} were found in the first place.

\appendix

\section{Verification}\label{app:verify}

Every claim the paper makes about the two bases is checked by short Python scripts using exact arithmetic on the data of Table \ref{tab:base}, plain fractions on the square side and pairs of fractions representing $a + b\sqrt3$ on the hexagonal one. No floating point number enters any accepted claim, and each individual check is small enough to easily redo in a computer algebra system.

\begin{center}
\small
\begin{tabular}{@{}l l l@{}}
\toprule
script & checks & in the text\\
\midrule
\file{field.py} & the ordered field $\Q(\sqrt3)$ & Section \ref{sec:bases}\\
\file{combinatorics.py} & each face list is a triangulated torus & Section \ref{sec:bases}\\
\file{separation.py} & the gap at the corners of each overlap, all $120$ face pairs & Prop.\ \ref{p:separation}\\
\file{development.py} & face signs, the developed coordinates, holonomies, $\tau$, $\det M \ne 0$ & Props.\ \ref{p:base}, \ref{p:jacobian}\\
\bottomrule
\end{tabular}
\end{center}

The code is a small public repository,
\begin{center}
\url{https://github.com/FabianLander/eight-vertex-paper-tori-verification}
\end{center}
and runs from the single command \file{python3 verify.py}. The whole run, for both tori, takes under two seconds and prints every check together with the value it recomputed. All four read their input from one file, which holds the two columns of Table \ref{tab:base} and the two lift directions of Proposition \ref{p:separation} and nothing else. The combinatorics script tests five conditions on the two face lists: distinct faces, every directed edge traversed by exactly one face cycle, a single cycle of faces around each vertex, connectedness, and $|F| = 2|V|$. The first four make each list a connected closed oriented surface, and the count $8 - 24 + 16 = 0$ then makes it a torus. The development script rebuilds the development of Section \ref{ss:developing} from the printed coordinates and derives from it everything Propositions \ref{p:base} and \ref{p:jacobian} assert, including that the nine columns Gaussian elimination selects are the nine the proposition prints. It checks as well that the eight vertex holonomies of each base are the identity, redundant given the nine gluing checks by the argument of Section \ref{ss:developing}, and kept as a sanity check. One control runs alongside, not part of any proof: the Jacobian agrees, entry by entry and to within $10^{-9}$, with central differences of an independent floating point evaluation of $\Phi$.

\section{Two inflated tori}\label{app:fat}

The tori of Theorem \ref{t:main} are nearly planar, and no scissors will reproduce them. The two tori of Figure \ref{fig:fat-renders}, with coordinates in Tables \ref{tab:fat-square} and \ref{tab:fat-hex}, are the same two flat tori, realized farther from the plane. Write $\mathcal{R}$ for the embedded flat configurations of marked modulus exactly $\tau_0$ on the triangulation at hand, the set $\mathcal{E} \cap \mathcal{F} \cap \tau^{-1}(\tau_0)$ of Section \ref{ss:geometry} with $\tau_0 = i$ on $T$ and $\tau_0 = \rho$ on $T'$, whose shape Section \ref{sec:discussion} asks about. The printed square torus is, to numerical precision, a point of $\mathcal{R}$, and a numerical path inside $\mathcal{R}$ joins it to the tori $q_t$ of Theorem \ref{t:main}, a flow that drives the torus toward the plane while a projection holds flatness and the modulus and barriers keep the faces apart, followed by a short walk in the fiber onto the family $q_t$ (Figure \ref{fig:inflation}). The hexagonal case reads the same. Run all the way down to the plane, the same flow lands on the two families of Section \ref{sec:bases} at their corners, which is how the folded bases were found.

The evidence for the two paths is numerical, a chain of configurations with small steps, each verified flat to $2 \times 10^{-12}$, at its modulus to $10^{-12}$, and embedded, with the separation between disjoint faces never below $3 \times 10^{-4}$, respectively $2 \times 10^{-5}$, of the configuration's size. Dense sampling is not continuity, so what this supports is the intuition behind the question of Section \ref{sec:discussion}. Of the two printed tori themselves we claim only this. Read as exact rational numbers, both are exactly embedded, verified by a separating-axis reduction in exact rational arithmetic, their cone angles equal $2\pi$ to within $3 \times 10^{-15}$, and their marked moduli match $i$ and $\rho$ to fifteen digits.

Figures \ref{fig:net-square} and \ref{fig:net-hex} are printable nets, one per torus, each carrying the whole surface as one connected sheet. Solid lines are cuts. A green dash-dotted crease is a mountain and a purple dashed one a valley, and each crease is labeled M or V together with the angle between its two faces, measured on the printed side, which is the side that ends up outside. A letter pairs a glue tab with the edge it is glued behind. The angles are only a guide. Once the faces are folded by roughly the right amount, the torus closes up correctly by itself. We can report that both tori do fold up in practice.

\begin{table}[!htbp]
\centering
\footnotesize
\setlength{\tabcolsep}{5pt}
\input{figs/table-coords.tex}
\caption{The inflated square torus, the eight vertex positions to full double precision.}\label{tab:fat-square}
\end{table}

\begin{table}[!htbp]
\centering
\footnotesize
\setlength{\tabcolsep}{5pt}
\input{figs/table-coords-rho.tex}
\caption{The inflated hexagonal torus, the same way.}\label{tab:fat-hex}
\end{table}

\begin{figure}[p]
\centering
\includegraphics[page=2,width=\textwidth,height=0.90\textheight,keepaspectratio]{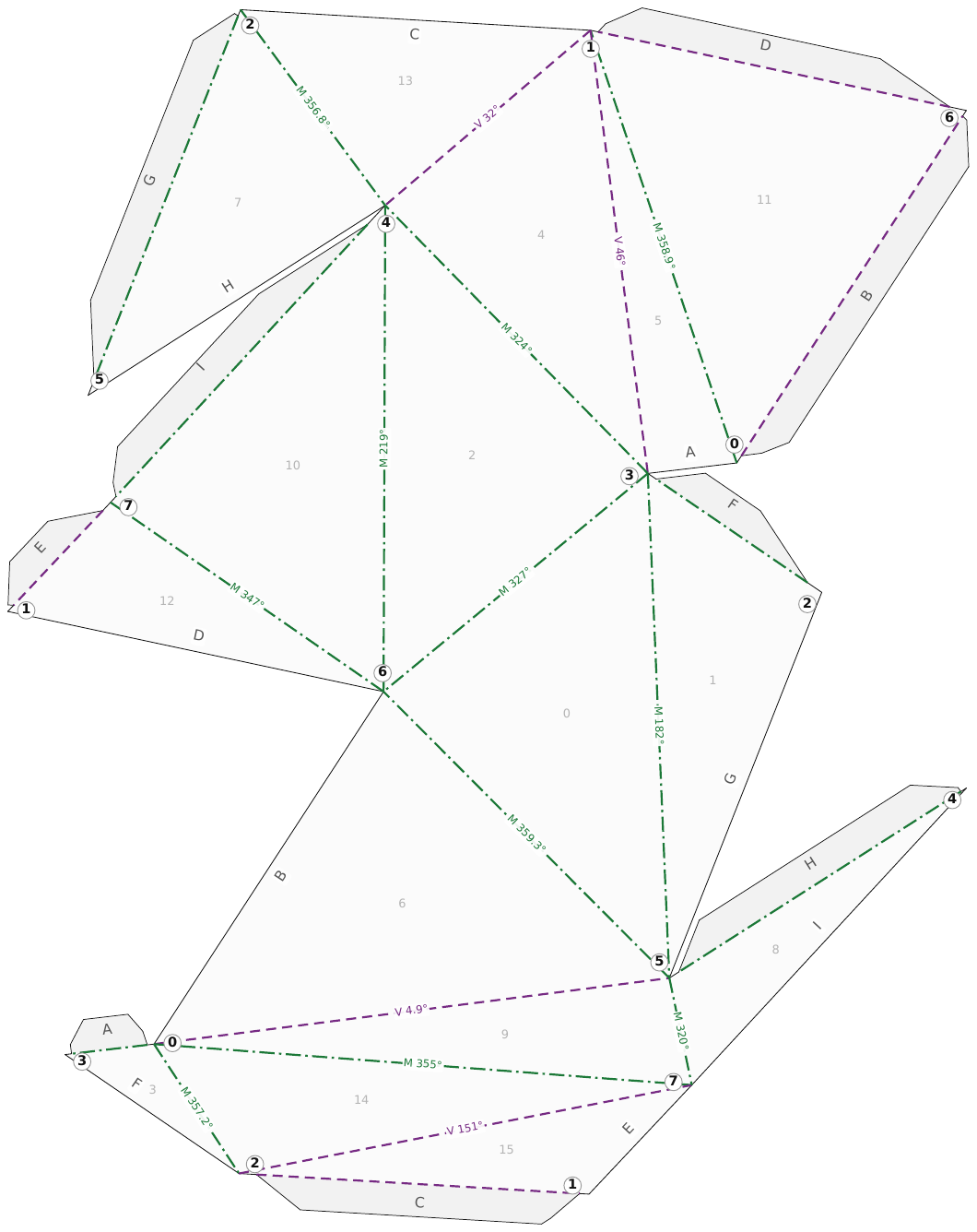}
\caption{The cut-out net of the inflated square torus, fifteen creases, nine glued seams, no interior vertex.}\label{fig:net-square}
\end{figure}

\begin{figure}[p]
\centering
\includegraphics[page=3,width=\textwidth,height=0.90\textheight,keepaspectratio]{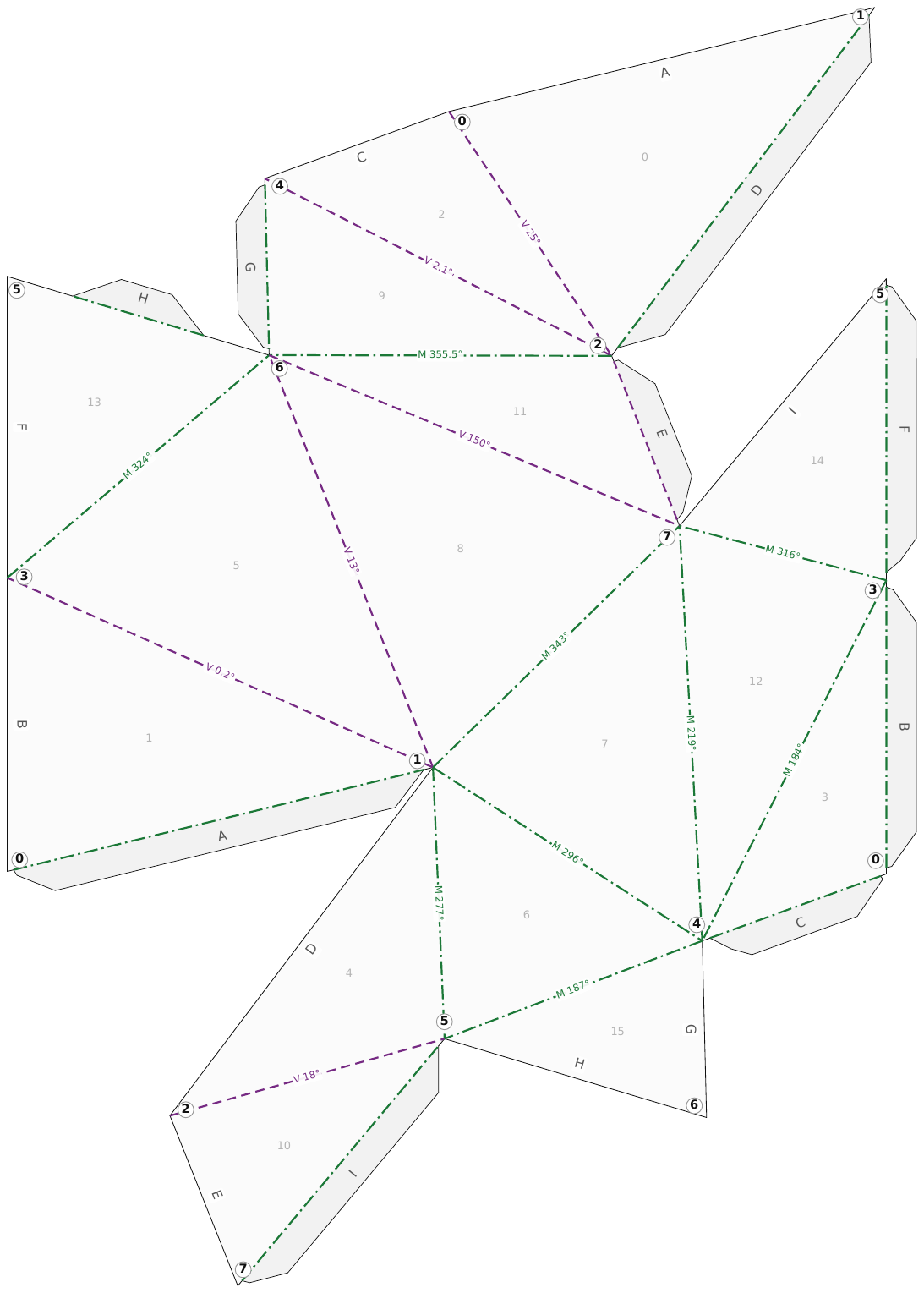}
\caption{The cut-out net of the inflated hexagonal torus, fifteen creases, nine glued seams, no interior vertex.}\label{fig:net-hex}
\end{figure}

\end{document}

%% file: figs/s5-fd.tex
\draw[->, black!70] (-2.6000,0) -- (2.9900,0);
\draw[black!70] (-1.3000,0) -- (-1.3000,0.14);
\draw[black!70] (1.3000,0) -- (1.3000,0.14);
\draw[black!35, dashed] (-1.3000,0) -- (-1.3000,2.2517);
\draw[black!35, dashed] (1.3000,0) -- (1.3000,2.2517);
\draw[->, black!70] (0,0) -- (0,4.9400) node[above] {$\mathrm{Im}\,\tau$};
\fill[figblue, fill opacity=0.06] (-1.3000,4.8100) -- (-1.3000,2.2517) arc (120:60:2.6000) -- (1.3000,4.8100) -- cycle;
\draw[black!35, dashed] (-2.6000,0) arc (180:0:2.6000);
\draw[black!75, line width=0.9pt, dash pattern=on 3pt off 2.4pt] (-1.3000,2.2517) arc (120:90:2.6000);
\draw[black!75, line width=0.9pt] (90:2.6000) arc (90:60:2.6000);
\draw[-{Stealth[length=7pt,width=6pt]}, figred, line width=1.4pt, dash pattern=on 3pt off 2.4pt] (-1.3000,3.3800) -- (-1.3000,3.8600);
\draw[-{Stealth[length=7pt,width=6pt]}, figred, line width=1.4pt] (1.3000,3.3800) -- (1.3000,3.8600);
\draw[-{Stealth[length=6pt,width=5pt]Stealth[length=6pt,width=5pt]}, black!75, line width=0.9pt] (73:2.6000) arc (73:69:2.6000);
\draw[-{Stealth[length=6pt,width=5pt]Stealth[length=6pt,width=5pt]}, black!75, line width=0.9pt] (107:2.6000) arc (107:111:2.6000);
\draw[figred, line width=1.5pt] (0,2.6000) -- (0,4.8100);
\draw[figred, line width=1.5pt] (1.3000,2.2517) -- (1.3000,4.8100);
\draw[figred, line width=1.5pt, dash pattern=on 3pt off 2.4pt] (-1.3000,2.2517) -- (-1.3000,4.8100);
\fill[figblue] (1.3000,2.2517) circle (2.2pt);
\node[figblue, right, inner sep=3pt] at (1.3000,2.2517) {$e^{i\pi/3}$};
\draw[figblue, line width=0.9pt, fill=white] (-1.3000,2.2517) circle (2.0pt);
\fill[figblue] (0,2.6000) circle (2.2pt);
\node[figblue, below right, inner sep=2.5pt] at (0.02,2.6000) {$i$};
\node[below] at (-1.3000,-0.05) {$-\tfrac12$};
\node[below] at (1.3000,-0.05) {$\tfrac12$};
\node[below] at (0,-0.05) {$0$};

%% file: figs/universal-square.tex
\providecolor{figblue}{RGB}{40,90,160}
\providecolor{figred}{RGB}{170,60,50}
\tikzset{
  ucover/.style={gray!45, line width=0.2pt},
  unet/.style={gray!70, line width=0.3pt},
  uout/.style={black, line width=0.9pt, line join=round},
  ucut/.style={black, line width=0.9pt, dash pattern=on 2.4pt off 1.8pt},
  ulab/.style={font=\tiny, inner sep=0.6pt, fill=white,
               fill opacity=0.72, text opacity=1},
  unum/.style={font=\tiny, text=black!55, inner sep=0.5pt},
  uloop1/.style={figblue, line width=1.4pt, opacity=0.75,
                 line cap=round, line join=round},
  uloop2/.style={figred, line width=1.4pt, opacity=0.75,
                 line cap=round, line join=round},
}
\path[use as bounding box] (-3.0000,-3.0000) rectangle (3.0000,3.0000);
\begin{scope}
\clip (-3.0000,-3.0000) rectangle (3.0000,3.0000);
\draw[ucover] (-0.5774,0.0000) -- (0.5774,0.0000) -- (0.0000,1.0000) -- cycle;
\draw[ucover] (0.5774,0.0000) -- (-0.5774,0.0000) -- (0.0000,-1.0000) -- cycle;
\draw[ucover] (0.0000,1.0000) -- (0.5774,0.0000) -- (1.1547,1.0000) -- cycle;
\draw[ucover] (-0.5774,0.0000) -- (0.0000,1.0000) -- (-1.1547,1.0000) -- cycle;
\draw[ucover] (0.0000,-1.0000) -- (-0.5774,0.0000) -- (-1.1547,-1.0000) -- cycle;
\draw[ucover] (0.5774,0.0000) -- (0.0000,-1.0000) -- (1.1547,-1.0000) -- cycle;
\draw[ucover] (1.1547,1.0000) -- (0.5774,0.0000) -- (1.7321,0.0000) -- cycle;
\draw[ucover] (0.0000,1.0000) -- (1.1547,1.0000) -- (0.5774,2.0000) -- cycle;
\draw[ucover] (-1.1547,1.0000) -- (0.0000,1.0000) -- (-0.5774,2.0000) -- cycle;
\draw[ucover] (-0.5774,0.0000) -- (-1.1547,1.0000) -- (-1.7321,0.0000) -- cycle;
\draw[ucover] (-1.1547,-1.0000) -- (-0.5774,0.0000) -- (-1.7321,0.0000) -- cycle;
\draw[ucover] (0.0000,-1.0000) -- (-1.1547,-1.0000) -- (-0.5774,-2.0000) -- cycle;
\draw[ucover] (1.1547,-1.0000) -- (0.0000,-1.0000) -- (0.5774,-2.0000) -- cycle;
\draw[ucover] (0.5774,0.0000) -- (1.1547,-1.0000) -- (1.7321,0.0000) -- cycle;
\draw[ucover] (1.1547,1.0000) -- (1.7321,0.0000) -- (2.3094,1.0000) -- cycle;
\draw[ucover] (0.5774,2.0000) -- (1.1547,1.0000) -- (1.7321,2.0000) -- cycle;
\draw[ucover] (0.0000,1.0000) -- (0.5774,2.0000) -- (-0.5774,2.0000) -- cycle;
\draw[ucover] (-1.1547,1.0000) -- (-0.5774,2.0000) -- (-1.7321,2.0000) -- cycle;
\draw[ucover] (-1.7321,0.0000) -- (-1.1547,1.0000) -- (-2.3094,1.0000) -- cycle;
\draw[ucover] (-1.1547,-1.0000) -- (-1.7321,0.0000) -- (-2.3094,-1.0000) -- cycle;
\draw[ucover] (-0.5774,-2.0000) -- (-1.1547,-1.0000) -- (-1.7321,-2.0000) -- cycle;
\draw[ucover] (0.0000,-1.0000) -- (-0.5774,-2.0000) -- (0.5774,-2.0000) -- cycle;
\draw[ucover] (1.1547,-1.0000) -- (0.5774,-2.0000) -- (1.7321,-2.0000) -- cycle;
\draw[ucover] (1.7321,0.0000) -- (1.1547,-1.0000) -- (2.3094,-1.0000) -- cycle;
\draw[ucover] (2.3094,1.0000) -- (1.7321,0.0000) -- (2.8868,0.0000) -- cycle;
\draw[ucover] (1.1547,1.0000) -- (2.3094,1.0000) -- (1.7321,2.0000) -- cycle;
\draw[ucover] (0.5774,2.0000) -- (1.7321,2.0000) -- (1.1547,3.0000) -- cycle;
\draw[ucover] (-0.5774,2.0000) -- (0.5774,2.0000) -- (0.0000,3.0000) -- cycle;
\draw[ucover] (-1.7321,2.0000) -- (-0.5774,2.0000) -- (-1.1547,3.0000) -- cycle;
\draw[ucover] (-1.1547,1.0000) -- (-1.7321,2.0000) -- (-2.3094,1.0000) -- cycle;
\draw[ucover] (-1.7321,0.0000) -- (-2.3094,1.0000) -- (-2.8868,0.0000) -- cycle;
\draw[ucover] (-2.3094,-1.0000) -- (-1.7321,0.0000) -- (-2.8868,0.0000) -- cycle;
\draw[ucover] (-1.1547,-1.0000) -- (-2.3094,-1.0000) -- (-1.7321,-2.0000) -- cycle;
\draw[ucover] (-0.5774,-2.0000) -- (-1.7321,-2.0000) -- (-1.1547,-3.0000) -- cycle;
\draw[ucover] (0.5774,-2.0000) -- (-0.5774,-2.0000) -- (0.0000,-3.0000) -- cycle;
\draw[ucover] (1.7321,-2.0000) -- (0.5774,-2.0000) -- (1.1547,-3.0000) -- cycle;
\draw[ucover] (1.1547,-1.0000) -- (1.7321,-2.0000) -- (2.3094,-1.0000) -- cycle;
\draw[ucover] (1.7321,0.0000) -- (2.3094,-1.0000) -- (2.8868,0.0000) -- cycle;
\draw[ucover] (2.3094,1.0000) -- (2.8868,0.0000) -- (3.4641,1.0000) -- cycle;
\draw[ucover] (1.7321,2.0000) -- (2.3094,1.0000) -- (2.8868,2.0000) -- cycle;
\draw[ucover] (1.1547,3.0000) -- (1.7321,2.0000) -- (2.3094,3.0000) -- cycle;
\draw[ucover] (0.5774,2.0000) -- (1.1547,3.0000) -- (0.0000,3.0000) -- cycle;
\draw[ucover] (-0.5774,2.0000) -- (0.0000,3.0000) -- (-1.1547,3.0000) -- cycle;
\draw[ucover] (-1.7321,2.0000) -- (-1.1547,3.0000) -- (-2.3094,3.0000) -- cycle;
\draw[ucover] (-2.3094,1.0000) -- (-1.7321,2.0000) -- (-2.8868,2.0000) -- cycle;
\draw[ucover] (-2.8868,0.0000) -- (-2.3094,1.0000) -- (-3.4641,1.0000) -- cycle;
\draw[ucover] (-2.3094,-1.0000) -- (-2.8868,0.0000) -- (-3.4641,-1.0000) -- cycle;
\draw[ucover] (-1.7321,-2.0000) -- (-2.3094,-1.0000) -- (-2.8868,-2.0000) -- cycle;
\draw[ucover] (-1.1547,-3.0000) -- (-1.7321,-2.0000) -- (-2.3094,-3.0000) -- cycle;
\draw[ucover] (-0.5774,-2.0000) -- (-1.1547,-3.0000) -- (0.0000,-3.0000) -- cycle;
\draw[ucover] (0.5774,-2.0000) -- (0.0000,-3.0000) -- (1.1547,-3.0000) -- cycle;
\draw[ucover] (1.7321,-2.0000) -- (1.1547,-3.0000) -- (2.3094,-3.0000) -- cycle;
\draw[ucover] (2.3094,-1.0000) -- (1.7321,-2.0000) -- (2.8868,-2.0000) -- cycle;
\draw[ucover] (2.8868,0.0000) -- (2.3094,-1.0000) -- (3.4641,-1.0000) -- cycle;
\draw[ucover] (3.4641,1.0000) -- (2.8868,0.0000) -- (4.0415,0.0000) -- cycle;
\draw[ucover] (2.3094,1.0000) -- (3.4641,1.0000) -- (2.8868,2.0000) -- cycle;
\draw[ucover] (1.7321,2.0000) -- (2.8868,2.0000) -- (2.3094,3.0000) -- cycle;
\draw[ucover] (1.1547,3.0000) -- (2.3094,3.0000) -- (1.7321,4.0000) -- cycle;
\draw[ucover] (0.0000,3.0000) -- (1.1547,3.0000) -- (0.5774,4.0000) -- cycle;
\draw[ucover] (-1.1547,3.0000) -- (0.0000,3.0000) -- (-0.5774,4.0000) -- cycle;
\draw[ucover] (-2.3094,3.0000) -- (-1.1547,3.0000) -- (-1.7321,4.0000) -- cycle;
\draw[ucover] (-1.7321,2.0000) -- (-2.3094,3.0000) -- (-2.8868,2.0000) -- cycle;
\draw[ucover] (-2.3094,1.0000) -- (-2.8868,2.0000) -- (-3.4641,1.0000) -- cycle;
\draw[ucover] (-2.8868,0.0000) -- (-3.4641,1.0000) -- (-4.0415,0.0000) -- cycle;
\draw[ucover] (-3.4641,-1.0000) -- (-2.8868,0.0000) -- (-4.0415,0.0000) -- cycle;
\draw[ucover] (-2.3094,-1.0000) -- (-3.4641,-1.0000) -- (-2.8868,-2.0000) -- cycle;
\draw[ucover] (-1.7321,-2.0000) -- (-2.8868,-2.0000) -- (-2.3094,-3.0000) -- cycle;
\draw[ucover] (-1.1547,-3.0000) -- (-2.3094,-3.0000) -- (-1.7321,-4.0000) -- cycle;
\draw[ucover] (0.0000,-3.0000) -- (-1.1547,-3.0000) -- (-0.5774,-4.0000) -- cycle;
\draw[ucover] (1.1547,-3.0000) -- (0.0000,-3.0000) -- (0.5774,-4.0000) -- cycle;
\draw[ucover] (2.3094,-3.0000) -- (1.1547,-3.0000) -- (1.7321,-4.0000) -- cycle;
\draw[ucover] (1.7321,-2.0000) -- (2.3094,-3.0000) -- (2.8868,-2.0000) -- cycle;
\draw[ucover] (2.3094,-1.0000) -- (2.8868,-2.0000) -- (3.4641,-1.0000) -- cycle;
\draw[ucover] (2.8868,0.0000) -- (3.4641,-1.0000) -- (4.0415,0.0000) -- cycle;
\draw[ucover] (2.8868,2.0000) -- (3.4641,1.0000) -- (4.0415,2.0000) -- cycle;
\draw[ucover] (2.3094,3.0000) -- (2.8868,2.0000) -- (3.4641,3.0000) -- cycle;
\draw[ucover] (1.7321,4.0000) -- (2.3094,3.0000) -- (2.8868,4.0000) -- cycle;
\draw[ucover] (1.1547,3.0000) -- (1.7321,4.0000) -- (0.5774,4.0000) -- cycle;
\draw[ucover] (0.0000,3.0000) -- (0.5774,4.0000) -- (-0.5774,4.0000) -- cycle;
\draw[ucover] (-1.1547,3.0000) -- (-0.5774,4.0000) -- (-1.7321,4.0000) -- cycle;
\draw[ucover] (-2.3094,3.0000) -- (-1.7321,4.0000) -- (-2.8868,4.0000) -- cycle;
\draw[ucover] (-2.8868,2.0000) -- (-2.3094,3.0000) -- (-3.4641,3.0000) -- cycle;
\draw[ucover] (-3.4641,1.0000) -- (-2.8868,2.0000) -- (-4.0415,2.0000) -- cycle;
\draw[ucover] (-2.8868,-2.0000) -- (-3.4641,-1.0000) -- (-4.0415,-2.0000) -- cycle;
\draw[ucover] (-2.3094,-3.0000) -- (-2.8868,-2.0000) -- (-3.4641,-3.0000) -- cycle;
\draw[ucover] (-1.7321,-4.0000) -- (-2.3094,-3.0000) -- (-2.8868,-4.0000) -- cycle;
\draw[ucover] (-1.1547,-3.0000) -- (-1.7321,-4.0000) -- (-0.5774,-4.0000) -- cycle;
\draw[ucover] (0.0000,-3.0000) -- (-0.5774,-4.0000) -- (0.5774,-4.0000) -- cycle;
\draw[ucover] (1.1547,-3.0000) -- (0.5774,-4.0000) -- (1.7321,-4.0000) -- cycle;
\draw[ucover] (2.3094,-3.0000) -- (1.7321,-4.0000) -- (2.8868,-4.0000) -- cycle;
\draw[ucover] (2.8868,-2.0000) -- (2.3094,-3.0000) -- (3.4641,-3.0000) -- cycle;
\draw[ucover] (3.4641,-1.0000) -- (2.8868,-2.0000) -- (4.0415,-2.0000) -- cycle;
\draw[ucover] (2.8868,2.0000) -- (4.0415,2.0000) -- (3.4641,3.0000) -- cycle;
\draw[ucover] (2.3094,3.0000) -- (3.4641,3.0000) -- (2.8868,4.0000) -- cycle;
\draw[ucover] (-2.3094,3.0000) -- (-2.8868,4.0000) -- (-3.4641,3.0000) -- cycle;
\draw[ucover] (-2.8868,2.0000) -- (-3.4641,3.0000) -- (-4.0415,2.0000) -- cycle;
\draw[ucover] (-2.8868,-2.0000) -- (-4.0415,-2.0000) -- (-3.4641,-3.0000) -- cycle;
\draw[ucover] (-2.3094,-3.0000) -- (-3.4641,-3.0000) -- (-2.8868,-4.0000) -- cycle;
\draw[ucover] (2.3094,-3.0000) -- (2.8868,-4.0000) -- (3.4641,-3.0000) -- cycle;
\draw[ucover] (2.8868,-2.0000) -- (3.4641,-3.0000) -- (4.0415,-2.0000) -- cycle;
\draw[ucover] (2.8868,4.0000) -- (3.4641,3.0000) -- (4.0415,4.0000) -- cycle;
\draw[ucover] (-3.4641,3.0000) -- (-2.8868,4.0000) -- (-4.0415,4.0000) -- cycle;
\draw[ucover] (-2.8868,-4.0000) -- (-3.4641,-3.0000) -- (-4.0415,-4.0000) -- cycle;
\draw[ucover] (3.4641,-3.0000) -- (2.8868,-4.0000) -- (4.0415,-4.0000) -- cycle;
\fill[figblue, fill opacity=0.09] (-0.5774,0.0000) -- (0.5774,0.0000) -- (0.0000,1.0000) -- cycle;
\draw[unet] (-0.5774,0.0000) -- (0.5774,0.0000) -- (0.0000,1.0000) -- cycle;
\fill[figblue, fill opacity=0.09] (-0.5774,0.0000) -- (0.0000,-1.0000) -- (0.5774,0.0000) -- cycle;
\draw[unet] (-0.5774,0.0000) -- (0.0000,-1.0000) -- (0.5774,0.0000) -- cycle;
\fill[figblue, fill opacity=0.09] (-0.5774,0.0000) -- (0.0000,1.0000) -- (-1.1547,1.0000) -- cycle;
\draw[unet] (-0.5774,0.0000) -- (0.0000,1.0000) -- (-1.1547,1.0000) -- cycle;
\fill[figblue, fill opacity=0.09] (-0.5774,0.0000) -- (-1.1547,-1.0000) -- (0.0000,-1.0000) -- cycle;
\draw[unet] (-0.5774,0.0000) -- (-1.1547,-1.0000) -- (0.0000,-1.0000) -- cycle;
\fill[figblue, fill opacity=0.09] (-0.5774,0.0000) -- (-1.1547,1.0000) -- (-1.7321,0.0000) -- cycle;
\draw[unet] (-0.5774,0.0000) -- (-1.1547,1.0000) -- (-1.7321,0.0000) -- cycle;
\fill[figblue, fill opacity=0.09] (-0.5774,0.0000) -- (-1.7321,0.0000) -- (-1.1547,-1.0000) -- cycle;
\draw[unet] (-0.5774,0.0000) -- (-1.7321,0.0000) -- (-1.1547,-1.0000) -- cycle;
\fill[figblue, fill opacity=0.09] (0.5774,0.0000) -- (1.1547,1.0000) -- (0.0000,1.0000) -- cycle;
\draw[unet] (0.5774,0.0000) -- (1.1547,1.0000) -- (0.0000,1.0000) -- cycle;
\fill[figblue, fill opacity=0.09] (0.5774,0.0000) -- (0.0000,-1.0000) -- (1.1547,-1.0000) -- cycle;
\draw[unet] (0.5774,0.0000) -- (0.0000,-1.0000) -- (1.1547,-1.0000) -- cycle;
\fill[figblue, fill opacity=0.09] (0.5774,0.0000) -- (1.1547,-1.0000) -- (1.7321,0.0000) -- cycle;
\draw[unet] (0.5774,0.0000) -- (1.1547,-1.0000) -- (1.7321,0.0000) -- cycle;
\fill[figblue, fill opacity=0.09] (0.5774,0.0000) -- (1.7321,0.0000) -- (1.1547,1.0000) -- cycle;
\draw[unet] (0.5774,0.0000) -- (1.7321,0.0000) -- (1.1547,1.0000) -- cycle;
\fill[figblue, fill opacity=0.09] (0.0000,1.0000) -- (-0.5774,2.0000) -- (-1.1547,1.0000) -- cycle;
\draw[unet] (0.0000,1.0000) -- (-0.5774,2.0000) -- (-1.1547,1.0000) -- cycle;
\fill[figblue, fill opacity=0.09] (0.0000,1.0000) -- (1.1547,1.0000) -- (0.5774,2.0000) -- cycle;
\draw[unet] (0.0000,1.0000) -- (1.1547,1.0000) -- (0.5774,2.0000) -- cycle;
\fill[figblue, fill opacity=0.09] (0.0000,1.0000) -- (0.5774,2.0000) -- (-0.5774,2.0000) -- cycle;
\draw[unet] (0.0000,1.0000) -- (0.5774,2.0000) -- (-0.5774,2.0000) -- cycle;
\fill[figblue, fill opacity=0.09] (0.0000,-1.0000) -- (0.5774,-2.0000) -- (1.1547,-1.0000) -- cycle;
\draw[unet] (0.0000,-1.0000) -- (0.5774,-2.0000) -- (1.1547,-1.0000) -- cycle;
\fill[figblue, fill opacity=0.09] (0.0000,-1.0000) -- (-1.1547,-1.0000) -- (-0.5774,-2.0000) -- cycle;
\draw[unet] (0.0000,-1.0000) -- (-1.1547,-1.0000) -- (-0.5774,-2.0000) -- cycle;
\fill[figblue, fill opacity=0.09] (0.0000,-1.0000) -- (-0.5774,-2.0000) -- (0.5774,-2.0000) -- cycle;
\draw[unet] (0.0000,-1.0000) -- (-0.5774,-2.0000) -- (0.5774,-2.0000) -- cycle;
\draw[uout] (-1.1547,1.0000) -- (-1.7321,0.0000) -- (-1.1547,-1.0000) -- (-0.5774,-2.0000) -- (0.5774,-2.0000) -- (1.1547,-1.0000) -- (1.7321,0.0000) -- (1.1547,1.0000) -- (0.5774,2.0000) -- (-0.5774,2.0000) -- cycle;
\draw[ucut] (0.5774,0.0000) -- (1.7321,0.0000);
\draw[ucut] (0.0000,1.0000) -- (-0.5774,2.0000);
\draw[ucut] (-1.7321,0.0000) -- (-0.5774,0.0000);
\draw[ucut] (0.0000,-1.0000) -- (0.5774,-2.0000);
\draw[uloop1] (1.1547,-1.0000) -- (0.0000,-1.0000) -- (-0.5774,0.0000) -- (-1.1547,1.0000);
\draw[uloop2] (1.1547,1.0000) -- (0.5774,0.0000) -- (0.0000,-1.0000) -- (-1.1547,-1.0000);
\end{scope}
\node[unum] at (0.0000,0.3333) {0};
\node[unum] at (0.0000,-0.3333) {1};
\node[unum] at (-0.5774,0.6667) {3};
\node[unum] at (-0.5774,-0.6667) {4};
\node[unum] at (-1.1547,0.3333) {9};
\node[unum] at (-1.1547,-0.3333) {10};
\node[unum] at (0.5774,0.6667) {2};
\node[unum] at (0.5774,-0.6667) {5};
\node[unum] at (1.1547,-0.3333) {13};
\node[unum] at (1.1547,0.3333) {6};
\node[unum] at (-0.5774,1.3333) {8};
\node[unum] at (0.5774,1.3333) {7};
\node[unum] at (0.0000,1.6667) {14};
\node[unum] at (0.5774,-1.3333) {12};
\node[unum] at (-0.5774,-1.3333) {11};
\node[unum] at (0.0000,-1.6667) {15};
\node[ulab] at (-1.7321,-2.0000) {$5$};
\node[ulab] at (-0.5774,-2.0000) {$7$};
\node[ulab] at (0.5774,-2.0000) {$1$};
\node[ulab] at (1.7321,-2.0000) {$3$};
\node[ulab] at (-2.3094,-1.0000) {$6$};
\node[ulab] at (-1.1547,-1.0000) {$0$};
\node[ulab] at (0.0000,-1.0000) {$2$};
\node[ulab] at (1.1547,-1.0000) {$4$};
\node[ulab] at (2.3094,-1.0000) {$6$};
\node[ulab] at (-1.7321,0.0000) {$1$};
\node[ulab] at (-0.5774,0.0000) {$3$};
\node[ulab] at (0.5774,0.0000) {$5$};
\node[ulab] at (1.7321,0.0000) {$7$};
\node[ulab] at (-2.3094,1.0000) {$2$};
\node[ulab] at (-1.1547,1.0000) {$4$};
\node[ulab] at (0.0000,1.0000) {$6$};
\node[ulab] at (1.1547,1.0000) {$0$};
\node[ulab] at (2.3094,1.0000) {$2$};
\node[ulab] at (-1.7321,2.0000) {$5$};
\node[ulab] at (-0.5774,2.0000) {$7$};
\node[ulab] at (0.5774,2.0000) {$1$};
\node[ulab] at (1.7321,2.0000) {$3$};

%% file: figs/universal-hex.tex
\providecolor{figblue}{RGB}{40,90,160}
\providecolor{figred}{RGB}{170,60,50}
\tikzset{
  ucover/.style={gray!45, line width=0.2pt},
  unet/.style={gray!70, line width=0.3pt},
  uout/.style={black, line width=0.9pt, line join=round},
  ucut/.style={black, line width=0.9pt, dash pattern=on 2.4pt off 1.8pt},
  ulab/.style={font=\tiny, inner sep=0.6pt, fill=white,
               fill opacity=0.72, text opacity=1},
  unum/.style={font=\tiny, text=black!55, inner sep=0.5pt},
  uloop1/.style={figblue, line width=1.4pt, opacity=0.75,
                 line cap=round, line join=round},
  uloop2/.style={figred, line width=1.4pt, opacity=0.75,
                 line cap=round, line join=round},
}
\path[use as bounding box] (-3.0000,-3.0000) rectangle (3.0000,3.0000);
\begin{scope}
\clip (-3.0000,-3.0000) rectangle (3.0000,3.0000);
\draw[ucover] (-3.9082,-0.1538) -- (-2.9756,-0.3333) -- (-3.7700,0.5641) -- cycle;
\draw[ucover] (-3.9082,-0.1538) -- (-4.0464,-0.8718) -- (-2.9756,-0.3333) -- cycle;
\draw[ucover] (-2.9756,-0.3333) -- (-2.6055,0.8718) -- (-3.7700,0.5641) -- cycle;
\draw[ucover] (-2.9756,-0.3333) -- (-4.0464,-0.8718) -- (-3.0792,-1.5897) -- cycle;
\draw[ucover] (-2.9756,-0.3333) -- (-1.5248,0.0256) -- (-2.6055,0.8718) -- cycle;
\draw[ucover] (-2.9756,-0.3333) -- (-1.8949,-1.1795) -- (-1.5248,0.0256) -- cycle;
\draw[ucover] (-2.9756,-0.3333) -- (-3.0792,-1.5897) -- (-1.8949,-1.1795) -- cycle;
\draw[ucover] (-3.7700,0.5641) -- (-2.6055,0.8718) -- (-3.5529,1.6923) -- cycle;
\draw[ucover] (-2.3884,2.0000) -- (-3.5529,1.6923) -- (-2.6055,0.8718) -- cycle;
\draw[ucover] (-1.5248,0.0256) -- (-1.4212,1.2821) -- (-2.6055,0.8718) -- cycle;
\draw[ucover] (-2.2502,2.7179) -- (-1.3175,2.5385) -- (-2.1120,3.4359) -- cycle;
\draw[ucover] (-2.2502,2.7179) -- (-2.3884,2.0000) -- (-1.3175,2.5385) -- cycle;
\draw[ucover] (-2.2502,2.7179) -- (-2.1120,3.4359) -- (-3.1828,2.8974) -- cycle;
\draw[ucover] (-2.2502,2.7179) -- (-3.1828,2.8974) -- (-2.3884,2.0000) -- cycle;
\draw[ucover] (-1.3175,2.5385) -- (-0.9474,3.7436) -- (-2.1120,3.4359) -- cycle;
\draw[ucover] (-1.3175,2.5385) -- (-2.3884,2.0000) -- (-1.4212,1.2821) -- cycle;
\draw[ucover] (-1.3175,2.5385) -- (0.1332,2.8974) -- (-0.9474,3.7436) -- cycle;
\draw[ucover] (-1.3175,2.5385) -- (-0.2369,1.6923) -- (0.1332,2.8974) -- cycle;
\draw[ucover] (-1.3175,2.5385) -- (-1.4212,1.2821) -- (-0.2369,1.6923) -- cycle;
\draw[ucover] (-2.1120,3.4359) -- (-3.0792,4.1538) -- (-3.1828,2.8974) -- cycle;
\draw[ucover] (-2.3884,2.0000) -- (-3.1828,2.8974) -- (-3.5529,1.6923) -- cycle;
\draw[ucover] (-2.3884,2.0000) -- (-2.6055,0.8718) -- (-1.4212,1.2821) -- cycle;
\draw[ucover] (0.1332,2.8974) -- (0.2369,4.1538) -- (-0.9474,3.7436) -- cycle;
\draw[ucover] (-2.2502,-3.0256) -- (-1.3175,-3.2051) -- (-2.1120,-2.3077) -- cycle;
\draw[ucover] (-2.2502,-3.0256) -- (-2.1120,-2.3077) -- (-3.1828,-2.8462) -- cycle;
\draw[ucover] (-2.2502,-3.0256) -- (-3.1828,-2.8462) -- (-2.3884,-3.7436) -- cycle;
\draw[ucover] (-1.3175,-3.2051) -- (-0.9474,-2.0000) -- (-2.1120,-2.3077) -- cycle;
\draw[ucover] (-1.3175,-3.2051) -- (0.1332,-2.8462) -- (-0.9474,-2.0000) -- cycle;
\draw[ucover] (-1.3175,-3.2051) -- (-0.2369,-4.0513) -- (0.1332,-2.8462) -- cycle;
\draw[ucover] (-2.1120,-2.3077) -- (-3.0792,-1.5897) -- (-3.1828,-2.8462) -- cycle;
\draw[ucover] (-2.1120,-2.3077) -- (-0.9474,-2.0000) -- (-1.8949,-1.1795) -- cycle;
\draw[ucover] (-2.1120,-2.3077) -- (-1.8949,-1.1795) -- (-3.0792,-1.5897) -- cycle;
\draw[ucover] (-2.3884,-3.7436) -- (-3.1828,-2.8462) -- (-3.5529,-4.0513) -- cycle;
\draw[ucover] (-0.7303,-0.8718) -- (-1.8949,-1.1795) -- (-0.9474,-2.0000) -- cycle;
\draw[ucover] (0.1332,-2.8462) -- (0.2369,-1.5897) -- (-0.9474,-2.0000) -- cycle;
\draw[ucover] (-0.5922,-0.1538) -- (0.3405,-0.3333) -- (-0.4540,0.5641) -- cycle;
\draw[ucover] (-0.5922,-0.1538) -- (-0.7303,-0.8718) -- (0.3405,-0.3333) -- cycle;
\draw[ucover] (-0.5922,-0.1538) -- (-0.4540,0.5641) -- (-1.5248,0.0256) -- cycle;
\draw[ucover] (-0.5922,-0.1538) -- (-1.5248,0.0256) -- (-0.7303,-0.8718) -- cycle;
\draw[ucover] (0.3405,-0.3333) -- (0.7106,0.8718) -- (-0.4540,0.5641) -- cycle;
\draw[ucover] (0.3405,-0.3333) -- (-0.7303,-0.8718) -- (0.2369,-1.5897) -- cycle;
\draw[ucover] (0.3405,-0.3333) -- (1.7913,0.0256) -- (0.7106,0.8718) -- cycle;
\draw[ucover] (0.3405,-0.3333) -- (1.4212,-1.1795) -- (1.7913,0.0256) -- cycle;
\draw[ucover] (0.3405,-0.3333) -- (0.2369,-1.5897) -- (1.4212,-1.1795) -- cycle;
\draw[ucover] (-0.4540,0.5641) -- (-1.4212,1.2821) -- (-1.5248,0.0256) -- cycle;
\draw[ucover] (-0.4540,0.5641) -- (0.7106,0.8718) -- (-0.2369,1.6923) -- cycle;
\draw[ucover] (-0.4540,0.5641) -- (-0.2369,1.6923) -- (-1.4212,1.2821) -- cycle;
\draw[ucover] (-0.7303,-0.8718) -- (-1.5248,0.0256) -- (-1.8949,-1.1795) -- cycle;
\draw[ucover] (-0.7303,-0.8718) -- (-0.9474,-2.0000) -- (0.2369,-1.5897) -- cycle;
\draw[ucover] (0.9277,2.0000) -- (-0.2369,1.6923) -- (0.7106,0.8718) -- cycle;
\draw[ucover] (1.7913,0.0256) -- (1.8949,1.2821) -- (0.7106,0.8718) -- cycle;
\draw[ucover] (1.0659,2.7179) -- (1.9985,2.5385) -- (1.2040,3.4359) -- cycle;
\draw[ucover] (1.0659,2.7179) -- (0.9277,2.0000) -- (1.9985,2.5385) -- cycle;
\draw[ucover] (1.0659,2.7179) -- (1.2040,3.4359) -- (0.1332,2.8974) -- cycle;
\draw[ucover] (1.0659,2.7179) -- (0.1332,2.8974) -- (0.9277,2.0000) -- cycle;
\draw[ucover] (1.9985,2.5385) -- (2.3686,3.7436) -- (1.2040,3.4359) -- cycle;
\draw[ucover] (1.9985,2.5385) -- (0.9277,2.0000) -- (1.8949,1.2821) -- cycle;
\draw[ucover] (1.9985,2.5385) -- (3.4493,2.8974) -- (2.3686,3.7436) -- cycle;
\draw[ucover] (1.9985,2.5385) -- (3.0792,1.6923) -- (3.4493,2.8974) -- cycle;
\draw[ucover] (1.9985,2.5385) -- (1.8949,1.2821) -- (3.0792,1.6923) -- cycle;
\draw[ucover] (1.2040,3.4359) -- (0.2369,4.1538) -- (0.1332,2.8974) -- cycle;
\draw[ucover] (0.9277,2.0000) -- (0.1332,2.8974) -- (-0.2369,1.6923) -- cycle;
\draw[ucover] (0.9277,2.0000) -- (0.7106,0.8718) -- (1.8949,1.2821) -- cycle;
\draw[ucover] (3.4493,2.8974) -- (3.5529,4.1538) -- (2.3686,3.7436) -- cycle;
\draw[ucover] (1.0659,-3.0256) -- (1.9985,-3.2051) -- (1.2040,-2.3077) -- cycle;
\draw[ucover] (1.0659,-3.0256) -- (1.2040,-2.3077) -- (0.1332,-2.8462) -- cycle;
\draw[ucover] (1.0659,-3.0256) -- (0.1332,-2.8462) -- (0.9277,-3.7436) -- cycle;
\draw[ucover] (1.9985,-3.2051) -- (2.3686,-2.0000) -- (1.2040,-2.3077) -- cycle;
\draw[ucover] (1.9985,-3.2051) -- (3.4493,-2.8462) -- (2.3686,-2.0000) -- cycle;
\draw[ucover] (1.9985,-3.2051) -- (3.0792,-4.0513) -- (3.4493,-2.8462) -- cycle;
\draw[ucover] (1.2040,-2.3077) -- (0.2369,-1.5897) -- (0.1332,-2.8462) -- cycle;
\draw[ucover] (1.2040,-2.3077) -- (2.3686,-2.0000) -- (1.4212,-1.1795) -- cycle;
\draw[ucover] (1.2040,-2.3077) -- (1.4212,-1.1795) -- (0.2369,-1.5897) -- cycle;
\draw[ucover] (0.9277,-3.7436) -- (0.1332,-2.8462) -- (-0.2369,-4.0513) -- cycle;
\draw[ucover] (2.5857,-0.8718) -- (1.4212,-1.1795) -- (2.3686,-2.0000) -- cycle;
\draw[ucover] (3.4493,-2.8462) -- (3.5529,-1.5897) -- (2.3686,-2.0000) -- cycle;
\draw[ucover] (2.7239,-0.1538) -- (3.6566,-0.3333) -- (2.8621,0.5641) -- cycle;
\draw[ucover] (2.7239,-0.1538) -- (2.5857,-0.8718) -- (3.6566,-0.3333) -- cycle;
\draw[ucover] (2.7239,-0.1538) -- (2.8621,0.5641) -- (1.7913,0.0256) -- cycle;
\draw[ucover] (2.7239,-0.1538) -- (1.7913,0.0256) -- (2.5857,-0.8718) -- cycle;
\draw[ucover] (3.6566,-0.3333) -- (4.0266,0.8718) -- (2.8621,0.5641) -- cycle;
\draw[ucover] (3.6566,-0.3333) -- (2.5857,-0.8718) -- (3.5529,-1.5897) -- cycle;
\draw[ucover] (2.8621,0.5641) -- (1.8949,1.2821) -- (1.7913,0.0256) -- cycle;
\draw[ucover] (2.8621,0.5641) -- (4.0266,0.8718) -- (3.0792,1.6923) -- cycle;
\draw[ucover] (2.8621,0.5641) -- (3.0792,1.6923) -- (1.8949,1.2821) -- cycle;
\draw[ucover] (2.5857,-0.8718) -- (1.7913,0.0256) -- (1.4212,-1.1795) -- cycle;
\draw[ucover] (2.5857,-0.8718) -- (2.3686,-2.0000) -- (3.5529,-1.5897) -- cycle;
\fill[figblue, fill opacity=0.09] (-0.5922,-0.1538) -- (0.3405,-0.3333) -- (-0.4540,0.5641) -- cycle;
\draw[unet] (-0.5922,-0.1538) -- (0.3405,-0.3333) -- (-0.4540,0.5641) -- cycle;
\fill[figblue, fill opacity=0.09] (-0.5922,-0.1538) -- (-0.7303,-0.8718) -- (0.3405,-0.3333) -- cycle;
\draw[unet] (-0.5922,-0.1538) -- (-0.7303,-0.8718) -- (0.3405,-0.3333) -- cycle;
\fill[figblue, fill opacity=0.09] (-0.5922,-0.1538) -- (-0.4540,0.5641) -- (-1.5248,0.0256) -- cycle;
\draw[unet] (-0.5922,-0.1538) -- (-0.4540,0.5641) -- (-1.5248,0.0256) -- cycle;
\fill[figblue, fill opacity=0.09] (-0.5922,-0.1538) -- (-1.5248,0.0256) -- (-0.7303,-0.8718) -- cycle;
\draw[unet] (-0.5922,-0.1538) -- (-1.5248,0.0256) -- (-0.7303,-0.8718) -- cycle;
\fill[figblue, fill opacity=0.09] (0.3405,-0.3333) -- (0.7106,0.8718) -- (-0.4540,0.5641) -- cycle;
\draw[unet] (0.3405,-0.3333) -- (0.7106,0.8718) -- (-0.4540,0.5641) -- cycle;
\fill[figblue, fill opacity=0.09] (0.3405,-0.3333) -- (-0.7303,-0.8718) -- (0.2369,-1.5897) -- cycle;
\draw[unet] (0.3405,-0.3333) -- (-0.7303,-0.8718) -- (0.2369,-1.5897) -- cycle;
\fill[figblue, fill opacity=0.09] (0.3405,-0.3333) -- (1.7913,0.0256) -- (0.7106,0.8718) -- cycle;
\draw[unet] (0.3405,-0.3333) -- (1.7913,0.0256) -- (0.7106,0.8718) -- cycle;
\fill[figblue, fill opacity=0.09] (0.3405,-0.3333) -- (1.4212,-1.1795) -- (1.7913,0.0256) -- cycle;
\draw[unet] (0.3405,-0.3333) -- (1.4212,-1.1795) -- (1.7913,0.0256) -- cycle;
\fill[figblue, fill opacity=0.09] (0.3405,-0.3333) -- (0.2369,-1.5897) -- (1.4212,-1.1795) -- cycle;
\draw[unet] (0.3405,-0.3333) -- (0.2369,-1.5897) -- (1.4212,-1.1795) -- cycle;
\fill[figblue, fill opacity=0.09] (-0.4540,0.5641) -- (-1.4212,1.2821) -- (-1.5248,0.0256) -- cycle;
\draw[unet] (-0.4540,0.5641) -- (-1.4212,1.2821) -- (-1.5248,0.0256) -- cycle;
\fill[figblue, fill opacity=0.09] (-0.4540,0.5641) -- (0.7106,0.8718) -- (-0.2369,1.6923) -- cycle;
\draw[unet] (-0.4540,0.5641) -- (0.7106,0.8718) -- (-0.2369,1.6923) -- cycle;
\fill[figblue, fill opacity=0.09] (-0.4540,0.5641) -- (-0.2369,1.6923) -- (-1.4212,1.2821) -- cycle;
\draw[unet] (-0.4540,0.5641) -- (-0.2369,1.6923) -- (-1.4212,1.2821) -- cycle;
\fill[figblue, fill opacity=0.09] (-0.7303,-0.8718) -- (-1.5248,0.0256) -- (-1.8949,-1.1795) -- cycle;
\draw[unet] (-0.7303,-0.8718) -- (-1.5248,0.0256) -- (-1.8949,-1.1795) -- cycle;
\fill[figblue, fill opacity=0.09] (-0.7303,-0.8718) -- (-0.9474,-2.0000) -- (0.2369,-1.5897) -- cycle;
\draw[unet] (-0.7303,-0.8718) -- (-0.9474,-2.0000) -- (0.2369,-1.5897) -- cycle;
\fill[figblue, fill opacity=0.09] (0.9277,2.0000) -- (-0.2369,1.6923) -- (0.7106,0.8718) -- cycle;
\draw[unet] (0.9277,2.0000) -- (-0.2369,1.6923) -- (0.7106,0.8718) -- cycle;
\fill[figblue, fill opacity=0.09] (1.7913,0.0256) -- (1.8949,1.2821) -- (0.7106,0.8718) -- cycle;
\draw[unet] (1.7913,0.0256) -- (1.8949,1.2821) -- (0.7106,0.8718) -- cycle;
\draw[uout] (1.4212,-1.1795) -- (1.7913,0.0256) -- (1.8949,1.2821) -- (0.7106,0.8718) -- (0.9277,2.0000) -- (-0.2369,1.6923) -- (-1.4212,1.2821) -- (-1.5248,0.0256) -- (-1.8949,-1.1795) -- (-0.7303,-0.8718) -- (-0.9474,-2.0000) -- (0.2369,-1.5897) -- cycle;
\draw[ucut] (-1.5248,0.0256) -- (-0.5922,-0.1538);
\draw[ucut] (-0.4540,0.5641) -- (-1.4212,1.2821);
\draw[ucut] (1.4212,-1.1795) -- (0.3405,-0.3333);
\draw[uloop1] (-1.5248,0.0256) -- (-0.5922,-0.1538) -- (0.3405,-0.3333) -- (1.7913,0.0256);
\draw[uloop2] (-0.9474,-2.0000) -- (0.2369,-1.5897) -- (0.3405,-0.3333) -- (0.7106,0.8718);
\end{scope}
\node[unum] at (-0.2352,0.0256) {0};
\node[unum] at (-0.3388,-0.5128) {1};
\node[unum] at (-0.8455,0.2051) {3};
\node[unum] at (-0.9491,-0.3333) {4};
\node[unum] at (0.1990,0.3675) {2};
\node[unum] at (-0.0510,-0.9316) {5};
\node[unum] at (0.9474,0.1880) {6};
\node[unum] at (1.1843,-0.4957) {12};
\node[unum] at (0.6662,-1.0342) {11};
\node[unum] at (-1.1333,0.6239) {8};
\node[unum] at (0.0066,1.0427) {7};
\node[unum] at (-0.7040,1.1795) {15};
\node[unum] at (-1.4164,-0.6378) {9};
\node[unum] at (-0.4803,-1.4872) {10};
\node[unum] at (0.4671,1.5214) {14};
\node[unum] at (1.4656,0.7265) {13};
\node[ulab] at (0.1332,-2.8462) {$4$};
\node[ulab] at (-2.1120,-2.3077) {$2$};
\node[ulab] at (1.2040,-2.3077) {$2$};
\node[ulab] at (-0.9474,-2.0000) {$5$};
\node[ulab] at (2.3686,-2.0000) {$5$};
\node[ulab] at (0.2369,-1.5897) {$6$};
\node[ulab] at (-1.8949,-1.1795) {$7$};
\node[ulab] at (1.4212,-1.1795) {$7$};
\node[ulab] at (-0.7303,-0.8718) {$3$};
\node[ulab] at (2.5857,-0.8718) {$3$};
\node[ulab] at (0.3405,-0.3333) {$1$};
\node[ulab] at (-0.5922,-0.1538) {$0$};
\node[ulab] at (2.7239,-0.1538) {$0$};
\node[ulab] at (-1.5248,0.0256) {$4$};
\node[ulab] at (1.7913,0.0256) {$4$};
\node[ulab] at (-0.4540,0.5641) {$2$};
\node[ulab] at (2.8621,0.5641) {$2$};
\node[ulab] at (-2.6055,0.8718) {$5$};
\node[ulab] at (0.7106,0.8718) {$5$};
\node[ulab] at (-1.4212,1.2821) {$6$};
\node[ulab] at (1.8949,1.2821) {$6$};
\node[ulab] at (-0.2369,1.6923) {$7$};
\node[ulab] at (-2.3884,2.0000) {$3$};
\node[ulab] at (0.9277,2.0000) {$3$};
\node[ulab] at (-1.3175,2.5385) {$1$};
\node[ulab] at (1.9985,2.5385) {$1$};
\node[ulab] at (-2.2502,2.7179) {$0$};
\node[ulab] at (1.0659,2.7179) {$0$};

%% file: figs/universal-square-wide.tex
\providecolor{figblue}{RGB}{40,90,160}
\providecolor{figred}{RGB}{170,60,50}
\tikzset{
  ucover/.style={gray!45, line width=0.2pt},
  unet/.style={gray!70, line width=0.3pt},
  uout/.style={black, line width=0.9pt, line join=round},
  ucut/.style={black, line width=0.9pt, dash pattern=on 2.4pt off 1.8pt},
  ulab/.style={font=\tiny, inner sep=0.6pt, fill=white,
               fill opacity=0.72, text opacity=1},
  unum/.style={font=\tiny, text=black!55, inner sep=0.5pt},
  uloop1/.style={figblue, line width=1.4pt, opacity=0.75,
                 line cap=round, line join=round},
  uloop2/.style={figred, line width=1.4pt, opacity=0.75,
                 line cap=round, line join=round},
}
\path[use as bounding box] (-3.0000,-3.0000) rectangle (3.0000,3.0000);
\begin{scope}
\clip (-3.0000,-3.0000) rectangle (3.0000,3.0000);
\draw[ucover, opacity=1.000] (0.5774,0.0000) -- (-0.5774,0.0000);
\draw[ucover, opacity=1.000] (0.0000,1.0000) -- (0.5774,0.0000);
\draw[ucover, opacity=1.000] (-0.5774,0.0000) -- (0.0000,1.0000);
\draw[ucover, opacity=1.000] (0.0000,-1.0000) -- (-0.5774,0.0000);
\draw[ucover, opacity=1.000] (0.5774,0.0000) -- (0.0000,-1.0000);
\draw[ucover, opacity=1.000] (1.1547,1.0000) -- (0.5774,0.0000);
\draw[ucover, opacity=1.000] (0.0000,1.0000) -- (1.1547,1.0000);
\draw[ucover, opacity=1.000] (-1.1547,1.0000) -- (0.0000,1.0000);
\draw[ucover, opacity=1.000] (-0.5774,0.0000) -- (-1.1547,1.0000);
\draw[ucover, opacity=1.000] (-1.1547,-1.0000) -- (-0.5774,0.0000);
\draw[ucover, opacity=1.000] (0.0000,-1.0000) -- (-1.1547,-1.0000);
\draw[ucover, opacity=1.000] (1.1547,-1.0000) -- (0.0000,-1.0000);
\draw[ucover, opacity=1.000] (0.5774,0.0000) -- (1.1547,-1.0000);
\draw[ucover, opacity=1.000] (1.7321,0.0000) -- (0.5774,0.0000);
\draw[ucover, opacity=0.898] (1.1547,1.0000) -- (1.7321,0.0000);
\draw[ucover, opacity=0.867] (0.5774,2.0000) -- (1.1547,1.0000);
\draw[ucover, opacity=0.867] (0.0000,1.0000) -- (0.5774,2.0000);
\draw[ucover, opacity=0.867] (-0.5774,2.0000) -- (0.0000,1.0000);
\draw[ucover, opacity=0.867] (-1.1547,1.0000) -- (-0.5774,2.0000);
\draw[ucover, opacity=0.898] (-1.7321,0.0000) -- (-1.1547,1.0000);
\draw[ucover, opacity=1.000] (-0.5774,0.0000) -- (-1.7321,0.0000);
\draw[ucover, opacity=0.898] (-1.1547,-1.0000) -- (-1.7321,0.0000);
\draw[ucover, opacity=0.867] (-0.5774,-2.0000) -- (-1.1547,-1.0000);
\draw[ucover, opacity=0.867] (0.0000,-1.0000) -- (-0.5774,-2.0000);
\draw[ucover, opacity=0.867] (0.5774,-2.0000) -- (0.0000,-1.0000);
\draw[ucover, opacity=0.867] (1.1547,-1.0000) -- (0.5774,-2.0000);
\draw[ucover, opacity=0.898] (1.7321,0.0000) -- (1.1547,-1.0000);
\draw[ucover, opacity=0.577] (2.3094,1.0000) -- (1.7321,0.0000);
\draw[ucover, opacity=0.738] (1.1547,1.0000) -- (2.3094,1.0000);
\draw[ucover, opacity=0.867] (1.7321,2.0000) -- (1.1547,1.0000);
\draw[ucover, opacity=0.589] (0.5774,2.0000) -- (1.7321,2.0000);
\draw[ucover, opacity=0.589] (-0.5774,2.0000) -- (0.5774,2.0000);
\draw[ucover, opacity=0.589] (-1.7321,2.0000) -- (-0.5774,2.0000);
\draw[ucover, opacity=0.867] (-1.1547,1.0000) -- (-1.7321,2.0000);
\draw[ucover, opacity=0.738] (-2.3094,1.0000) -- (-1.1547,1.0000);
\draw[ucover, opacity=0.577] (-1.7321,0.0000) -- (-2.3094,1.0000);
\draw[ucover, opacity=0.577] (-2.3094,-1.0000) -- (-1.7321,0.0000);
\draw[ucover, opacity=0.738] (-1.1547,-1.0000) -- (-2.3094,-1.0000);
\draw[ucover, opacity=0.867] (-1.7321,-2.0000) -- (-1.1547,-1.0000);
\draw[ucover, opacity=0.589] (-0.5774,-2.0000) -- (-1.7321,-2.0000);
\draw[ucover, opacity=0.589] (0.5774,-2.0000) -- (-0.5774,-2.0000);
\draw[ucover, opacity=0.589] (1.7321,-2.0000) -- (0.5774,-2.0000);
\draw[ucover, opacity=0.867] (1.1547,-1.0000) -- (1.7321,-2.0000);
\draw[ucover, opacity=0.738] (2.3094,-1.0000) -- (1.1547,-1.0000);
\draw[ucover, opacity=0.577] (1.7321,0.0000) -- (2.3094,-1.0000);
\draw[ucover, opacity=0.417] (2.8868,0.0000) -- (1.7321,0.0000);
\draw[ucover, opacity=0.257] (2.3094,1.0000) -- (2.8868,0.0000);
\draw[ucover, opacity=0.577] (1.7321,2.0000) -- (2.3094,1.0000);
\draw[ucover, opacity=0.311] (1.1547,3.0000) -- (1.7321,2.0000);
\draw[ucover, opacity=0.311] (0.5774,2.0000) -- (1.1547,3.0000);
\draw[ucover, opacity=0.311] (0.0000,3.0000) -- (0.5774,2.0000);
\draw[ucover, opacity=0.311] (-0.5774,2.0000) -- (0.0000,3.0000);
\draw[ucover, opacity=0.311] (-1.1547,3.0000) -- (-0.5774,2.0000);
\draw[ucover, opacity=0.311] (-1.7321,2.0000) -- (-1.1547,3.0000);
\draw[ucover, opacity=0.577] (-2.3094,1.0000) -- (-1.7321,2.0000);
\draw[ucover, opacity=0.257] (-2.8868,0.0000) -- (-2.3094,1.0000);
\draw[ucover, opacity=0.417] (-1.7321,0.0000) -- (-2.8868,0.0000);
\draw[ucover, opacity=0.257] (-2.3094,-1.0000) -- (-2.8868,0.0000);
\draw[ucover, opacity=0.577] (-1.7321,-2.0000) -- (-2.3094,-1.0000);
\draw[ucover, opacity=0.311] (-1.1547,-3.0000) -- (-1.7321,-2.0000);
\draw[ucover, opacity=0.311] (-0.5774,-2.0000) -- (-1.1547,-3.0000);
\draw[ucover, opacity=0.311] (0.0000,-3.0000) -- (-0.5774,-2.0000);
\draw[ucover, opacity=0.311] (0.5774,-2.0000) -- (0.0000,-3.0000);
\draw[ucover, opacity=0.311] (1.1547,-3.0000) -- (0.5774,-2.0000);
\draw[ucover, opacity=0.311] (1.7321,-2.0000) -- (1.1547,-3.0000);
\draw[ucover, opacity=0.577] (2.3094,-1.0000) -- (1.7321,-2.0000);
\draw[ucover, opacity=0.257] (2.8868,0.0000) -- (2.3094,-1.0000);
\draw[ucover, opacity=0.096] (2.3094,1.0000) -- (3.4641,1.0000);
\draw[ucover, opacity=0.257] (2.8868,2.0000) -- (2.3094,1.0000);
\draw[ucover, opacity=0.417] (1.7321,2.0000) -- (2.8868,2.0000);
\draw[ucover, opacity=0.311] (2.3094,3.0000) -- (1.7321,2.0000);
\draw[ucover, opacity=0.033] (1.1547,3.0000) -- (2.3094,3.0000);
\draw[ucover, opacity=0.033] (0.0000,3.0000) -- (1.1547,3.0000);
\draw[ucover, opacity=0.033] (-1.1547,3.0000) -- (0.0000,3.0000);
\draw[ucover, opacity=0.033] (-2.3094,3.0000) -- (-1.1547,3.0000);
\draw[ucover, opacity=0.311] (-1.7321,2.0000) -- (-2.3094,3.0000);
\draw[ucover, opacity=0.417] (-2.8868,2.0000) -- (-1.7321,2.0000);
\draw[ucover, opacity=0.257] (-2.3094,1.0000) -- (-2.8868,2.0000);
\draw[ucover, opacity=0.096] (-3.4641,1.0000) -- (-2.3094,1.0000);
\draw[ucover, opacity=0.096] (-2.3094,-1.0000) -- (-3.4641,-1.0000);
\draw[ucover, opacity=0.257] (-2.8868,-2.0000) -- (-2.3094,-1.0000);
\draw[ucover, opacity=0.417] (-1.7321,-2.0000) -- (-2.8868,-2.0000);
\draw[ucover, opacity=0.311] (-2.3094,-3.0000) -- (-1.7321,-2.0000);
\draw[ucover, opacity=0.033] (-1.1547,-3.0000) -- (-2.3094,-3.0000);
\draw[ucover, opacity=0.033] (0.0000,-3.0000) -- (-1.1547,-3.0000);
\draw[ucover, opacity=0.033] (1.1547,-3.0000) -- (0.0000,-3.0000);
\draw[ucover, opacity=0.033] (2.3094,-3.0000) -- (1.1547,-3.0000);
\draw[ucover, opacity=0.311] (1.7321,-2.0000) -- (2.3094,-3.0000);
\draw[ucover, opacity=0.417] (2.8868,-2.0000) -- (1.7321,-2.0000);
\draw[ucover, opacity=0.257] (2.3094,-1.0000) -- (2.8868,-2.0000);
\draw[ucover, opacity=0.096] (3.4641,-1.0000) -- (2.3094,-1.0000);
\draw[ucover, opacity=0.257] (2.3094,3.0000) -- (2.8868,2.0000);
\draw[ucover, opacity=0.257] (-2.8868,2.0000) -- (-2.3094,3.0000);
\draw[ucover, opacity=0.257] (-2.3094,-3.0000) -- (-2.8868,-2.0000);
\draw[ucover, opacity=0.257] (2.8868,-2.0000) -- (2.3094,-3.0000);
\draw[ucover, opacity=0.033] (2.3094,3.0000) -- (3.4641,3.0000);
\draw[ucover, opacity=0.033] (-3.4641,3.0000) -- (-2.3094,3.0000);
\draw[ucover, opacity=0.033] (-2.3094,-3.0000) -- (-3.4641,-3.0000);
\draw[ucover, opacity=0.033] (3.4641,-3.0000) -- (2.3094,-3.0000);
\fill[figblue, fill opacity=0.09] (-0.5774,0.0000) -- (0.5774,0.0000) -- (0.0000,1.0000) -- cycle;
\draw[unet] (-0.5774,0.0000) -- (0.5774,0.0000) -- (0.0000,1.0000) -- cycle;
\fill[figblue, fill opacity=0.09] (-0.5774,0.0000) -- (0.0000,-1.0000) -- (0.5774,0.0000) -- cycle;
\draw[unet] (-0.5774,0.0000) -- (0.0000,-1.0000) -- (0.5774,0.0000) -- cycle;
\fill[figblue, fill opacity=0.09] (-0.5774,0.0000) -- (0.0000,1.0000) -- (-1.1547,1.0000) -- cycle;
\draw[unet] (-0.5774,0.0000) -- (0.0000,1.0000) -- (-1.1547,1.0000) -- cycle;
\fill[figblue, fill opacity=0.09] (-0.5774,0.0000) -- (-1.1547,-1.0000) -- (0.0000,-1.0000) -- cycle;
\draw[unet] (-0.5774,0.0000) -- (-1.1547,-1.0000) -- (0.0000,-1.0000) -- cycle;
\fill[figblue, fill opacity=0.09] (-0.5774,0.0000) -- (-1.1547,1.0000) -- (-1.7321,0.0000) -- cycle;
\draw[unet] (-0.5774,0.0000) -- (-1.1547,1.0000) -- (-1.7321,0.0000) -- cycle;
\fill[figblue, fill opacity=0.09] (-0.5774,0.0000) -- (-1.7321,0.0000) -- (-1.1547,-1.0000) -- cycle;
\draw[unet] (-0.5774,0.0000) -- (-1.7321,0.0000) -- (-1.1547,-1.0000) -- cycle;
\fill[figblue, fill opacity=0.09] (0.5774,0.0000) -- (1.1547,1.0000) -- (0.0000,1.0000) -- cycle;
\draw[unet] (0.5774,0.0000) -- (1.1547,1.0000) -- (0.0000,1.0000) -- cycle;
\fill[figblue, fill opacity=0.09] (0.5774,0.0000) -- (0.0000,-1.0000) -- (1.1547,-1.0000) -- cycle;
\draw[unet] (0.5774,0.0000) -- (0.0000,-1.0000) -- (1.1547,-1.0000) -- cycle;
\fill[figblue, fill opacity=0.09] (0.5774,0.0000) -- (1.1547,-1.0000) -- (1.7321,0.0000) -- cycle;
\draw[unet] (0.5774,0.0000) -- (1.1547,-1.0000) -- (1.7321,0.0000) -- cycle;
\fill[figblue, fill opacity=0.09] (0.5774,0.0000) -- (1.7321,0.0000) -- (1.1547,1.0000) -- cycle;
\draw[unet] (0.5774,0.0000) -- (1.7321,0.0000) -- (1.1547,1.0000) -- cycle;
\fill[figblue, fill opacity=0.09] (0.0000,1.0000) -- (-0.5774,2.0000) -- (-1.1547,1.0000) -- cycle;
\draw[unet] (0.0000,1.0000) -- (-0.5774,2.0000) -- (-1.1547,1.0000) -- cycle;
\fill[figblue, fill opacity=0.09] (0.0000,1.0000) -- (1.1547,1.0000) -- (0.5774,2.0000) -- cycle;
\draw[unet] (0.0000,1.0000) -- (1.1547,1.0000) -- (0.5774,2.0000) -- cycle;
\fill[figblue, fill opacity=0.09] (0.0000,1.0000) -- (0.5774,2.0000) -- (-0.5774,2.0000) -- cycle;
\draw[unet] (0.0000,1.0000) -- (0.5774,2.0000) -- (-0.5774,2.0000) -- cycle;
\fill[figblue, fill opacity=0.09] (0.0000,-1.0000) -- (0.5774,-2.0000) -- (1.1547,-1.0000) -- cycle;
\draw[unet] (0.0000,-1.0000) -- (0.5774,-2.0000) -- (1.1547,-1.0000) -- cycle;
\fill[figblue, fill opacity=0.09] (0.0000,-1.0000) -- (-1.1547,-1.0000) -- (-0.5774,-2.0000) -- cycle;
\draw[unet] (0.0000,-1.0000) -- (-1.1547,-1.0000) -- (-0.5774,-2.0000) -- cycle;
\fill[figblue, fill opacity=0.09] (0.0000,-1.0000) -- (-0.5774,-2.0000) -- (0.5774,-2.0000) -- cycle;
\draw[unet] (0.0000,-1.0000) -- (-0.5774,-2.0000) -- (0.5774,-2.0000) -- cycle;
\draw[uout] (-1.1547,1.0000) -- (-1.7321,0.0000) -- (-1.1547,-1.0000) -- (-0.5774,-2.0000) -- (0.5774,-2.0000) -- (1.1547,-1.0000) -- (1.7321,0.0000) -- (1.1547,1.0000) -- (0.5774,2.0000) -- (-0.5774,2.0000) -- cycle;
\draw[ucut] (0.5774,0.0000) -- (1.7321,0.0000);
\draw[ucut] (0.0000,1.0000) -- (-0.5774,2.0000);
\draw[ucut] (-1.7321,0.0000) -- (-0.5774,0.0000);
\draw[ucut] (0.0000,-1.0000) -- (0.5774,-2.0000);
\draw[uloop1] (1.1547,-1.0000) -- (0.0000,-1.0000) -- (-0.5774,0.0000) -- (-1.1547,1.0000);
\draw[uloop2] (1.1547,1.0000) -- (0.5774,0.0000) -- (0.0000,-1.0000) -- (-1.1547,-1.0000);
\end{scope}
\node[unum] at (0.0000,0.3333) {0};
\node[unum] at (0.0000,-0.3333) {1};
\node[unum] at (-0.5774,0.6667) {3};
\node[unum] at (-0.5774,-0.6667) {4};
\node[unum] at (-1.1547,0.3333) {9};
\node[unum] at (-1.1547,-0.3333) {10};
\node[unum] at (0.5774,0.6667) {2};
\node[unum] at (0.5774,-0.6667) {5};
\node[unum] at (1.1547,-0.3333) {13};
\node[unum] at (1.1547,0.3333) {6};
\node[unum] at (-0.5774,1.3333) {8};
\node[unum] at (0.5774,1.3333) {7};
\node[unum] at (0.0000,1.6667) {14};
\node[unum] at (0.5774,-1.3333) {12};
\node[unum] at (-0.5774,-1.3333) {11};
\node[unum] at (0.0000,-1.6667) {15};
\node[ulab, opacity=0.589] at (-1.7321,-2.0000) {$5$};
\node[ulab, opacity=0.589] at (-0.5774,-2.0000) {$7$};
\node[ulab, opacity=0.589] at (0.5774,-2.0000) {$1$};
\node[ulab, opacity=0.589] at (1.7321,-2.0000) {$3$};
\node[ulab, opacity=0.417] at (-2.3094,-1.0000) {$6$};
\node[ulab, opacity=1.000] at (-1.1547,-1.0000) {$0$};
\node[ulab, opacity=1.000] at (0.0000,-1.0000) {$2$};
\node[ulab, opacity=1.000] at (1.1547,-1.0000) {$4$};
\node[ulab, opacity=0.417] at (2.3094,-1.0000) {$6$};
\node[ulab, opacity=0.738] at (-1.7321,0.0000) {$1$};
\node[ulab, opacity=1.000] at (-0.5774,0.0000) {$3$};
\node[ulab, opacity=1.000] at (0.5774,0.0000) {$5$};
\node[ulab, opacity=0.738] at (1.7321,0.0000) {$7$};
\node[ulab, opacity=0.417] at (-2.3094,1.0000) {$2$};
\node[ulab, opacity=1.000] at (-1.1547,1.0000) {$4$};
\node[ulab, opacity=1.000] at (0.0000,1.0000) {$6$};
\node[ulab, opacity=1.000] at (1.1547,1.0000) {$0$};
\node[ulab, opacity=0.417] at (2.3094,1.0000) {$2$};
\node[ulab, opacity=0.589] at (-1.7321,2.0000) {$5$};
\node[ulab, opacity=0.589] at (-0.5774,2.0000) {$7$};
\node[ulab, opacity=0.589] at (0.5774,2.0000) {$1$};
\node[ulab, opacity=0.589] at (1.7321,2.0000) {$3$};

%% file: figs/net-pushed.tex
\providecolor{figblue}{RGB}{40,90,160}
\providecolor{figred}{RGB}{170,60,50}
\tikzset{
  pcell/.style={gray!70, line width=0.3pt},
  pout/.style={black, line width=0.9pt, line join=round},
  plab/.style={font=\scriptsize, inner sep=0.6pt, fill=white},
  pnum/.style={font=\tiny, text=black!55, inner sep=0.5pt},
  pcal/.style={font=\scriptsize, inner sep=1pt},
  pmkone/.style={figblue, line width=1.2pt, line cap=round},
  pmktwo/.style={figred, line width=1.2pt, line cap=round},
}
\fill[figblue, fill opacity=0.09] (-0.8359,0.1416) -- (1.5079,0.1416) -- (-0.5344,0.9826) -- cycle;
\draw[pcell] (-0.8359,0.1416) -- (1.5079,0.1416) -- (-0.5344,0.9826) -- cycle;
\fill[figblue, fill opacity=0.09] (-0.8359,0.1416) -- (0.3768,-0.8975) -- (1.5079,0.1416) -- cycle;
\draw[pcell] (-0.8359,0.1416) -- (0.3768,-0.8975) -- (1.5079,0.1416) -- cycle;
\fill[figblue, fill opacity=0.09] (-0.8359,0.1416) -- (-0.5344,0.9826) -- (-2.4284,0.3868) -- cycle;
\draw[pcell] (-0.8359,0.1416) -- (-0.5344,0.9826) -- (-2.4284,0.3868) -- cycle;
\fill[figblue, fill opacity=0.09] (-0.8359,0.1416) -- (-0.3414,-0.8081) -- (0.3768,-0.8975) -- cycle;
\draw[pcell] (-0.8359,0.1416) -- (-0.3414,-0.8081) -- (0.3768,-0.8975) -- cycle;
\fill[figblue, fill opacity=0.09] (-0.8359,0.1416) -- (-2.4284,0.3868) -- (-2.2000,-1.1640) -- cycle;
\draw[pcell] (-0.8359,0.1416) -- (-2.4284,0.3868) -- (-2.2000,-1.1640) -- cycle;
\fill[figblue, fill opacity=0.09] (-0.8359,0.1416) -- (-2.0112,-1.3363) -- (-0.3414,-0.8081) -- cycle;
\draw[pcell] (-0.8359,0.1416) -- (-2.0112,-1.3363) -- (-0.3414,-0.8081) -- cycle;
\fill[figblue, fill opacity=0.09] (1.5079,0.1416) -- (0.8623,2.1160) -- (-0.5344,0.9826) -- cycle;
\draw[pcell] (1.5079,0.1416) -- (0.8623,2.1160) -- (-0.5344,0.9826) -- cycle;
\fill[figblue, fill opacity=0.09] (1.5079,0.1416) -- (0.3768,-0.8975) -- (1.2024,-1.7912) -- cycle;
\draw[pcell] (1.5079,0.1416) -- (0.3768,-0.8975) -- (1.2024,-1.7912) -- cycle;
\fill[figblue, fill opacity=0.09] (1.5079,0.1416) -- (1.2024,-1.7912) -- (2.4284,-0.3855) -- cycle;
\draw[pcell] (1.5079,0.1416) -- (1.2024,-1.7912) -- (2.4284,-0.3855) -- cycle;
\fill[figblue, fill opacity=0.09] (1.5079,0.1416) -- (2.3490,0.7879) -- (0.8623,2.1160) -- cycle;
\draw[pcell] (1.5079,0.1416) -- (2.3490,0.7879) -- (0.8623,2.1160) -- cycle;
\fill[figblue, fill opacity=0.09] (-0.5344,0.9826) -- (-1.3945,1.9391) -- (-2.4284,0.3868) -- cycle;
\draw[pcell] (-0.5344,0.9826) -- (-1.3945,1.9391) -- (-2.4284,0.3868) -- cycle;
\fill[figblue, fill opacity=0.09] (-0.5344,0.9826) -- (0.8623,2.1160) -- (-0.7888,2.7000) -- cycle;
\draw[pcell] (-0.5344,0.9826) -- (0.8623,2.1160) -- (-0.7888,2.7000) -- cycle;
\fill[figblue, fill opacity=0.09] (-0.5344,0.9826) -- (-0.7888,2.7000) -- (-1.1333,2.1211) -- cycle;
\draw[pcell] (-0.5344,0.9826) -- (-0.7888,2.7000) -- (-1.1333,2.1211) -- cycle;
\fill[figblue, fill opacity=0.09] (0.3768,-0.8975) -- (-0.3452,-2.0401) -- (1.2024,-1.7912) -- cycle;
\draw[pcell] (0.3768,-0.8975) -- (-0.3452,-2.0401) -- (1.2024,-1.7912) -- cycle;
\fill[figblue, fill opacity=0.09] (0.3768,-0.8975) -- (-0.3414,-0.8081) -- (0.2868,-2.7000) -- cycle;
\draw[pcell] (0.3768,-0.8975) -- (-0.3414,-0.8081) -- (0.2868,-2.7000) -- cycle;
\fill[figblue, fill opacity=0.09] (0.3768,-0.8975) -- (0.2868,-2.7000) -- (0.7381,-2.1999) -- cycle;
\draw[pcell] (0.3768,-0.8975) -- (0.2868,-2.7000) -- (0.7381,-2.1999) -- cycle;
\draw[pout] (-2.4284,0.3868) -- (-2.2000,-1.1640);
\draw[pout] (-2.2000,-1.1640) -- (-0.8359,0.1416);
\draw[pout] (-0.8359,0.1416) -- (-2.0112,-1.3363);
\draw[pout] (-2.0112,-1.3363) -- (-0.3414,-0.8081);
\draw[pout] (1.2024,-1.7912) -- (2.4284,-0.3855);
\draw[pout] (2.4284,-0.3855) -- (1.5079,0.1416);
\draw[pout] (1.5079,0.1416) -- (2.3490,0.7879);
\draw[pout] (2.3490,0.7879) -- (0.8623,2.1160);
\draw[pout] (-0.5344,0.9826) -- (-1.3945,1.9391);
\draw[pout] (-1.3945,1.9391) -- (-2.4284,0.3868);
\draw[pout] (0.8623,2.1160) -- (-0.7888,2.7000);
\draw[pout] (-0.7888,2.7000) -- (-1.1333,2.1211);
\draw[pout] (-1.1333,2.1211) -- (-0.5344,0.9826);
\draw[pout] (0.3768,-0.8975) -- (-0.3452,-2.0401);
\draw[pout] (-0.3452,-2.0401) -- (1.2024,-1.7912);
\draw[pout] (-0.3414,-0.8081) -- (0.2868,-2.7000);
\draw[pout] (0.2868,-2.7000) -- (0.7381,-2.1999);
\draw[pout] (0.7381,-2.1999) -- (0.3768,-0.8975);
\draw[pmkone] (1.2024,-1.7912) -- (0.3768,-0.8975);
\draw[pmkone] (0.3768,-0.8975) -- (-0.8359,0.1416);
\draw[pmkone] (-0.8359,0.1416) -- (-2.4284,0.3868);
\draw[pmktwo] (1.5079,0.1416) -- (0.3768,-0.8975);
\draw[pmktwo] (0.3768,-0.8975) -- (-0.3414,-0.8081);
\draw[pmktwo] (0.8623,2.1160) -- (1.5079,0.1416);
\node[pnum] at (0.0458,0.4220) {0};
\node[pnum] at (0.3496,-0.2048) {1};
\node[pnum] at (-1.2662,0.5037) {3};
\node[pnum] at (-0.2668,-0.5213) {4};
\node[pnum] at (-1.8214,-0.2119) {9};
\node[pnum] at (-1.0628,-0.6676) {10};
\node[pnum] at (0.6119,1.0801) {2};
\node[pnum] at (1.0290,-0.8490) {5};
\node[pnum] at (1.7129,-0.6784) {13};
\node[pnum] at (1.5731,1.0152) {6};
\node[pnum] at (-1.4524,1.1029) {8};
\node[pnum] at (-0.1536,1.9329) {7};
\node[pnum] at (-0.8188,1.9346) {14};
\node[pnum] at (0.4113,-1.5763) {12};
\node[pnum] at (0.1074,-1.4685) {11};
\node[pnum] at (0.4672,-1.9325) {15};
\node[plab] at (-0.8359,0.1416) {3};
\node[plab] at (1.5079,0.1416) {5};
\node[plab] at (-0.5344,0.9826) {6};
\node[plab] at (0.3768,-0.8975) {2};
\node[plab] at (-2.4284,0.3868) {4};
\node[plab] at (1.2024,-1.7912) {4};
\node[plab] at (-0.7888,2.7000) {1};
\node[plab] at (-1.3945,1.9391) {7};
\node[plab] at (-2.2000,-1.1640) {1};
\node[plab] at (-2.0112,-1.3363) {1};
\node[plab] at (-0.3452,-2.0401) {1};
\node[plab] at (2.4284,-0.3855) {7};
\node[plab] at (-1.1333,2.1211) {7};
\node[plab] at (0.7381,-2.1999) {1};
\fill[black] (0.8623,2.1160) circle (1.5pt);
\node[pcal, right, xshift=2pt] at (0.8623,2.1160) {$u_{570}(0)$};
\fill[black] (2.3490,0.7879) circle (1.5pt);
\node[pcal, above right] at (2.3490,0.7879) {$u_{570}(7)$};
\fill[black] (-0.3414,-0.8081) circle (1.5pt);
\node[pcal, below left, xshift=3pt, yshift=-6pt] at (-0.3414,-0.8081) {$u_{207}(0)$};
\fill[black] (0.2868,-2.7000) circle (1.5pt);
\node[pcal, below right] at (0.2868,-2.7000) {$u_{207}(7)$};

%% file: figs/corner-data.tex
\providecolor{figblue}{RGB}{40,90,160}
\providecolor{figred}{RGB}{170,60,50}
\tikzset{
  vlab/.style={font=\scriptsize, inner sep=1pt, fill=white, fill opacity=0.75,
               text opacity=1},
  chair/.style={gray!70, line width=0.3pt},
  cbdry/.style={black, line width=0.9pt, line join=round},
  elab/.style={font=\scriptsize, inner sep=1pt},
}
\coordinate (A) at (0,0);
\coordinate (B) at (4,0);
\coordinate (C) at (1.2,2.6);
\coordinate (F) at (1.2,0);

\fill[figblue, fill opacity=0.14] (A)--(B)--(C)--cycle;
\draw[cbdry] (A)--(B)--(C)--cycle;
\draw[chair] (C)--(F);
\draw[chair] (1.0,0) -- (1.0,0.20) -- (1.2,0.20);

\draw[{Stealth[length=1.6mm]}-{Stealth[length=1.6mm]}, figblue, line width=0.6pt]
      (0,-0.55) -- (1.2,-0.55)
      node[midway, below, vlab, text=figblue] {$p/\sqrt\alpha$};
\draw[chair] (0,-0.08)--(0,-0.65);
\draw[chair] (1.2,-0.08)--(1.2,-0.65);
\draw[{Stealth[length=1.6mm]}-{Stealth[length=1.6mm]}, figred, line width=0.6pt]
      (-0.55,0) -- (-0.55,2.6)
      node[midway, left, vlab, text=figred] {$S/\sqrt\alpha$};
\draw[chair] (-0.08,2.6)--(-0.65,2.6);
\draw[chair] (-0.08,0)--(-0.65,0);

\fill (A) circle (1.1pt);
\fill (B) circle (1.1pt);
\fill (C) circle (1.1pt);
\node[vlab, below left]  at (A) {$A$};
\node[vlab, below right] at (B) {$B$};
\node[vlab, above]       at (C) {$C$};

\node[elab, below] at (2.0,0.02) {$\sqrt\alpha$};
\node[elab] at (0.40,1.5) {$\sqrt\beta$};
\node[elab] at (2.90,1.5) {$\sqrt\gamma$};
\node[elab] at (2.05,0.80) {$S = 2\,\mathrm{Area}$};

\pic[draw, figred, line width=0.9pt, angle radius=7mm,
     "$\theta_A$"{text=figred, font=\scriptsize, inner sep=1.5pt},
     angle eccentricity=1.5] {angle=B--A--C};

%% file: figs/fold-relations-square-allblue.tex
\providecolor{figblue}{RGB}{40,90,160}
\providecolor{figred}{RGB}{170,60,50}
\providecolor{figgreen}{RGB}{56,124,92}
\providecolor{figpurple}{RGB}{132,96,158}
\tikzset{
  rflat/.style={figgreen, opacity=0.6, line width=1.2pt, line cap=round,
                dash pattern=on 3.2pt off 2.6pt},
  rcirc/.style={figpurple, opacity=0.6, line width=1.2pt, line cap=round,
                dash pattern=on 3.2pt off 2.6pt},
}
\fill[figblue, fill opacity=0.14] (-2.5423,0.0591) -- (2.8870,0.8158) -- (-1.5549,2.5656) -- cycle;
\fill[figblue, fill opacity=0.14] (-2.5423,0.0591) -- (0.0000,-2.5656) -- (2.8870,0.8158) -- cycle;
\fill[figblue, fill opacity=0.14] (-2.5423,0.0591) -- (-1.5549,2.5656) -- (1.5549,-2.5656) -- cycle;
\fill[figblue, fill opacity=0.14] (-2.5423,0.0591) -- (-1.5549,-2.5656) -- (0.0000,-2.5656) -- cycle;
\fill[figblue, fill opacity=0.14] (-2.5423,0.0591) -- (1.5549,-2.5656) -- (2.2287,0.8552) -- cycle;
\fill[figblue, fill opacity=0.14] (-2.5423,0.0591) -- (2.2287,0.8552) -- (-1.5549,-2.5656) -- cycle;
\fill[figblue, fill opacity=0.14] (2.8870,0.8158) -- (-1.5549,-2.5656) -- (-1.5549,2.5656) -- cycle;
\fill[figblue, fill opacity=0.14] (2.8870,0.8158) -- (0.0000,-2.5656) -- (1.5549,-2.5656) -- cycle;
\fill[figblue, fill opacity=0.14] (2.8870,0.8158) -- (1.5549,-2.5656) -- (1.5549,2.5656) -- cycle;
\fill[figblue, fill opacity=0.14] (2.8870,0.8158) -- (1.5549,2.5656) -- (-1.5549,-2.5656) -- cycle;
\fill[figblue, fill opacity=0.14] (-1.5549,2.5656) -- (1.5549,2.5656) -- (1.5549,-2.5656) -- cycle;
\fill[figblue, fill opacity=0.14] (-1.5549,2.5656) -- (-1.5549,-2.5656) -- (2.2287,0.8552) -- cycle;
\fill[figblue, fill opacity=0.14] (-1.5549,2.5656) -- (2.2287,0.8552) -- (1.5549,2.5656) -- cycle;
\fill[figblue, fill opacity=0.14] (0.0000,-2.5656) -- (2.2287,0.8552) -- (1.5549,-2.5656) -- cycle;
\fill[figblue, fill opacity=0.14] (0.0000,-2.5656) -- (-1.5549,-2.5656) -- (1.5549,2.5656) -- cycle;
\fill[figblue, fill opacity=0.14] (0.0000,-2.5656) -- (1.5549,2.5656) -- (2.2287,0.8552) -- cycle;
\draw[gray!70, line width=0.3pt] (-2.5423,0.0591) -- (2.8870,0.8158) -- (-1.5549,2.5656) -- cycle;
\draw[gray!70, line width=0.3pt] (-2.5423,0.0591) -- (0.0000,-2.5656) -- (2.8870,0.8158) -- cycle;
\draw[gray!70, line width=0.3pt] (-2.5423,0.0591) -- (-1.5549,2.5656) -- (1.5549,-2.5656) -- cycle;
\draw[gray!70, line width=0.3pt] (-2.5423,0.0591) -- (-1.5549,-2.5656) -- (0.0000,-2.5656) -- cycle;
\draw[gray!70, line width=0.3pt] (-2.5423,0.0591) -- (1.5549,-2.5656) -- (2.2287,0.8552) -- cycle;
\draw[gray!70, line width=0.3pt] (-2.5423,0.0591) -- (2.2287,0.8552) -- (-1.5549,-2.5656) -- cycle;
\draw[gray!70, line width=0.3pt] (2.8870,0.8158) -- (-1.5549,-2.5656) -- (-1.5549,2.5656) -- cycle;
\draw[gray!70, line width=0.3pt] (2.8870,0.8158) -- (0.0000,-2.5656) -- (1.5549,-2.5656) -- cycle;
\draw[gray!70, line width=0.3pt] (2.8870,0.8158) -- (1.5549,-2.5656) -- (1.5549,2.5656) -- cycle;
\draw[gray!70, line width=0.3pt] (2.8870,0.8158) -- (1.5549,2.5656) -- (-1.5549,-2.5656) -- cycle;
\draw[gray!70, line width=0.3pt] (-1.5549,2.5656) -- (1.5549,2.5656) -- (1.5549,-2.5656) -- cycle;
\draw[gray!70, line width=0.3pt] (-1.5549,2.5656) -- (-1.5549,-2.5656) -- (2.2287,0.8552) -- cycle;
\draw[gray!70, line width=0.3pt] (-1.5549,2.5656) -- (2.2287,0.8552) -- (1.5549,2.5656) -- cycle;
\draw[gray!70, line width=0.3pt] (0.0000,-2.5656) -- (2.2287,0.8552) -- (1.5549,-2.5656) -- cycle;
\draw[gray!70, line width=0.3pt] (0.0000,-2.5656) -- (-1.5549,-2.5656) -- (1.5549,2.5656) -- cycle;
\draw[gray!70, line width=0.3pt] (0.0000,-2.5656) -- (1.5549,2.5656) -- (2.2287,0.8552) -- cycle;
\draw[rflat] (-1.9281,-2.5656) -- (1.9281,-2.5656);
\draw[rcirc] (0.0000,0.0000) circle (3.0000);
\draw[rflat] (-1.5549,-2.5656) -- (1.5549,-2.5656) -- (1.5549,2.5656) -- (-1.5549,2.5656) -- cycle;
\fill (-1.5549,-2.5656) circle (1.1pt);
\node[font=\scriptsize] at (-1.7004,-2.7564) {$0$};
\fill (2.2287,0.8552) circle (1.1pt);
\node[font=\scriptsize] at (2.4431,0.9631) {$1$};
\fill (0.0000,-2.5656) circle (1.1pt);
\node[font=\scriptsize] at (-0.0311,-2.8036) {$2$};
\fill (-2.5423,0.0591) circle (1.1pt);
\node[font=\scriptsize] at (-2.7819,0.0728) {$3$};
\fill (1.5549,-2.5656) circle (1.1pt);
\node[font=\scriptsize] at (1.6624,-2.7802) {$4$};
\fill (2.8870,0.8158) circle (1.1pt);
\node[font=\scriptsize] at (3.1129,0.8968) {$5$};
\fill (-1.5549,2.5656) circle (1.1pt);
\node[font=\scriptsize] at (-1.6929,2.7619) {$6$};
\fill (1.5549,2.5656) circle (1.1pt);
\node[font=\scriptsize] at (1.6555,2.7835) {$7$};

%% file: figs/fold-relations-hex-allblue.tex
\providecolor{figblue}{RGB}{40,90,160}
\providecolor{figred}{RGB}{170,60,50}
\providecolor{figgreen}{RGB}{56,124,92}
\providecolor{figpurple}{RGB}{132,96,158}
\tikzset{
  rflat/.style={figgreen, opacity=0.6, line width=1.2pt, line cap=round,
                dash pattern=on 3.2pt off 2.6pt},
  rcirc/.style={figpurple, opacity=0.6, line width=1.2pt, line cap=round,
                dash pattern=on 3.2pt off 2.6pt},
}
\fill[figblue, fill opacity=0.14] (-2.7105,1.2732) -- (-1.3628,2.5294) -- (-1.6373,-2.3757) -- cycle;
\fill[figblue, fill opacity=0.14] (-2.7105,1.2732) -- (-0.8974,-2.3484) -- (-1.3628,2.5294) -- cycle;
\fill[figblue, fill opacity=0.14] (-2.7105,1.2732) -- (-1.6373,-2.3757) -- (2.3260,1.2732) -- cycle;
\fill[figblue, fill opacity=0.14] (-2.7105,1.2732) -- (2.3260,1.2732) -- (-0.8974,-2.3484) -- cycle;
\fill[figblue, fill opacity=0.14] (-1.3628,2.5294) -- (-0.0084,-2.7701) -- (-1.6373,-2.3757) -- cycle;
\fill[figblue, fill opacity=0.14] (-1.3628,2.5294) -- (-0.8974,-2.3484) -- (-2.5267,1.5917) -- cycle;
\fill[figblue, fill opacity=0.14] (-1.3628,2.5294) -- (2.3260,1.2732) -- (-0.0084,-2.7701) -- cycle;
\fill[figblue, fill opacity=0.14] (-1.3628,2.5294) -- (0.7490,2.0180) -- (2.3260,1.2732) -- cycle;
\fill[figblue, fill opacity=0.14] (-1.3628,2.5294) -- (-2.5267,1.5917) -- (0.7490,2.0180) -- cycle;
\fill[figblue, fill opacity=0.14] (-1.6373,-2.3757) -- (-2.5267,1.5917) -- (2.3260,1.2732) -- cycle;
\fill[figblue, fill opacity=0.14] (-1.6373,-2.3757) -- (-0.0084,-2.7701) -- (0.7490,2.0180) -- cycle;
\fill[figblue, fill opacity=0.14] (-1.6373,-2.3757) -- (0.7490,2.0180) -- (-2.5267,1.5917) -- cycle;
\fill[figblue, fill opacity=0.14] (-0.8974,-2.3484) -- (2.3260,1.2732) -- (0.7490,2.0180) -- cycle;
\fill[figblue, fill opacity=0.14] (-0.8974,-2.3484) -- (-0.0084,-2.7701) -- (-2.5267,1.5917) -- cycle;
\fill[figblue, fill opacity=0.14] (-0.8974,-2.3484) -- (0.7490,2.0180) -- (-0.0084,-2.7701) -- cycle;
\fill[figblue, fill opacity=0.14] (2.3260,1.2732) -- (-2.5267,1.5917) -- (-0.0084,-2.7701) -- cycle;
\draw[gray!70, line width=0.3pt] (-2.7105,1.2732) -- (-1.3628,2.5294) -- (-1.6373,-2.3757) -- cycle;
\draw[gray!70, line width=0.3pt] (-2.7105,1.2732) -- (-0.8974,-2.3484) -- (-1.3628,2.5294) -- cycle;
\draw[gray!70, line width=0.3pt] (-2.7105,1.2732) -- (-1.6373,-2.3757) -- (2.3260,1.2732) -- cycle;
\draw[gray!70, line width=0.3pt] (-2.7105,1.2732) -- (2.3260,1.2732) -- (-0.8974,-2.3484) -- cycle;
\draw[gray!70, line width=0.3pt] (-1.3628,2.5294) -- (-0.0084,-2.7701) -- (-1.6373,-2.3757) -- cycle;
\draw[gray!70, line width=0.3pt] (-1.3628,2.5294) -- (-0.8974,-2.3484) -- (-2.5267,1.5917) -- cycle;
\draw[gray!70, line width=0.3pt] (-1.3628,2.5294) -- (2.3260,1.2732) -- (-0.0084,-2.7701) -- cycle;
\draw[gray!70, line width=0.3pt] (-1.3628,2.5294) -- (0.7490,2.0180) -- (2.3260,1.2732) -- cycle;
\draw[gray!70, line width=0.3pt] (-1.3628,2.5294) -- (-2.5267,1.5917) -- (0.7490,2.0180) -- cycle;
\draw[gray!70, line width=0.3pt] (-1.6373,-2.3757) -- (-2.5267,1.5917) -- (2.3260,1.2732) -- cycle;
\draw[gray!70, line width=0.3pt] (-1.6373,-2.3757) -- (-0.0084,-2.7701) -- (0.7490,2.0180) -- cycle;
\draw[gray!70, line width=0.3pt] (-1.6373,-2.3757) -- (0.7490,2.0180) -- (-2.5267,1.5917) -- cycle;
\draw[gray!70, line width=0.3pt] (-0.8974,-2.3484) -- (2.3260,1.2732) -- (0.7490,2.0180) -- cycle;
\draw[gray!70, line width=0.3pt] (-0.8974,-2.3484) -- (-0.0084,-2.7701) -- (-2.5267,1.5917) -- cycle;
\draw[gray!70, line width=0.3pt] (-0.8974,-2.3484) -- (0.7490,2.0180) -- (-0.0084,-2.7701) -- cycle;
\draw[gray!70, line width=0.3pt] (2.3260,1.2732) -- (-2.5267,1.5917) -- (-0.0084,-2.7701) -- cycle;
\draw[rcirc] (-0.1923,0.0316) circle (2.8077);
\draw[rcirc] (-0.4128,0.1589) circle (2.5537);
\draw[rcirc] (0.5679,-0.4073) circle (2.4321);
\fill (-2.7105,1.2732) circle (1.1pt);
\node[font=\scriptsize] at (-2.9185,1.3930) {$0$};
\fill (-1.3628,2.5294) circle (1.1pt);
\node[font=\scriptsize] at (-1.4219,2.7620) {$1$};
\fill (-1.6373,-2.3757) circle (1.1pt);
\node[font=\scriptsize] at (-1.7162,-2.6023) {$2$};
\fill (-0.8974,-2.3484) circle (1.1pt);
\node[font=\scriptsize] at (-0.9107,-2.5880) {$3$};
\fill (2.3260,1.2732) circle (1.1pt);
\node[font=\scriptsize] at (2.5515,1.3554) {$4$};
\fill (-0.0084,-2.7701) circle (1.1pt);
\node[font=\scriptsize] at (0.0513,-3.0025) {$5$};
\fill (-2.5267,1.5917) circle (1.1pt);
\node[font=\scriptsize] at (-2.7126,1.7434) {$6$};
\fill (0.7490,2.0180) circle (1.1pt);
\node[font=\scriptsize] at (0.8997,2.2049) {$7$};

%% file: figs/net-square.tex
\providecolor{figblue}{RGB}{40,90,160}
\providecolor{figred}{RGB}{170,60,50}
\tikzset{
  ncell/.style={gray!70, line width=0.3pt},
  nout/.style={black, line width=0.9pt, line join=round},
  ncut/.style={black, line width=0.9pt, dash pattern=on 2.4pt off 1.8pt},
  nlab/.style={font=\scriptsize, inner sep=0.6pt, fill=white,
               fill opacity=0.72, text opacity=1},
  nnum/.style={font=\tiny, text=black!55, inner sep=0.5pt},
  nmkone/.style={figblue, line width=1.2pt},
  nmktwo/.style={figred, line width=1.2pt},
}
\fill[figblue, fill opacity=0.09] (-0.5655,0.3485) -- (1.4607,0.3485) -- (-0.0762,1.2157) -- cycle;
\draw[ncell] (-0.5655,0.3485) -- (1.4607,0.3485) -- (-0.0762,1.2157) -- cycle;
\fill[figblue, fill opacity=0.09] (-0.5655,0.3485) -- (0.2313,-0.7422) -- (1.4607,0.3485) -- cycle;
\draw[ncell] (-0.5655,0.3485) -- (0.2313,-0.7422) -- (1.4607,0.3485) -- cycle;
\fill[figblue, fill opacity=0.09] (-0.5655,0.3485) -- (-0.0762,1.2157) -- (-2.2732,0.9130) -- cycle;
\draw[ncell] (-0.5655,0.3485) -- (-0.0762,1.2157) -- (-2.2732,0.9130) -- cycle;
\fill[figblue, fill opacity=0.09] (-0.5655,0.3485) -- (-0.3379,-0.6628) -- (0.2313,-0.7422) -- cycle;
\draw[ncell] (-0.5655,0.3485) -- (-0.3379,-0.6628) -- (0.2313,-0.7422) -- cycle;
\fill[figblue, fill opacity=0.09] (-0.5655,0.3485) -- (-2.2732,0.9130) -- (-2.2011,-0.3738) -- cycle;
\draw[ncell] (-0.5655,0.3485) -- (-2.2732,0.9130) -- (-2.2011,-0.3738) -- cycle;
\fill[figblue, fill opacity=0.09] (-0.5655,0.3485) -- (-2.2011,-0.3738) -- (-0.3379,-0.6628) -- cycle;
\draw[ncell] (-0.5655,0.3485) -- (-2.2011,-0.3738) -- (-0.3379,-0.6628) -- cycle;
\fill[figblue, fill opacity=0.09] (1.4607,0.3485) -- (1.3965,2.4109) -- (-0.0762,1.2157) -- cycle;
\draw[ncell] (1.4607,0.3485) -- (1.3965,2.4109) -- (-0.0762,1.2157) -- cycle;
\fill[figblue, fill opacity=0.09] (1.4607,0.3485) -- (0.2313,-0.7422) -- (0.8005,-0.8215) -- cycle;
\draw[ncell] (1.4607,0.3485) -- (0.2313,-0.7422) -- (0.8005,-0.8215) -- cycle;
\fill[figblue, fill opacity=0.09] (1.4607,0.3485) -- (0.8005,-0.8215) -- (2.2732,0.3738) -- cycle;
\draw[ncell] (1.4607,0.3485) -- (0.8005,-0.8215) -- (2.2732,0.3738) -- cycle;
\fill[figblue, fill opacity=0.09] (1.4607,0.3485) -- (2.2732,0.3738) -- (1.3965,2.4109) -- cycle;
\draw[ncell] (1.4607,0.3485) -- (2.2732,0.3738) -- (1.3965,2.4109) -- cycle;
\fill[figblue, fill opacity=0.09] (-0.0762,1.2157) -- (-0.8005,2.1082) -- (-2.2732,0.9130) -- cycle;
\draw[ncell] (-0.0762,1.2157) -- (-0.8005,2.1082) -- (-2.2732,0.9130) -- cycle;
\fill[figblue, fill opacity=0.09] (-0.0762,1.2157) -- (1.3965,2.4109) -- (-0.4666,2.7000) -- cycle;
\draw[ncell] (-0.0762,1.2157) -- (1.3965,2.4109) -- (-0.4666,2.7000) -- cycle;
\fill[figblue, fill opacity=0.09] (-0.0762,1.2157) -- (-0.4666,2.7000) -- (-0.8005,2.1082) -- cycle;
\draw[ncell] (-0.0762,1.2157) -- (-0.4666,2.7000) -- (-0.8005,2.1082) -- cycle;
\fill[figblue, fill opacity=0.09] (0.2313,-0.7422) -- (0.8727,-2.1082) -- (0.8005,-0.8215) -- cycle;
\draw[ncell] (0.2313,-0.7422) -- (0.8727,-2.1082) -- (0.8005,-0.8215) -- cycle;
\fill[figblue, fill opacity=0.09] (0.2313,-0.7422) -- (-0.3379,-0.6628) -- (0.5388,-2.7000) -- cycle;
\draw[ncell] (0.2313,-0.7422) -- (-0.3379,-0.6628) -- (0.5388,-2.7000) -- cycle;
\fill[figblue, fill opacity=0.09] (0.2313,-0.7422) -- (0.5388,-2.7000) -- (0.8727,-2.1082) -- cycle;
\draw[ncell] (0.2313,-0.7422) -- (0.5388,-2.7000) -- (0.8727,-2.1082) -- cycle;
\draw[nout] (-2.2732,0.9130) -- (-2.2011,-0.3738) -- (-0.3379,-0.6628) -- (0.5388,-2.7000) -- (0.8727,-2.1082) -- (0.8005,-0.8215) -- (2.2732,0.3738) -- (1.3965,2.4109) -- (-0.4666,2.7000) -- (-0.8005,2.1082) -- cycle;
\draw[ncut] (1.4607,0.3485) -- (2.2732,0.3738);
\draw[ncut] (-0.0762,1.2157) -- (-0.8005,2.1082);
\draw[ncut] (-2.2011,-0.3738) -- (-0.5655,0.3485);
\draw[ncut] (0.2313,-0.7422) -- (0.8727,-2.1082);
\draw[nmkone] (0.8005,-0.8215) -- (0.2313,-0.7422);
\fill[figblue] (0.4169,-0.7680) -- (0.5998,-0.9046) -- (0.6301,-0.6867) -- cycle;
\draw[nmkone] (0.2313,-0.7422) -- (-0.5655,0.3485);
\fill[figblue] (-0.2261,-0.1161) -- (-0.1969,-0.3425) -- (-0.0193,-0.2127) -- cycle;
\draw[nmkone] (-0.5655,0.3485) -- (-2.2732,0.9130);
\fill[figblue] (-1.5143,0.6621) -- (-1.3589,0.4949) -- (-1.2899,0.7038) -- cycle;
\draw[nmktwo] (1.4607,0.3485) -- (0.2313,-0.7422);
\fill[figred] (0.7712,-0.2632) -- (0.9938,-0.2128) -- (0.8478,-0.0482) -- cycle;
\draw[nmktwo] (0.2313,-0.7422) -- (-0.3379,-0.6628);
\fill[figred] (-0.1524,-0.6887) -- (0.0305,-0.8252) -- (0.0609,-0.6073) -- cycle;
\draw[nmktwo] (1.3965,2.4109) -- (1.4607,0.3485);
\fill[figred] (1.4317,1.2797) -- (1.5354,1.4831) -- (1.3156,1.4762) -- cycle;
\node[nnum] at (0.2730,0.6375) {0};
\node[nnum] at (0.3755,-0.0151) {1};
\node[nnum] at (-0.9716,0.8257) {3};
\node[nnum] at (-0.2098,-0.3133) {4};
\node[nnum] at (-1.6799,0.2959) {9};
\node[nnum] at (-1.0348,-0.2294) {10};
\node[nnum] at (0.9270,1.3250) {2};
\node[nnum] at (0.8309,-0.4051) {5};
\node[nnum] at (1.5115,-0.0331) {13};
\node[nnum] at (1.7101,1.0444) {6};
\node[nnum] at (-1.0500,1.4123) {8};
\node[nnum] at (0.2846,2.1089) {7};
\node[nnum] at (-0.4037,1.9954) {14};
\node[nnum] at (0.6348,-1.2240) {12};
\node[nnum] at (0.1203,-1.3650) {11};
\node[nnum] at (0.5348,-1.7685) {15};
\node[nlab] at (-0.5655,0.3485) {$3$};
\node[nlab] at (1.4607,0.3485) {$5$};
\node[nlab] at (-0.0762,1.2157) {$6$};
\node[nlab] at (0.2313,-0.7422) {$2$};
\node[nlab] at (-2.2732,0.9130) {$4$};
\node[nlab] at (-0.3379,-0.6628) {$0$};
\node[nlab] at (-2.2011,-0.3738) {$1$};
\node[nlab] at (1.3965,2.4109) {$0$};
\node[nlab] at (0.8005,-0.8215) {$4$};
\node[nlab] at (2.2732,0.3738) {$7$};
\node[nlab] at (-0.8005,2.1082) {$7$};
\node[nlab] at (-0.4666,2.7000) {$1$};
\node[nlab] at (0.8727,-2.1082) {$1$};
\node[nlab] at (0.5388,-2.7000) {$7$};

%% file: figs/net-hex.tex
\providecolor{figblue}{RGB}{40,90,160}
\providecolor{figred}{RGB}{170,60,50}
\tikzset{
  ncell/.style={gray!70, line width=0.3pt},
  nout/.style={black, line width=0.9pt, line join=round},
  ncut/.style={black, line width=0.9pt, dash pattern=on 2.4pt off 1.8pt},
  nlab/.style={font=\scriptsize, inner sep=0.6pt, fill=white,
               fill opacity=0.72, text opacity=1},
  nnum/.style={font=\tiny, text=black!55, inner sep=0.5pt},
  nlead/.style={black!40, line width=0.5pt},
  nmkone/.style={figblue, line width=1.2pt},
  nmktwo/.style={figred, line width=1.2pt},
}
\fill[figblue, fill opacity=0.09] (0.4236,-0.6992) -- (1.1842,-0.6992) -- (-0.2793,0.7048) -- cycle;
\draw[ncell] (0.4236,-0.6992) -- (1.1842,-0.6992) -- (-0.2793,0.7048) -- cycle;
\fill[figblue, fill opacity=0.09] (0.4236,-0.6992) -- (-0.0481,-2.3033) -- (1.1842,-0.6992) -- cycle;
\draw[ncell] (0.4236,-0.6992) -- (-0.0481,-2.3033) -- (1.1842,-0.6992) -- cycle;
\fill[figblue, fill opacity=0.09] (0.4236,-0.6992) -- (-0.2793,0.7048) -- (-1.6227,-1.0676) -- cycle;
\draw[ncell] (0.4236,-0.6992) -- (-0.2793,0.7048) -- (-1.6227,-1.0676) -- cycle;
\fill[figblue, fill opacity=0.09] (0.4236,-0.6992) -- (-1.6227,-1.0676) -- (-0.0481,-2.3033) -- cycle;
\draw[ncell] (0.4236,-0.6992) -- (-1.6227,-1.0676) -- (-0.0481,-2.3033) -- cycle;
\fill[figblue, fill opacity=0.09] (1.1842,-0.6992) -- (0.1016,1.2824) -- (-0.2793,0.7048) -- cycle;
\draw[ncell] (1.1842,-0.6992) -- (0.1016,1.2824) -- (-0.2793,0.7048) -- cycle;
\fill[figblue, fill opacity=0.09] (1.1842,-0.6992) -- (-0.0481,-2.3033) -- (1.3857,-1.2824) -- cycle;
\draw[ncell] (1.1842,-0.6992) -- (-0.0481,-2.3033) -- (1.3857,-1.2824) -- cycle;
\fill[figblue, fill opacity=0.09] (1.1842,-0.6992) -- (1.9447,0.7184) -- (0.1016,1.2824) -- cycle;
\draw[ncell] (1.1842,-0.6992) -- (1.9447,0.7184) -- (0.1016,1.2824) -- cycle;
\fill[figblue, fill opacity=0.09] (1.1842,-0.6992) -- (1.6781,0.0496) -- (1.9447,0.7184) -- cycle;
\draw[ncell] (1.1842,-0.6992) -- (1.6781,0.0496) -- (1.9447,0.7184) -- cycle;
\fill[figblue, fill opacity=0.09] (1.1842,-0.6992) -- (1.3857,-1.2824) -- (1.6781,0.0496) -- cycle;
\draw[ncell] (1.1842,-0.6992) -- (1.3857,-1.2824) -- (1.6781,0.0496) -- cycle;
\fill[figblue, fill opacity=0.09] (-0.2793,0.7048) -- (-1.9447,0.9140) -- (-1.6227,-1.0676) -- cycle;
\draw[ncell] (-0.2793,0.7048) -- (-1.9447,0.9140) -- (-1.6227,-1.0676) -- cycle;
\fill[figblue, fill opacity=0.09] (-0.2793,0.7048) -- (0.1016,1.2824) -- (-1.6523,2.2460) -- cycle;
\draw[ncell] (-0.2793,0.7048) -- (0.1016,1.2824) -- (-1.6523,2.2460) -- cycle;
\fill[figblue, fill opacity=0.09] (-0.2793,0.7048) -- (-1.6523,2.2460) -- (-1.9447,0.9140) -- cycle;
\draw[ncell] (-0.2793,0.7048) -- (-1.6523,2.2460) -- (-1.9447,0.9140) -- cycle;
\fill[figblue, fill opacity=0.09] (-0.0481,-2.3033) -- (-1.6227,-1.0676) -- (-1.8893,-1.7364) -- cycle;
\draw[ncell] (-0.0481,-2.3033) -- (-1.6227,-1.0676) -- (-1.8893,-1.7364) -- cycle;
\fill[figblue, fill opacity=0.09] (-0.0481,-2.3033) -- (-0.1353,-2.7000) -- (1.3857,-1.2824) -- cycle;
\draw[ncell] (-0.0481,-2.3033) -- (-0.1353,-2.7000) -- (1.3857,-1.2824) -- cycle;
\fill[figblue, fill opacity=0.09] (0.1888,1.6791) -- (-1.6523,2.2460) -- (0.1016,1.2824) -- cycle;
\draw[ncell] (0.1888,1.6791) -- (-1.6523,2.2460) -- (0.1016,1.2824) -- cycle;
\fill[figblue, fill opacity=0.09] (1.9447,0.7184) -- (1.6227,2.7000) -- (0.1016,1.2824) -- cycle;
\draw[ncell] (1.9447,0.7184) -- (1.6227,2.7000) -- (0.1016,1.2824) -- cycle;
\draw[nout] (1.6781,0.0496) -- (1.9447,0.7184) -- (1.6227,2.7000) -- (0.1016,1.2824) -- (0.1888,1.6791) -- (-1.6523,2.2460) -- (-1.9447,0.9140) -- (-1.6227,-1.0676) -- (-1.8893,-1.7364) -- (-0.0481,-2.3033) -- (-0.1353,-2.7000) -- (1.3857,-1.2824) -- cycle;
\draw[ncut] (-1.6227,-1.0676) -- (0.4236,-0.6992);
\draw[ncut] (-0.2793,0.7048) -- (-1.9447,0.9140);
\draw[ncut] (1.6781,0.0496) -- (1.1842,-0.6992);
\draw[nmkone] (1.1842,-0.6992) -- (1.9447,0.7184);
\fill[figblue] (1.5052,-0.1008) -- (1.3138,-0.2250) -- (1.5076,-0.3290) -- cycle;
\draw[nmkone] (-1.6227,-1.0676) -- (0.4236,-0.6992);
\fill[figblue] (-0.5011,-0.8657) -- (-0.7174,-0.7929) -- (-0.6784,-1.0094) -- cycle;
\draw[nmkone] (0.4236,-0.6992) -- (1.1842,-0.6992);
\fill[figblue] (0.9039,-0.6992) -- (0.7039,-0.5892) -- (0.7039,-0.8092) -- cycle;
\draw[nmktwo] (1.1842,-0.6992) -- (0.1016,1.2824);
\fill[figred] (0.5950,0.3793) -- (0.5943,0.1511) -- (0.7874,0.2566) -- cycle;
\draw[nmktwo] (-0.1353,-2.7000) -- (1.3857,-1.2824);
\fill[figred] (0.6984,-1.9230) -- (0.4770,-1.9789) -- (0.6270,-2.1398) -- cycle;
\draw[nmktwo] (1.3857,-1.2824) -- (1.1842,-0.6992);
\fill[figred] (1.2523,-0.8963) -- (1.2136,-1.1213) -- (1.4216,-1.0494) -- cycle;
\node[nnum] at (0.4429,-0.2312) {0};
\node[nnum] at (0.5199,-1.2339) {1};
\node[nnum] at (-0.4928,-0.3540) {3};
\node[nnum] at (-0.4157,-1.3567) {4};
\node[nnum] at (0.3355,0.4293) {2};
\node[nnum] at (0.8406,-1.4283) {5};
\node[nnum] at (1.0768,0.4339) {6};
\node[nnum] (ncallout13) at (1.9241,-0.2772) {12};
\draw[nlead] (ncallout13) -- (1.6023,0.0229);
\node[nnum] (ncallout12) at (1.8001,-0.6709) {11};
\draw[nlead] (ncallout12) -- (1.4160,-0.6440);
\node[nnum] at (-1.2822,0.1837) {8};
\node[nnum] at (-0.6100,1.4111) {7};
\node[nnum] at (-1.2921,1.2883) {15};
\node[nnum] at (-1.1867,-1.7024) {9};
\node[nnum] at (0.4008,-2.0952) {10};
\node[nnum] at (-0.4540,1.7358) {14};
\node[nnum] at (1.2230,1.5669) {13};
\node[nlab] at (0.4236,-0.6992) {$0$};
\node[nlab] at (1.1842,-0.6992) {$1$};
\node[nlab] at (-0.2793,0.7048) {$2$};
\node[nlab] at (-0.0481,-2.3033) {$3$};
\node[nlab] at (-1.6227,-1.0676) {$4$};
\node[nlab] at (0.1016,1.2824) {$5$};
\node[nlab] at (1.3857,-1.2824) {$6$};
\node[nlab] at (1.9447,0.7184) {$4$};
\node[nlab] at (1.6781,0.0496) {$7$};
\node[nlab] at (-1.9447,0.9140) {$6$};
\node[nlab] at (-1.6523,2.2460) {$7$};
\node[nlab] at (-1.8893,-1.7364) {$7$};
\node[nlab] at (-0.1353,-2.7000) {$5$};
\node[nlab] at (0.1888,1.6791) {$3$};
\node[nlab] at (1.6227,2.7000) {$6$};

%% file: figs/sep-corners.tex
\providecolor{figblue}{RGB}{40,90,160}
\providecolor{figred}{RGB}{170,60,50}
\fill[black!15] (0.836578,1.687914) -- (2.040626,0.916580) -- (3.936164,2.359556) -- (4.042281,2.898307) -- (0.836578,2.363426) -- cycle;
\draw[figblue, line width=0.8pt] (0.000000,2.223840) -- (3.471403,0.000000) -- (4.042281,2.898307) -- cycle;
\draw[figred, line width=0.8pt] (4.600000,2.864901) -- (0.836578,0.000000) -- (0.836578,4.347461) -- cycle;
\fill[black] (4.042281,2.898307) circle (1.0pt);
\fill[black] (2.040626,0.916580) circle (1.0pt);
\fill[black] (0.836578,1.687914) circle (1.0pt);
\fill[black] (3.936164,2.359556) circle (1.0pt);
\fill[black] (0.836578,2.363426) circle (1.0pt);
\draw[black, line width=0.5pt] (0.836578,2.363426) circle (2.6pt);
\node[font=\small, figblue] at (3.627201,-0.275129) {$341$};
\node[font=\small, figred] at (0.665101,4.613101) {$506$};

%% file: figs/sep-profile.tex
\providecolor{figblue}{RGB}{40,90,160}
\providecolor{figred}{RGB}{170,60,50}
\begin{scope}[scale=1.05]
\draw[black!45, line width=0.5pt] (-0.1,0) -- (4.1,0);
\draw[figred!60!black, line width=1.8pt] (0.9,0) -- (3.1,0);
\node[font=\small, below=1pt] at (2.0,0) {$K$};
\fill[black] (0.9,0) circle (1.0pt);
\draw[black, line width=0.5pt] (0.9,0) circle (2.6pt);
\draw[black!35, dotted, line width=0.5pt] (0.9,0.1) -- (0.9,0.58);
\draw[figblue, dashed, line width=0.6pt] (0.25,0.585) -- (0.55,0.675);
\draw[figblue, line width=1.5pt] (0.55,0.675) -- (3.1,1.44);
\draw[figblue, dashed, line width=0.6pt] (3.1,1.44) -- (3.6,1.59);
\node[font=\scriptsize, figblue, right=1pt] at (3.6,1.59) {sheet of $F$};
\draw[figred, dashed, line width=0.6pt] (0.4,0.61) -- (0.9,0.62);
\draw[figred, line width=1.5pt] (0.9,0.62) -- (3.9,0.68);
\node[font=\scriptsize, figred, right=1pt] at (3.9,0.68) {sheet of $G$};
\end{scope}

%% file: figs/table-coords.tex
\begin{tabular}{@{}c r r r@{}}
\toprule
$v$ & $x_v$ & $y_v$ & $z_v$ \\
\midrule
$0$ & $0$ & $0$ & $0$ \\
$1$ & $1$ & $0$ & $0$ \\
$2$ & $0.26278582057559752$ & $0.21418749136689999$ & $0$ \\
$3$ & $0.040648781980067765$ & $-0.18878547675520668$ & $0.034114710864745844$ \\
$4$ & $0.6161696729561239$ & $0.20692899510328877$ & $-0.39870402196652055$ \\
$5$ & $1.1297992036975992$ & $-0.022238686825010845$ & $0.13031507592364372$ \\
$6$ & $0.5703850909491307$ & $-0.70911659242674696$ & $0.14211112502680828$ \\
$7$ & $1.1433556271572907$ & $-0.26117826830428981$ & $0.13749394148627028$ \\
\bottomrule
\end{tabular}

%% file: figs/table-coords-rho.tex
\begin{tabular}{@{}c r r r@{}}
\toprule
$v$ & $x_v$ & $y_v$ & $z_v$ \\
\midrule
$0$ & $-0.0052919576658824819$ & $0.25189213380754671$ & $-0.19636317590069088$ \\
$1$ & $-0.10330037049737666$ & $0.097954858850342955$ & $0.23022989757571455$ \\
$2$ & $0.088435194516004872$ & $-0.042881936616869155$ & $-0.16602959156361255$ \\
$3$ & $0.10713507949378637$ & $-0.03807120945413147$ & $-0.19492235829557528$ \\
$4$ & $-0.15821722275466177$ & $0.26928131058897076$ & $-0.057266978045909966$ \\
$5$ & $0.099801133842719214$ & $0.17571318949877185$ & $0.042053203465396087$ \\
$6$ & $-0.026706140635424078$ & $0.3003490487315576$ & $-0.18681063266426795$ \\
$7$ & $-0.004032608759816151$ & $-0.14181859043688222$ & $-0.02715724967818926$ \\
\bottomrule
\end{tabular}